\documentclass[a4paper,twoside,11pt]{article}

\usepackage{color}
\usepackage{appendix}
\usepackage{subfigure}
\usepackage[latin1]{inputenc}
\usepackage{lmodern}
\usepackage[T1]{fontenc}

\usepackage[a4paper,top=3cm,bottom=3cm,left=3cm,right=3cm,
bindingoffset=5mm]{geometry}

\usepackage{amsmath,amssymb,amsthm}
\usepackage{subfigure}

\usepackage{graphicx}
\usepackage{bmpsize}

 \usepackage{epstopdf}

\newtheorem{thm}{Theorem}[section]
\newtheorem{prop}[thm]{Proposition}
\newtheorem{lem}[thm]{Lemma}
\newtheorem{cor}[thm]{Corollary}

\newtheorem{remark}{Remark}

\renewcommand{\O}{\Omega}

\newcommand\Th{{\mathcal T}_h}

\newcommand\Eh{{\mathcal E}_h}

\newcommand\R{\mathbb{R}}
\newcommand\C{\mathbb{C}}
\newcommand\D{\mathbb{D}}

\renewcommand{\P}{{\mathcal P}}  

\def\decapita#1{}
\def\grecomultibold#1#2{\grecobolddef#1\def\secondobold{#2}%
    \ifx#2\finemultibold\let\next\relax\let\secondobold\relax
    \else\let\next\grecomultibold
    \fi\expandafter\next\secondobold}

\def\grecobolddef#1{%
  \edef\dadef{bf\expandafter\decapita\string#1}%
  \expandafter\def\csname\dadef\endcsname{{\neretto #1}}}

\def\neretto#1{\setbox0=\hbox{\mathsurround=0pt$#1$}%
  \kern.02em\copy0 \kern-\wd0
  \kern-.02em\copy0 \kern-\wd0
  \raise.03em\box0 \kern.02em}
\grecomultibold\alpha\gamma\delta\xi\sigma\tau\eta\theta\varepsilon\rho\pi\varphi\varepsilon\finemultibold
\def\bdiv{\mathop{\bf div}\nolimits}

\def\curl{\mathop{\rm curl}\nolimits}
\def\bcurl{\mathop{\bf curl}\nolimits}

\def\teps{{\bfvarepsilon}}

\def\A{{\mathcal A}}

\def\Ih{{\mathcal I}_h}
\def\IE{{\mathcal I}_E}

\def\Mh{{\mathcal M}_h}

\def\wbox#1;#2;{\vbox{\hrule\hbox{\vrule height#1mm\kern#2mm\vrule
  height#1mm}\hrule}}

\let\phi\varphi

\let\b\b

\newcommand{\bba}   {\mathbf{a}}

\newcommand{\bbc}   {\mathbf{c}}

\newcommand{\bbe}   {\mathbf{e}}
\newcommand{\bbf}   {\mathbf{f}}

\newcommand{\bbm}   {\mathbf{m}}
\newcommand{\bbn}   {\mathbf{n}}

\newcommand{\bbr}   {\mathbf{r}}

\newcommand{\bbt}   {\mathbf{t}}
\newcommand{\bbu}   {\mathbf{u}}
\newcommand{\bbv}   {\mathbf{v}}
\newcommand{\bbw}   {\mathbf{w}}
\newcommand{\bbx}   {\mathbf{x}}

\def\C{\mathbb C}

\def\bo#1{\text{\bf{#1}}}

\def\hpoint#1.#2.#3{{\underline{#1}}_{#2}\cdot
 {\underline{\mathop{\smash{#3}\vphantom{{#1}_{#2}}}}}}

\def\npointp#1.#2{{\underline{#1}}\cdot
 {\underline{\mathop{\smash{#2}\vphantom{{#1}}}}}}

\def\bx{\bo{x}}

\def\beq{\begin{equation}}
\def\enq{\end{equation}}

\def\bfzero{{\bf 0}}

\author{
E. Artioli\thanks{Department of Civil Engineering and Computer Science,
	University of Rome Tor Vergata,
	Via del Politecnico 1, 00133 Rome, Italy,
	{\tt artioli@ing.uniroma2.it}},
S. {d}e Miranda\thanks{DICAM, University of Bologna, Viale Risorgimento 2, 40136 Bologna, Italy,
	{\tt stefano.demiranda@unibo.it}},
C. Lovadina\thanks{
Dipartimento di Matematica, Universit\`a di Milano, Via Saldini 50, 20133 Milano,
and IMATI del CNR, Via Ferrata 1, 27100 Pavia, Italy,
{\tt carlo.lovadina@unimi.it}},
L. Patruno\thanks{DICAM, University of Bologna, Viale Risorgimento 2, 40136 Bologna, Italy, {\tt luca.patruno@unibo.it}}
}

\date{}

\title{A Dual Hybrid Virtual Element Method for Plane Elasticity Problems}

\begin{document}

\maketitle

\begin{abstract}
A dual hybrid Virtual Element scheme for plane linear elastic problems is presented and analysed. In particular, stability and convergence results have been established. The method, which is first order convergent, has been numerically tested on two benchmarks with closed form solution, and on a typical microelectromechanical system. The numerical outcomes have proved that the dual hybrid scheme represents a valid alternative to the more classical low-order displacement-based Virtual Element Method.
\end{abstract}
{
\section{Introduction}

The Virtual Element Method (VEM) is a recent methodology to approximate partial differential equation problems. Introduced in \cite{volley}, it is a Galerkin method which can be considered as an evolution of the Finite Element Method (FEM). In contrast to FEM, VEM is able to naturally manage several mesh complexities, such as polytopal shapes or hanging nodes. In addition to this flexibility, it has been realized that VEM is also able to efficiently deal with other non-trivial situations, for instance problems with internal constraints (incompressibility for solids and fluids is an example), or problems with high-continuity requirements (the fourth-order Kirchhoff plate is an example). The price to pay is that the shape functions are not {\em explicitly} known and thus they are called {\em virtual}. However, the available information on them is sufficient to form the stiffness matrix and the right-hand side of the discretized problem. For the analysis of the VEM technique for the basic second-order elliptic problems, we refer to \cite{volley,BLRXX,Brenner-Guan-Sung}. 

Focusing on the linear elasticity problem, the VEM philosophy has been already conjugated in several ways. In the easiest framework, the displacement-based variational formulation, VEMs have been considered in \cite{BBM-VEMelast,ABLS_part_I,GTP14}, and a procedure to recover an accurate stress field has been proposed in \cite{Artioli_IJNME_2019}. For nearly-incompressible materials, the VEM schemes proposed in \cite{BLV} for the Stokes problem has been applied in connection with the displacement/pressure formulation, see for example \cite{Hughes:book,TA2018}.  Furhermore, the Hellinger-Reissner variational principle has been recently employed to develop Virtual Element Methods with approximated stresses exhibiting, {\em a priori}, symmetry and inter-element traction continuity features, see \cite{ADLP_HR,ADLP_HR_FAM}. Among other works regarding linear elasticity problems, also concerning incompressible materials or plate structures, as well as non-conforming schemes or different variational frameworks, we here mention \cite{Antonietti-Beirao-Mora-Verani,Brezzi-Marini-VEMplate,Zhang-Zhao-Yang,Gatica-pseudo,Beirao-Mora-Rivera,MGV,ZZCM,Chinosi-plate,ZHANG20181,ZHANG201949}.

In the present contribution, we explore the possibility to develop VEM schemes in the framework of the so-called dual hybrid formulation, see \cite{Bo-Bre-For} for instance. Thus, we are concerned with a variational setting where the unknowns are both the stress and the displacement fields. Once the computational domain is partitioned into (polytopal) elements, the stress field is assumed to {\em a priori} satisfy the equlibrium equation locally on each element. It is then required to locally maximize the complementary energy; the displacement field enters into play only on the interelement boundaries and plays the r\^ole of the Lagrange multiplier for the traction continuity constraint. For Finite Elements, this variational approach has been used in \cite{PianWu,PianSumihara}, for example. In our VEM scheme, the local stress space is borrowed from the one introduced in \cite{ADLP_HR}, while the standard low-order nodal Virtual Element space is essentially used for the displacement field. As usual for VEMs, a local polynomial projection is introduced as a basic ingredient to form the stiffness matrix of the method. Specifically, here, given a virtual stress, the VEM projection returns a suitable computable polynomial stress. In this paper we present and numerically investigate two different projections, the second of which is a low-cost improvement of the first one (already used in \cite{ADLP_HR}). We remark that hybrid formulations for the linear elastic problem are not only interesting {\em per se}, but also they may be used as a building block for other more complex situations, for instance plate problems (see \cite{deMirandaUbertini}).

A brief outline of the paper is as follows. In Section \ref{sec:1} we present the (2D) elastic problem together with its dual hybrid variational formulation. Section \ref{s:HR-VEM} describes the Virtual Element approximation we propose. In Section \ref{s:theoretical} we develop the stability and convergence analysis, by using a suitable mesh-dependent norm for the stress field. The mumerical results, which confirm the theoretical predictions, are detailed in Section \ref{s:numer}, while some concluding remarks are drawn in Section \ref{s:conclusions}.

Throughout the paper, given two quantities $a$ and $b$, we use the notation $a\lesssim b$ to mean: there exists a constant $C$, independent of the mesh-size, such that $a\leq C\, b$. If $a\lesssim b$ and $b\lesssim a$, we will write $a\approx b$. Moreover, we use standard notations for Sobolev spaces, norms and semi-norms (cf. \cite{Lions-Magenes}, for example). Finally, given a subset $\omega\subseteq \R^2$, we will denote with $\P_k(\omega)$ ($k\ge 0$) the space of polynomials of degree up to $k$ and defined on $\omega$.

\section{The elasticity problem in the dual hybrid form}\label{sec:1}

In this section we briefly present the linear elasticity problem in dual hybrid form. More details about the dual hybrid formulations of second-order problems can be found in \cite{Bo-Bre-For}. We start by considering the strong form of the problem we are interested in:


\begin{equation}\label{strong}
\left\lbrace{
	\begin{aligned}
	&\mbox{Find } (\bfsigma,\bbu)~\mbox{such that}\\
	&-\bdiv \bfsigma= \bbf\qquad \mbox{in $\Omega$}\\
	& \bfsigma = \C \teps(\bbu)\qquad \mbox{in $\Omega$}\\
	&\bbu_{|\partial\Omega}=\bfzero,
	\end{aligned}
} \right.
\end{equation}
where $\O\subset \R^2$ is a polygonal domain, $\bfsigma$ and $\bbu$ are the unknown stress and displacement fields, respectively. Moreover, $\bbf$ represents the body force density, $\C$ is the elasticity tensor and $\teps(\cdot)$ is the usual symmetric gradient operator. We consider only the clamped case along the whole boundary, but other boundary conditions can be treated in standard ways (see \cite{Bo-Bre-For}, for instance).

Let now $\{\mathcal{T}_h\}$ be a sequence of decompositions of $\Omega$ into general polygonal elements $E$ with
\[
h_E := {\rm diameter}(E) , \quad
h := \sup_{E \in \mathcal{T}_h} h_E .
\]
Given $\Th$, let us denote with $\Eh =\bigcup_{E\in\Th}\partial E$ the skeleton of $\Th$. The dual hybrid formulation of Problem \eqref{strong} is a variation approach for which:

\begin{itemize}

\item the stress field $\bfsigma$ {\em a-priori} satisfies the equilibrium equation in each element $E\in\Th$;

\item the displacement field $\bbu$ enters into play essentially only on the skeleton $\Eh$, where it acts as a Lagrange multiplier for the continuity of the tractions $\bfsigma\bbn$.

\end{itemize}

More precisely, defining with $(\cdot,\cdot)_E$ the scalar product in $L^2(E)$,
and $a_E(\bfsigma,\bftau):=(\D \bfsigma, \bftau)_E$, with $\D$ the compliance tensor (i.e. the inverse of the tensor $\C$), the dual hybrid formulation stems from considering the critical points of the following functional:

\begin{equation}\label{eq:dh-functional}
{\mathcal E}(\bftau,\bbv)= -\frac{1}{2}\sum_{E\in\Th}a_E(\bftau,\bftau) +\sum_{E} \int_{\partial E} \bftau\bbn\cdot \bbv
\qquad \bftau\in\Sigma^f\, , \, \bbv\in U^0 .
\end{equation}
Above, the spaces $\Sigma^f$ and $U^0$ are defined by:

\begin{equation}\label{eq:cont-spaces}
\left\lbrace{
\begin{aligned}
&\Sigma^f = \prod_ {E\in\Th} \Sigma^f(E), \\
&U^0 =  H^1_0(\Omega)^2 ,
\end{aligned}
} \right.
\end{equation}
where

\begin{equation}\label{eq:cont-loc_spaces}
	\Sigma^f(E) = \Big\{ \bftau\in H(\bdiv;E)_s\, :\, \bdiv\bftau +\bbf =\bfzero \Big\}
\end{equation}
and
$$
H(\bdiv;E)_s=\Big\{ \bftau\in L^2(E)^{2\times 2}\, :\, \mbox{$\bftau$ is symmetric,}\, \bdiv\bftau\in L^2(E)^2  \Big\}.
$$
%
In $\Sigma^f$ we introduce the obvious norm:

\begin{equation}\label{eq:sigma-norm}
||\bftau ||_\Sigma := \left(
\sum_{E\in\Th} ( || \bftau ||_{0,E}^2 + || \bdiv\bftau ||_{0,E}^2)
\right)^{1/2} .
\end{equation}

Selecting a {\em particular locally self-equilibrated symmetric stress solution} $\widehat{\bfsigma}_f\in \Sigma^f$, i.e. such that
$(\bdiv\widehat{\bfsigma}_f +\bbf)_{|E}=\bfzero$ for every $E\in\Th$, the stress solution $\bfsigma$ can be decomposed as

\begin{equation}\label{eq:sigma-dec}
\bfsigma = \bfsigma^0 + \widehat{\bfsigma}_f \qquad \mbox{with $\bfsigma^0\in \Sigma^0$}.
\end{equation}
Consequently, stationarity of functional \eqref{eq:dh-functional} leads to the variational problem for the unknowns $\bfsigma^0$ and $\bbu$:

\begin{equation}\label{cont-pbl}
\left\lbrace{
\begin{aligned}
&\mbox{Find } (\bfsigma^0,\bbu)\in \Sigma^0\times U^0~\mbox{such that}\\
&a(\bfsigma^0,\bftau^0) + b(\bftau^0, \bbu)= F(\bftau^0) \quad \forall \bftau^0\in \Sigma^0\\
& b(\bfsigma^0, \bbv) = G(\bbv) \quad \forall \bbv\in U^0 ,
\end{aligned}
} \right.
\end{equation}
where $\Sigma^0$ is defined according with \eqref{eq:cont-spaces} and \eqref{eq:cont-loc_spaces}, choosing $\bbf=\bfzero$.
Above, for $\bftau^0\in\Sigma^0$ and $\bbv\in U^0$, we have set:

\begin{equation}\label{eq:cont-glob_forms}
\left\lbrace{
	\begin{aligned}
	&a(\bfsigma^0,\bftau^0) = \sum_{E\in\Th} a_E(\bfsigma^0,\bftau^0) ,\\
	&b(\bftau^0, \bbv) = - \sum_{E\in\Th} \int_{\partial E}\bftau^0\bbn\cdot\bbv ,\\
	& F(\bftau^0 ) = - \sum_{E\in\Th} a_E(\widehat{\bfsigma}_f,\bftau^0)\\
	& G(\bbv) = \sum_{E\in\Th} \int_{\partial E}\widehat{\bfsigma}_f\bbn\cdot\bbv  .
	\end{aligned}
} \right.
\end{equation}
Once $\bfsigma^0$ has been found, the stress solution $\bfsigma$ is simply recovered by using \eqref{eq:sigma-dec}. We remark that $\bbu$, part of the solution to Problem \eqref{cont-pbl}, is not unique, but it is defined up to an element of the subspace

\begin{equation}\label{kerbt}
H = \Big\{ \bbv\in U^0 \, :\, b(\bftau,\bbv) = 0 \quad \forall\, \bftau\in \Sigma^0  \Big\} = \Big\{ \bbv\in U^0 \, :\, \bbv_{|\Eh}  = \bfzero  \Big\} .
\end{equation}
Moreover, the following {\em inf-sup} condition holds (see \cite{Bo-Bre-For}).

\begin{equation}\label{eq:cont-inf-sup}
\sup_{\bftau\in \Sigma^0}\frac{b( \bftau,\bbv)}{|| \bftau||_\Sigma }\gtrsim ||\bbv||_{U^0/H} \qquad\forall\bbv\in U^0.
\end{equation}

In the sequel, we will also consider the following important subspace of $\Sigma^0$:

\begin{equation}\label{eq:kernel}
K = \Big\{ \bftau\in \Sigma^0 \, :\, b(\bftau,\bbv) = 0 \quad \forall\, \bbv\in U^0  \Big\} = \Big\{ \bftau\in H(\bdiv;\O) \, :\, \bdiv\bftau = \bfzero  \Big\} .
\end{equation}
We remark that, on $K$, we have $||\bftau||_\Sigma = ||\bftau||_{0,\O}$ (cf. \eqref{eq:sigma-norm}), and it holds:

\begin{equation}\label{eq:elker-cont}
a(\bftau,\bftau)\gtrsim ||\bftau||_\Sigma^2\qquad \forall \bftau\in K .
\end{equation}

Obviously, Problem \eqref{cont-pbl} is equivalent to:

\begin{equation}\label{cont-pbl-v2}
\left\lbrace{
	\begin{aligned}
	&\mbox{Find } (\bfsigma,\bbu)\in \Sigma^f\times U^0~\mbox{such that}\\
	&a(\bfsigma,\bftau^0) + b(\bftau^0, \bbu)= 0 \quad \forall \bftau^0\in \Sigma^0\\
	& b(\bfsigma, \bbv) = 0 \quad \forall \bbv\in U^0 ,
	\end{aligned}
} \right.
\end{equation}

From \eqref{eq:cont-inf-sup} and \eqref{eq:elker-cont}, the general theory of mixed methods (see \cite{Bo-Bre-For}, for instance) gives that Problem \eqref{cont-pbl} has a unique solution $(\bfsigma^0,\bbu)\in \Sigma^0\times U^0/H$ (with an abuse of notation, we here use $\bbu$ to denote the equivalence class of the function $\bbu$ in $U^0/H$). Therefore, we infer that also Problem \eqref{cont-pbl-v2} has a unique solution $(\bfsigma,\bbu)\in \Sigma^0\times U^0/H$. The quotient space $U^0/H$ essentially means that in Problems \eqref{cont-pbl} and \eqref{cont-pbl-v2} the displacement field $\bbu$ enters only through its trace $\bbu_{|\Eh}$ on the skeleton $\Eh$, see \eqref{kerbt}.
The following result can be deduced using the discussion in \cite{RT1991}, and justifies the variation framework \eqref{cont-pbl} (or \eqref{cont-pbl-v2}) for the elasticity Problem \eqref{strong}.

\begin{prop}\label{pr:equival}
Let $(\bfsigma,\bbu)$ be a sufficiently smooth solution to Problem \eqref{strong}. Then $(\bfsigma,\bbu)\in \Sigma^0\times U^0/H$ is the solution of Problem \eqref{cont-pbl-v2}. \qed
\end{prop}

\section{The Virtual Element Method}
\label{s:HR-VEM}

We now outline the Virtual Element Method we propose for the discretization of Problem \eqref{cont-pbl}.
To develop the theoretical analysis of the scheme, we suppose that for all $h$, each element $E$ in $\mathcal{T}_h$ fulfils the following assumptions:
\begin{itemize}
	\item $\mathbf{(A1)}$ $E$ is star-shaped with respect to a ball of radius $ \ge\, \gamma \, h_E$,
	\item $\mathbf{(A2)}$ the distance between any two vertexes of $E$ is $\ge c \, h_E$,
\end{itemize}
where $\gamma$ and $c$ are positive constants. We remark that the hypotheses above, though not too restrictive in many practical cases,
can be further relaxed, as noted in ~\cite{volley}.

We suppose that the compliance tensor $\D$ is piecewise constant with respect to the underlying mesh $\Th$. The analysis for a general (sufficiently smooth) tensor $\D$ follows from this case using the same arguments of \cite{ADLP_HR}. In addition, we assume that the load term $\bbf$ is piecewise constant with respect to the underlying mesh $\Th$. If $\bbf$ is indeed smooth, this modification introduces an $O(h)$ perturbation on the solution which does not spoil the convergence rate of our scheme.

\subsection{The local spaces}\label{ss:E-spaces}

Given a polygon $E\in\Th$ with $n_E$ edges, we first introduce the space of local infinitesimal rigid body motions:

\begin{equation}\label{eq:rigid}
RM(E)=\left\{ \bbr(\bbx) = \bba + b(\bbx -\bbx_C)^\perp \quad \bba\in\R^2, \ b\in \R  \right\}.
\end{equation}
Here above, given $\bbc=(c_1,c_2)^T\in\R^2$, $\bbc^\perp$ is the clock-wise rotated vector  $\bbc^\perp=(c_2,-c_1)^T$, and $\bbx_C$ is the baricenter of $E$. For each edge $e$ of $\partial E$, we introduce the space

\begin{equation}\label{eq:edge_approx}
R(e)=\left\{ \bbt(s) = \bbc + d\,s\,\bbn \quad \bbc\in\R^2, \ d\in \R, \ s\in [-1/2,1/2]  \right\}.
\end{equation}
Here above, $s$ is a local linear coordinate on $e$, such that $s=0$ corresponds to the edge midpoint. Furthermore, $\bbn$ is the outward unit normal to the edge $e$. Hence, $R(e)$ consists of vectorial functions which are constant in the edge tangential direction, while they are linear along the edge normal direction. Then, we set:

\begin{equation}\label{eq:local_stress}
\begin{aligned}
 \Sigma_h(E)=\Big\{ \bftau_h\in & H(\bdiv;E)\ :\ \exists \bbw^\ast\in H^1(E)^2 \mbox{ such that } \bftau_h=\C\teps(\bbw^\ast);\\  &(\bftau_h\,\bbn)_{|e}\in R(e) \quad \forall e\in \partial E;\quad
\bdiv\bftau_h\in RM(E) \Big\}.
\end{aligned}
\end{equation}

\begin{remark} Alternatively, the space \eqref{eq:local_stress} can be defined as follows.
	\begin{equation}\label{eq:local_stress-alt}
	\begin{aligned}
	 \Sigma_h(E)=\Big\{ \bftau_h\in & H(\bdiv;E)_s\ :\  \curl \bcurl(\D \bftau_h) = 0 ; \\  &(\bftau_h\,\bbn)_{|e}\in R(e) \quad \forall e\in \partial E;\quad
	 \bdiv\bftau_h\in RM(E) \Big\}.
	\end{aligned}
	\end{equation}
Here above, the equation $\curl \bcurl (\D \bftau_h) = 0$ is to be intended in the distribution sense.	
\end{remark}

We remark that, once $(\bftau_h\,\bbn)_{|e}=\bbc_e +d_es\,\bbn$ is given for all $e\in\partial E$, cf. \eqref{eq:edge_approx}, the quantity $\bdiv\bftau_h\in RM(E)$ is determined. Indeed, denoting with $\bfvarphi:\partial E\to \R^2$ the function such that $\bfvarphi_{|e}:=\bbc_e+d_es\,\bbn$, the integration by parts formula

\begin{equation}\label{eq:compat}
\int_E \bdiv\bftau_h\cdot \bbr =
\int_{\partial E}\bfvarphi\cdot \bbr \qquad \forall \bbr\in RM(E)
\end{equation}
allows to compute $\bdiv\bftau_h$ using the $\bbc_e$'s and the $d_e$'s. More precisely, setting (cf \eqref{eq:rigid})

\begin{equation}\label{eq:div1}
\bdiv\bftau_h = \bfalpha_E + \beta_E (\bbx -\bbx_C)^\perp ,
\end{equation}
from \eqref{eq:compat} we infer

\begin{equation}\label{eq:div2}
\left\lbrace{
	\begin{aligned}
&\bfalpha_E =\frac{1}{|E|}\int_{\partial E}\bfvarphi = \frac{1}{|E|}\sum_{e\in\partial E}\int_e \bbc_e \\
&\beta_E = \frac{1}{ \int_E | \bbx -\bbx_C |^2 }\int_{\partial E} \bfvarphi\cdot (\bbx -\bbx_C)^\perp  = \frac{1}{ \int_E | \bbx -\bbx_C |^2 }\sum_{e\in\partial E}\int_e (\bbc_e+d_e s\,\bbn)\cdot (\bbx -\bbx_C)^\perp.
\end{aligned}
} \right.
\end{equation}

We now define the affine space for the local approximation of the stress field:

\begin{equation}\label{eq:affine-sigma}
\Sigma^f_h(E) = \Big\{ \bftau_h\in \Sigma_h(E)\, :\, \bdiv\bftau_h + \bbf =\bfzero \Big\} .
\end{equation}
The space $\Sigma^0_h(E)$ is defined according with \eqref{eq:affine-sigma}, choosing $\bbf=\bfzero$.


The local approximation space for the displacement field is defined as follows:

\begin{equation}\label{eq:local_displ}
U_h(E)=\Big\{ \bbv_h\in  H^1(E)^2\ :\ \bbv_{h|e}\in \P_1(e)^2 \quad\forall e\in \partial E\Big\}.
\end{equation}
We notice that $U_h(E)$ is an infinite dimensional space. However, this will not lead to any computational trouble, since only $\bbv_{h|\partial E}$, the trace of $\bbv_h\in U_h(E)$ on $\partial E$, enters into play in the discrete formulation. Therefore, we may think that the degrees of freedom for $U_h(E)$ are linear functionals which uniquely determine $\bbv_h\in U_h(E)$ on $\partial E$. For instance, one may take the point values of $\bbv_h$ at the vertices of $E$.


\subsubsection{Computation of $\Sigma^f_h(E)$}
\label{sss:computations}
We now show how the space $\Sigma^f_h(E)$ can be described, i.e. how a suitable set of degrees of freedom can be selected. We first number the edges of $\partial E$ as $e_1,e_2\ldots,e_{n_E}$, once and for all. We recall that we have supposed $\bbf_{|E}\in \P_0(E)$.
%
%
%
%
The following result holds.

\begin{lem}\label{lm:sigmaf}
Let $\bbf_{|E}\in \P_0(E)$. Then, the space $\Sigma^f_h(E)$ is characterised by:

\begin{equation}\label{eq:f3}
\begin{aligned}
\Sigma^f_h(E) &= \Big\{ \bftau_h\in \Sigma_h(E)\, :\, \bbc_{n_E} = -\frac{1}{|e_{n_E}|} \Big(|E|\,\bbf_{|E} +\sum_{i=1}^{n_E-1}\int_{e_i}\bbc_{i} \Big) ,  \\
& d_{n_E} =-\frac{1}{ \int_{e_{n_E}} s\,\bbn\cdot (\bbx -\bbx_C)^\perp } \Big(  \sum_{i=1}^{n_E}\int_{e_i} \bbc_{i}\cdot (\bbx -\bbx_C)^\perp  \\
&\quad + \sum_{i=1}^{n_E-1}\int_{e_i} d_{i} s\,\bbn\cdot (\bbx -\bbx_C)^\perp \Big)   \Big\} ,
\end{aligned}
\end{equation}
where we have set $\bbc_i = \bbc_{e_i}$ and $d_i = d_{e_i}$.

\end{lem}

\begin{proof}
We notice, see \cite{ADLP_HR}, that $\int_{e_{n_E}} s\,\bbn\cdot (\bbx -\bbx_C)^\perp\ne 0$. Hence, the right-hand side of \eqref{eq:f3} is well-defined.
Now, the proof easily follows by a direct computation from \eqref{eq:div1} and \eqref{eq:div2}.
\end{proof}

\begin{remark}\label{rm:stress_space}
Obviously, the space $\Sigma^0_h(E)$ is defined by \eqref{eq:f3} by choosing $\bbf_{|E}=\bfzero$. In addition, using the results in \cite{ADLP_HR}, we infer that $\Sigma^f_h(E)$ and $\Sigma^0_h(E)$ are completely characterised once the quantities $\bbc_i$ and $d_i$ are given, for $i=1,\ldots,n_E-1$. Therefore, the dimension of both the spaces is $3n_E-3$.  	
\end{remark}


\subsection{The local bilinear forms}\label{ss:E-bforms}

Given $E\in \Th$, we first notice that, for every $\bftau_h\in  \Sigma_h(E)$ and $\bbv_h\in U_h(E)$, the term

\begin{equation}\label{eq:div-ex}
\int_{\partial E} \bftau_h\bbn\cdot \bbv_h
\end{equation}
is computable from the knowledge of the degrees of freedom. Therefore, there is no need to introduce any approximation in the structure of the terms $b( \bftau^0, \bbu)$ and $b(\bfsigma^0, \bbv)$ in problem \eqref{cont-pbl}. Instead, the term

\begin{equation}
a_E(\bfsigma_h,\bftau_h)  = \int_E \D \bfsigma_h : \bftau_h
\end{equation}
is not computable for a general couple $(\bfsigma_h,\bftau_h)\in \Sigma_h(E)\times \Sigma_h(E)$. As usual in the VEM approach (see \cite{volley}, for instance), we then need to introduce a suitable approximation $a_E^h(\cdot,\cdot)$. We first define the projection operator

\begin{equation}\label{eq:proj}
\left\{
\begin{aligned}
& \Pi_E \, :\,  \Sigma_h(E)\to \P_\ast(E)^{2\times2}_s\\
& \bftau_h \mapsto \Pi_E\bftau_h\\
&a_E(\Pi_E\bftau_h,\bfpi)  = a_E(\bftau_h,\bfpi)\qquad \forall\bfpi \in \P_\ast(E)^{2\times2}_s,
\end{aligned}
\right.
\end{equation}
where $\P_\ast(E)^{2\times2}_s$ is a suitable space of polynomial symmetric tensors.
We then set

\begin{equation}\label{eq:ah1}
\begin{aligned}
a_E^h(\bfsigma_h,\bftau_h)  &=
a_E(\Pi_E\bfsigma_h,\Pi_E\bftau_h) + s_E\left( (Id-\Pi_E)\bfsigma_h, (Id-\Pi_E)\bftau_h \right)\\
&=\int_E \D (\Pi_E\bfsigma_h) : (\Pi_E\bftau_h)  + s_E\left( (Id-\Pi_E)\bfsigma_h, (Id-\Pi_E)\bftau_h \right) ,
\end{aligned}
\end{equation}
where $s_E(\cdot,\cdot)$ is a suitable stabilization term. We propose the following:
	
	\begin{equation}\label{eq:stab1}
	s_E(\bfsigma_h,\bftau_h) : = \kappa_E\, h_E\int_{\partial E} \bfsigma_h\bbn\cdot \bftau_h\bbn  ,
	\end{equation}
	where $\kappa_E$ is a positive constant to be chosen (for instance, any norm of $\D_{|E}$). A variant of \eqref{eq:stab1} is provided by
	
	\begin{equation}\label{eq:stab1bis}
	s_E(\bfsigma_h,\bftau_h) : = \kappa_E\, \sum_{e\in\partial E}h_e\int_{e} \bfsigma_h\bbn\cdot \bftau_h\bbn.
	\end{equation}

Morevover, we will make two different choices for $\P_\ast(E)^{2\times2}_s$, namely:

\begin{equation}\label{eq:pi0}
 \P_\ast(E)^{2\times2}_s=\P_0(E)^{2\times2}_s\qquad \mbox{(constant symmetric tensor functions)}
\end{equation}
and
\begin{equation}\label{eq:pi1}
 \P_\ast(E)^{2\times2}_s=\P_1(E)^{2\times2}_s\qquad \mbox{(linear symmetric tensor functions)}.
\end{equation}

\begin{remark}\label{rm:twoproj}
We remark that choice \eqref{eq:pi1} leads to a VEM projection onto a reacher space than choice \eqref{eq:pi0}. Although we do not have any improvement in the convergence rate, we nonetheless expect more accurate discrete solutions when selecting \eqref{eq:pi1} instead of \eqref{eq:pi0}. This behaviour is generally confirmed by the numerical experiments presented in Section \ref{s:numer}.
\end{remark}
%
%
	

		%


\subsection{The local term $a_E(\widehat{\bfsigma}_f,\bftau)$}\label{ss:E-load}

In order to discretise $F(\bftau)$, see \eqref{cont-pbl}, we need to consider the term, see \eqref{eq:cont-glob_forms}:

\begin{equation}\label{eq:fh}
a_E(\widehat{\bfsigma}_f,\bftau) .
\end{equation}
To this aim, we will proceed in two steps:

\begin{enumerate}
	\item we first choose $\widehat{\bfsigma}_{h,f}$, which may be considered as a suitable approximation of $\widehat{\bfsigma}_f$;
	\item we then consider a $\widetilde a_E^h(\cdot,\cdot)$, which may be considered as a suitable approximation of $a_E(\cdot,\cdot)$, and we finally form $\widetilde a_E^h(\widehat{\bfsigma}_{h,f},\bftau_h)$ to discretise \eqref{eq:fh}.	
\end{enumerate}

In particular, recalling that $\bbf_{|E}\in\P_0(E)$, we define $\widehat{\bfsigma}_{h,f}$ as:

\begin{equation}\label{eq:part-approx}
(\widehat{\bfsigma}_{h,f})_{|E} = -
\begin{pmatrix}
(\bbf_{|E})_1 (\bx - \bbx_C)_1& 0 \\
0 & (\bbf_{|E})_2 (\bbx - \bbx_C)_2
\end{pmatrix}.
\end{equation}
Above, for a given vector $\bbv$, $(\bbv)_i$ denotes its $i$-th component.
Furthermore, we set

\begin{equation}\label{eq:fhh}
\widetilde a_E^h(\widehat{\bfsigma}_{h,f},\bftau_h) :=
a_E(\widehat{\bfsigma}_{h,f},\Pi_E\bftau_h) .
\end{equation}

\begin{remark}\label{rem:altern-part}
	A different option could be to select $\widehat{\bfsigma}_{h,f}\in\Sigma^f_h(E)$, exploiting \eqref{eq:f3} and Remark \ref{rm:stress_space}. For instance, we may set:

\begin{equation}\label{eq:fh2}
\bbc_i =\bfzero \ , \quad d_i = 0 \qquad \forall\, i=1,\ldots,n_E-1 ,
\end{equation}
and compute $\bbc_{n_E}$ and $d_{n_E}$ according to \eqref{eq:f3}. However, other choices can be made (more balanced among the polygon edges). To approximate \eqref{eq:fh}, we could consider:

\begin{equation}\label{eq:fhh-alt}
\widetilde a_E^h(\widehat{\bfsigma}_{h,f},\bftau_h) :=
a_E^h(\widehat{\bfsigma}_{h,f},\bftau_h) ,
\end{equation}
where $a_E^h(\cdot,\cdot)$ is defined as in Section \ref{ss:E-bforms}.
\end{remark}

%
%


%


\subsection{The discrete scheme}\label{ss:discrete}

We are now ready to introduce the discrete scheme. We introduce global approximation spaces, see \eqref{eq:cont-spaces}, \eqref{eq:affine-sigma} and \eqref{eq:local_displ}:

\begin{equation}\label{eq:discr-spaces}
\left\lbrace{
	\begin{aligned}
	&\Sigma^f_h = \prod_ {E\in\Th} \Sigma^f_h(E), \\
	&U^0_h = \Big(\prod_ {E\in\Th} U_h(E)\Big)\cap H^1_0(\Omega)^2 ,
	\end{aligned}
} \right.
\end{equation}


%

Furthermore, given a local approximation of $a_E(\cdot,\cdot)$, see \eqref{eq:ah1}, we set

\begin{equation}\label{eq:global-ah}
a_h(\bfsigma_h,\bftau_h):= \sum_{E\in\Th}a_E^h(\bfsigma_h,\bftau_h) .
\end{equation}

%
%
%
%
The method we consider is then defined by

\begin{equation}\label{eq:discr-pbl-ls}
\left\lbrace{
	\begin{aligned}
	&\mbox{Find } (\bfsigma_h^0,\bbu_h)\in \Sigma^0_h\times U^0_h~\mbox{such that}\\
	&a_h(\bfsigma_h^0,\bftau_h^0) + b(\bftau_h^0, \bbu_h)= F_h(\bftau_h^0) \quad \forall \bftau_h\in \Sigma^0_h\\
	& b(\bfsigma_h^0, \bbv_h) = G_h(\bbv_h)\quad \forall \bbv_h\in U^0_h .
	\end{aligned}
} \right.
\end{equation}
Above, $F_h(\cdot)$ is given by, see \eqref{eq:cont-glob_forms}, \eqref{eq:part-approx} and \eqref{eq:fhh}:

\begin{equation}\label{eq:fh-glob}
F_h(\bftau_h) = -\sum_{E\in\Th} a_E(\widehat{\bfsigma}_{h,f},\Pi_E\bftau_h)
\end{equation}
while $G_h(\bbv_h)$ reads, see \eqref{eq:cont-glob_forms}:

\begin{equation}\label{eq:gh-glob}
G_h(\bbv_h) = \sum_{E\in\Th} \int_{\partial E}\widehat{\bfsigma}_{h,f}\bbn\cdot\bbv_h .
\end{equation}

Introducing the bilinear form $\A_h:(\Sigma_h^0\times U_h^0)\times (\Sigma_h^0\times U_h^0)\to \R$ defined by

\begin{equation}\label{eq:totbilinh}
\begin{aligned}
\A_h(\bfsigma_h^0,\bbu_h ;\bftau_h^0,\bbv_h):=
a_h(\bfsigma_h^0,\bftau_h^0)  + b( \bftau_h^0, \bbu_h)+b(\bfsigma_h^0, \bbv_h) ,
\end{aligned}
\end{equation}
problem \eqref{eq:discr-pbl-ls} can be written as

\begin{equation}\label{discr-pbl-cpt}
\left\lbrace{
	\begin{aligned}
	&\mbox{Find } (\bfsigma_h^0,\bbu_h)\in \Sigma_h^0\times U_h^0~\mbox{such that}\\
	&\A_h(\bfsigma_h^0,\bbu_h;\bftau_h^0,\bbv_h)= F_h(\bftau_h^0)+G(\bbv_h)\quad \forall (\bftau_h^0,\bbv_h)\in \Sigma_h^0\times U_h^0 .
	\end{aligned}
} \right.
\end{equation}



\section{Theoretical analysis using a mesh-dependent norm}\label{s:theoretical}

A significant part of the theoretical analysis follows the guidelines developed in \cite{ADLP_HR}. Hence, in many points we limit to state the results. However, we develop here an error analysis using a mesh-dependent norm for the stresses, which turns out to be a flexible tool in our case. In addition,
we remark that the {\em inf-sup} condition of Proposition \ref{pr:fortin-glob2} is new, and therefore its proof is given in full details.

 According to the assumption $\bbf_{|E}\in \P_0(E)$, we select the locally self-equilibrated solution $\widehat{\bfsigma}_{f}$ as

\begin{equation}\label{eq:part-const}
(\widehat{\bfsigma}_{f})_{|E} = -
\begin{pmatrix}
(\bbf_{|E})_1 (\bx - \bbx_C)_1& 0 \\
0 & (\bbf_{|E})_2 (\bbx - \bbx_C)_2
\end{pmatrix}.
\end{equation}
Therefore, we have $\widehat{\bfsigma}_{f}=\widehat{\bfsigma}_{h,f}$, see \eqref{eq:part-approx}. As a consequence, we get (cf. \eqref{eq:cont-glob_forms}, \eqref{eq:fh-glob} and \eqref{eq:gh-glob})

\begin{equation}\label{eq:magic-rhs}
\begin{aligned}
& F_h(\bftau_h) - F(\bftau_h) = \sum_{E\in\Th}\left(a_E(\widehat{\bfsigma}_{f},\bftau_h) -a_E(\widehat{\bfsigma}_{f},\Pi_E\bftau_h)\right) =
\sum_{E\in\Th}a_E(\widehat{\bfsigma}_{f}- \Pi_E \widehat{\bfsigma}_{f},\bftau_h) \\
& G_h(\bbv_h) - G(\bbv_h) = 0 .
\end{aligned}
\end{equation}

\subsection{Stability conditions}\label{ss:stability}
We introduce the following mesh-dependent quantity in $\Sigma_h := \prod_ {E\in\Th} \Sigma_h(E)$:

\begin{equation}\label{eq:meshdep}
||\bftau_h||_h^2 : = \sum_{E\in\Th} ||\bftau_h\bbn ||_{h,\partial E}^2 \qquad \mbox{ where
	$||\bftau_h\bbn ||_{h,\partial E} :=h_E^{1/2} ||\bftau_h\bbn ||_{0,\partial E}$} .
\end{equation}
It is easily seen that \eqref{eq:meshdep} defines a norm on $\Sigma_h$, see \eqref{eq:local_stress} along with \eqref{eq:div1} and \eqref{eq:div2}. Moreover, we have the following lemma.

\begin{lem}\label{lm:equiv_norm}
Under assumptions $\mathbf{(A1)}$ and $\mathbf{(A2)}$, it holds:

\begin{equation}\label{eq:normequiv}
||\bftau_h||_{0,\O}\lesssim ||\bftau_h||_h \lesssim ||\bftau_h||_{0,\O} \qquad \forall\,\bftau_h \in\Sigma_h.
\end{equation}

\end{lem}

\begin{proof}
Fix $E\in\Th$. Using Lemma 5.1 of \cite{ADLP_HR} we get

\begin{equation}\label{eq:first-equiv}
||\bftau_h||_{0,E}\lesssim h_E || \bdiv \bftau_h ||_{0,E} + h_E^{1/2}||\bftau_h\bbn ||_{0,\partial E}.
\end{equation}
Recalling that $\bdiv\bftau_h\in RM(E)$, an integration by parts, the Agmon inequality (see for instance \cite{Agmon}) and an inverse estimate on polygons (see Lemma 6.3 of \cite{BLRXX}) give:

\begin{equation}\label{eq:first-equiv2}
\begin{aligned}
|| \bdiv \bftau_h ||_{0,E}^2 &= \int_E \bdiv \bftau_h\cdot \bdiv \bftau_h =
\int_{\partial E}\bftau_h\bbn\cdot\bdiv\bftau_h\lesssim ||\bftau_h\bbn||_{0,\partial E}||\bdiv\bftau_h||_{0,\partial E}\\
& \lesssim ||\bftau_h\bbn||_{0,\partial E} \left( h_E^{-1/2} ||\bdiv\bftau_h||_{0,E} + |\bdiv\bftau_h|_{1,E}\right)\\
& \lesssim ||\bftau_h\bbn||_{0,\partial E}\, h_E^{-1/2} ||\bdiv\bftau_h||_{0,E} .
\end{aligned}
\end{equation}
Hence it holds

\begin{equation}\label{eq:first-equiv3}
|| \bdiv \bftau_h ||_{0,E}  \lesssim h_E^{-1/2} ||\bftau_h\bbn||_{0,\partial E}.
\end{equation}
Combining \eqref{eq:first-equiv} with \eqref{eq:first-equiv3} we obtain

\begin{equation}\label{eq:first-equiv4}
|| \bftau_h ||_{0,E}  \lesssim h_E^{1/2} ||\bftau_h\bbn||_{0,\partial E},
\end{equation}
from which the first estimate of \eqref{eq:normequiv} follows.

We now prove the second estimate in \eqref{eq:normequiv}.
For every edge $e$ in $E$, let us denote with $b_e$ the function defined on $\partial E$ such that: on $e$ it is the quadratic bubble with $\sup\,b_e =1$, and $b_e=0$ elsewhere. One has:

\begin{equation}\label{eq:second-equiv}
\begin{aligned}
|| \bftau_h\cdot\bbn||_{0,\partial E}^2&\lesssim \int_{\partial E} \bftau_h\cdot\bbn\cdot \big(\sum_{e\subset \partial E} b_e \bftau_h\cdot\bbn\big)
\leq ||\bftau_h\cdot\bbn\cdot||_{-1/2,\partial E}\, \big| \sum_{e\subset \partial E} b_e \bftau_h\cdot\bbn\big|_{1/2,\partial E}	\\
&\lesssim ||\bftau_h\cdot\bbn\cdot||_{-1/2,\partial E}\,  h_E^{-1/2} \big|\big| \sum_{e\subset \partial E} b_e \bftau_h\cdot\bbn\big|\big|_{0,\partial E}\\
&\lesssim h_E^{-1/2} ||\bftau_h\cdot\bbn||_{-1/2,\partial E}\, ||\bftau_h\cdot\bbn||_{0,\partial E} ,
\end{aligned}
\end{equation}
where we have also used a 1D inverse estimate. Hence, after exploiting a scaled trace inequality (cf. \cite{BLRXX}), we get

\begin{equation}\label{eq:second-equiv2}
h_E^{1/2}\, || \bftau_h\cdot\bbn||_{0,\partial E} \lesssim ||\bftau_h\cdot\bbn||_{-1/2,\partial E} \lesssim ||\bftau_h||_{0,E} + h_E ||\bdiv \bftau_h||_{0,E} .
\end{equation}
Now, using the technique of Lemma 6.3 in \cite{BLRXX}, we obtain

$$||\bdiv \bftau_h||_{0,E}\lesssim h_E^{-1}|| \bftau_h||_{0,E}.$$
Therefore, we have:

\begin{equation}\label{eq:second-equiv3}
h_E^{1/2}\, || \bftau_h\cdot\bbn||_{0,\partial E} \lesssim ||\bftau_h||_{0,E} ,
\end{equation}
from which the second inequality in \eqref{eq:normequiv} follows.

\end{proof}

As it is well-known (see for instance \cite{Bo-Bre-For} or \cite{Braess:book}), stability for problems with the format as in \eqref{eq:totbilinh}, is implied by the satisfaction of two conditions: the {\em ellipticity-on-the-kernel} condition, and the {\em inf-sup} condition. As far as the first one is concerned, we first prove that it holds:

\begin{equation}
\label{eq:ellkerh}
||\bftau_h||_\Sigma^2 \lesssim a_h(\bftau_h,\bftau_h)\qquad \forall\,\bftau_h\in K_h.
\end{equation}
Above $K_h\subseteq \Sigma^0_h$ is the discrete kernel, defined by:

\begin{equation}\label{eq:kernelh0.1}
K_h = \Big\{ \bftau_h\in\Sigma^0_h\, :\, b(\bftau_h,\bbv_h)=0\quad \forall \bbv_h\in U^0_h    \Big\} .
\end{equation}

Estimate \eqref{eq:ellkerh} is a consequence of the following stronger result.

\begin{lem}\label{lm:elker}
Suppose that  assumptions $\mathbf{(A1)}$ and $\mathbf{(A2)}$ are fulfilled. Then it holds:

\begin{equation}
\label{eq:ellkerh0.2}
||\bftau_h||_\Sigma^2 = || \bftau_h||_{0,\O}^2\lesssim a_h(\bftau_h,\bftau_h)\qquad \forall\,\bftau_h\in \Sigma^0_h.
\end{equation}	
	
\end{lem}	

\begin{proof}
By the norm definition \eqref{eq:sigma-norm} and the space definition \eqref{eq:affine-sigma}, we immediately infer

\begin{equation}\label{eq:ellkerh1}
||\bftau_h||_\Sigma = || \bftau_h||_{0,\O} \qquad \forall\, \bftau_h\in\Sigma_h^0.
\end{equation}
Using \eqref{eq:normequiv} and standard VEM arguments as in \cite{volley}, we get

\begin{equation}\label{eq:ellkerh2}
|| \bftau_h||_{0,\O}^2\lesssim a_h(\bftau_h,\bftau_h)\lesssim  || \bftau_h||_{0,\O}^2 \qquad \forall\, \bftau_h\in\Sigma_h^0.
\end{equation}
Estimate \eqref{eq:ellkerh} follows from \eqref{eq:ellkerh1} and \eqref{eq:ellkerh2}.

\end{proof}

From Lemma \ref{lm:elker} and \eqref{eq:normequiv} we infer the coercivity property:

\begin{equation}
\label{eq:ellker-nd}
||\bftau_h||_h^2 \lesssim a_h(\bftau_h,\bftau_h)\qquad \forall\,\bftau_h\in \Sigma^0_h.
\end{equation}	

The following lemma will be useful to prove the {\em inf-sup} condition.

\begin{lem}\label{lm:midpoint-norm}
Suppose that assumptions $\mathbf{(A1)}$ and $\mathbf{(A2)}$ are fulfilled, and fix $E\in\Th$. Take any $\bbw\in E$, node of $\Th$. Then it holds:

\begin{equation}\label{eq:mp1}
h_E \left( |\bbv_h(\bbw)|^2 + \sum_{\bbm\in\Mh(E)} |\bbv_h(\bbm)|^2 \right) \gtrsim ||\bbv_h||^2_{0,\partial E}
\qquad \forall \bbv_h\in U_h(E),
\end{equation} 	
where $\Mh(E)$ is the set of edge mid-points of $E$.

\end{lem}

\begin{proof} Denote with $\{ \bbw=\bbw_1, \bbw_2 \ldots, \bbw_{n_E} \}$ the set of nodes for $\Th$ on $\partial E$, ordered counter-clockwise. Furthermore, let $e_i = [\bbw_i, \bbw_{i+1}]$ ($i=1,\ldots,n_E$) be the edges of $E$ (here we have set $\bbw_{n_E+1}=\bbw_1$), and let $\bbm_i$ be the midpoint of $e_i$.
Fix $\bbv_h\in U_h(E)$; by using the Cavalieri-Simpson rule, we get
	
	\begin{equation}\label{eq:mp2}
	\int_{\partial E} |\bbv_h|^2 =
	\sum_{i=1}^{n_E} \frac{|e_i|}{6}\, \left[ |\bbv_h(\bbw_i)|^2 + 4\,|\bbv_h(\bbm_i)|^2  +
	|\bbv_h(\bbw_{i+1})|^2
	\right] .
	\end{equation}
	Due to assumption $\mathbf{(A2)}$, we get
	
	\begin{equation}\label{eq:mp3}
	\begin{aligned}
	\int_{\partial E} |\bbv_h|^2 &\approx
	h_E \sum_{i=1}^{n_E} \left[ |\bbv_h(\bbw_i)|^2 + 4\,|\bbv_h(\bbm_i)|^2  +
	|\bbv_h(\bbw_{i+1})|^2
	\right] \\
	&\approx h_E\left(
	\sum_{i=1}^{n_E} |\bbv_h(\bbw_i)|^2 + \sum_{i=1}^{n_E} |\bbv_h(\bbm_i)|^2
	\right) .
	\end{aligned}
	\end{equation}
We now notice that, since each component of $\bbv_h$ is a piecewise linear and continuous function on $\partial E$, it follows that for $i=2,\ldots,n_E$, the quantity $\bbv_h(\bbw_i)$ is uniquely determined by $\bbv_h(\bbw_1)$ and $\{ \bbv_h(\bbm_1),\ldots, \bbv_h(\bbm_{i-1}) \}$. Indeed, a direct computation shows that

$$
\bbv_h(\bbw_i) = 2\sum_{k=1}^{i-1} (-1)^{i-1-k} \bbv_h(\bbm_k) + (-1)^{i-1} \bbv_h(\bbw_1) \qquad i=2,\ldots,n_E .
$$
Hence, assumption $\mathbf{(A1)}$ and $\mathbf{(A2)}$ implies that

\begin{equation}\label{eq:mp4}
| \bbv_h(\bbw_i) |^2 \lesssim  |\bbv_h(\bbw_1)|^2 + \sum_{i=1}^{n_E} |\bbv_h(\bbm_i)|^2
\qquad i=1,\ldots,n_E .
\end{equation}
Estimate \eqref{eq:mp1} now follows from a combination of \eqref{eq:mp3} and \eqref{eq:mp4}.
	
\end{proof}

\begin{remark}\label{rm:evenodd}
It is easy to see that, if $n_E$ is odd, then the values $\bbv_h(\bbw_i)$ ($1\le i \le n_E$) can be determined without using $\bbv_h(\bbw)$, but only the midpoint values $\{ \bbv_h(\bbm_1),\ldots, \bbv_h(\bbm_{n_E}) \}$.  	
\end{remark}

We now prove the following local {\em inf-sup} condition.

\begin{lem}\label{lm:inf-sup-loc}
Suppose that assumptions $\mathbf{(A1)}$ and $\mathbf{(A2)}$ are fulfilled. Then, there exists $\beta>0$ such that
	
	\begin{equation}\label{eq:inf-sup-loc}
	\sup_{\bftau_h\in \Sigma_h(E)}\frac{b_E( \bftau_h,\bbv_h)}{|| \bftau_h\bbn||_{0,\partial E}}\ge \beta\, ||\bbv_h||_{0,\partial E}\qquad \forall\, \bbv_h\in U_h(E) ,
	\end{equation}
where the bilinear form $b_E(\cdot,\cdot)$ is defined by (cf. \eqref{eq:cont-glob_forms}):

\begin{equation}\label{eq:inf-sup-loc2}
b_E( \bftau,\bbv)= -\int_{\partial E}\bftau\bbn\cdot\bbv \qquad \forall\, (\bftau,\bbv)\in \Sigma(E)\times U(E) .	
\end{equation}

\end{lem}

\begin{proof}
Fix $\bbv_h\in U_h(E)$ and choose $\bftau_h\in\Sigma_h(E)$ such that:

\begin{equation}\label{eq:l2.3}
(\bftau_h\bbn\cdot\bbn)_{|e} = - (\bbv_h\cdot\bbn)_{|e}\ \, \qquad
(\bftau_h\bbn\cdot\bbt)_{|e} = - (\bbv_h\cdot\bbt)(\bbm_e) \qquad \forall\, e\in\Eh .
\end{equation}
Due to \eqref{eq:local_stress}, the above choice is admissible. Thus, using also the mid-point rule, we get:

\begin{equation}\label{eq:l2.4}
b_E(\bftau_h,\bbv_h) =
\int_{\partial E} |\bbv_h\cdot\bbn|^2 +
\sum_{e\in\Eh(E)} |e|\, |(\bbv_h\cdot\bbt)(\bbm_e)|^2 .
\end{equation}
Applying to $\bbv_h\cdot\bbn$ the same argument as in \eqref{eq:mp2}-\eqref{eq:mp3}, from \eqref{eq:l2.4} we infer

\begin{equation}\label{eq:l2.5}
b_E(\bftau_h,\bbv_h) \approx
h_E\left(
\sum_{i=1}^{n_E}\left[ |(\bbv_h\cdot\bbn_i^-)(\bbw_i)|^2 +
|(\bbv_h\cdot\bbn_i^+)(\bbw_i)|^2\right]+ \sum_{i=1}^{n_E} |\bbv_h(\bbm_i)|^2
\right)  .
\end{equation}
Above, $\bbn_i^-$ and $\bbn_i^+$ are the two normals to the edges which share $\bbw_i$ as a common vertex. We now notice that, due to assumptions $\mathbf{(A1)}$ and $\mathbf{(A2)}$, there exists a node $\bbw_m$ ($1\le m \le n_E$), for which $ |(\bbv_h\cdot\bbn_m^-)(\bbw_m)|^2 +
|(\bbv_h\cdot\bbn_m^+)(\bbw_m)|^2 \approx |\bbv_h(\bbw_m)|^2$. Therefore, from \eqref{eq:l2.5} we get

\begin{equation}\label{eq:l2.5bis}
b_E(\bftau_h,\bbv_h) \gtrsim
h_E\left(
|\bbv_h(\bbw_m)|^2 + \sum_{\bbm\in\Mh(E)} |\bbv_h(\bbm)|^2
\right)  .
\end{equation}

Applying Lemma \ref{lm:midpoint-norm} we thus infer

\begin{equation}\label{eq:l2.9}
b_E(\bftau_h,\bbv_h) \gtrsim || \bbv_h ||_{0,\partial E}^2 .
\end{equation}
Furthermore, from \eqref{eq:l2.3} we immediately get

\begin{equation}\label{eq:l2.10}
 || \bftau_h\bbn ||_{0,\partial E}\lesssim  || \bbv_h ||_{0,\partial E} .
\end{equation}
Estimate \eqref{eq:inf-sup-loc} is now a consequence of \eqref{eq:l2.9} and \eqref{eq:l2.10}.

\end{proof}

Introducing the local discrete kernel $K_h(\partial E) $ defined by:

\begin{equation}\label{eq:kernelh-loc}
K_h(\partial E) = \Big\{ \bftau_h\bbn_{|\partial E} \ :\ \bftau_h\in\Sigma_h(E)\, ,\ b_E(\bftau_h,\bbv_h)=0\quad \forall \bbv_h\in U_h(E)    \Big\} ,
\end{equation}
Lemma \ref{lm:inf-sup-loc} implies the following result (cf. \cite{Bo-Bre-For}).

\begin{lem}\label{lm:inf-sup-loc2}
	Suppose that  assumptions $\mathbf{(A1)}$ and $\mathbf{(A2)}$ are fulfilled. Then, there exists $\beta>0$ such that
	
	\begin{equation}\label{eq:inf-sup-loc-2}
	\sup_{\bbv_h\in U_h(E)}\frac{b_E( \bftau_h,\bbv_h)}{|| \bbv_h||_{0,\partial E}}\ge \beta\, ||\bftau_h\bbn||_{L^2(\partial E)^{2}/K_h(\partial E)}\qquad \forall\, \bftau_h\in \Sigma_h(E) .
	\end{equation}

\end{lem}

We are now ready to prove the Proposition:

\begin{prop}\label{pr:loal-comp}
Under assumptions $\mathbf{(A1)}$ and $\mathbf{(A2)}$, there exists a linear operator $\pi_E\, :\, \Sigma(E)\to \Sigma_h(E)$ such that:

\begin{equation}\label{eq:com-loc}
\begin{aligned}
&b_E(\pi_E\,\bftau,\bbv_h)= b_E(\bftau,\bbv_h)\qquad \forall\, \bftau\in\Sigma(E)\, , \ \forall\, \bbv_h\in U_h(E),\\
& || (\pi_E\,\bftau)\bbn||_{h,\partial E}\lesssim ||\bftau||_{\Sigma(E)}.
\end{aligned}
\end{equation}

\end{prop}

\begin{proof}
Fix $\bftau\in\Sigma(E)$.	
Due to Lemma \ref{lm:inf-sup-loc}, the linear system in the first line of \eqref{eq:com-loc} is solvable, and two solutions differ up to an element of $K_h(\partial E)$, cf. \eqref{eq:kernelh-loc}. To prove the continuity estimate in \eqref{eq:com-loc}, let us take $\pi_E\bftau\in\Sigma_h(E)$ as the solution which minimizes $||(\pi_E\bftau)\bbn ||_{0,\partial E}$. From \eqref{eq:inf-sup-loc-2} and the first of \eqref{eq:com-loc}, we thus get:

\begin{equation}\label{eq:fortin1}
||(\pi_E\bftau)\bbn ||_{0,\partial E} \lesssim
\sup_{\bbv_h\in U_h(E)}\frac{b_E( \pi_E\bftau,\bbv_h)}{|| \bbv_h||_{0,\partial E}} =
\sup_{\bbv_h\in U_h(E)}\frac{b_E( \bftau,\bbv_h)}{|| \bbv_h||_{0,\partial E}}
\end{equation}
By recalling \eqref{eq:inf-sup-loc2}, a (scaled) duality estimate and a 1D inverse estimate for piecewise linear polynomials, give:

\begin{equation}\label{eq:fortin2}
\begin{aligned}
b_E( \bftau,\bbv_h)= - \int_{\partial E}\bftau\bbn\cdot \bbv_h & \lesssim || \bftau\bbn ||_{-1/2,\partial E}\left( | \bbv_h |_{1/2,\partial E} +
h_E^{-1/2}|| \bbv_h||_{0,\partial E} \right)\\
& \lesssim || \bftau\bbn ||_{-1/2,\partial E}\, h_E^{-1/2}|| \bbv_h||_{0,\partial E} .
\end{aligned}
\end{equation}
Therefore, from \eqref{eq:fortin1} and \eqref{eq:fortin2} we obtain

\begin{equation}\label{eq:fortin3}
h_E^{1/2}||(\pi_E\bftau)\bbn ||_{0,\partial E} \lesssim
|| \bftau\bbn ||_{-1/2,\partial E} .
\end{equation}
The continuity estimate in \eqref{eq:com-loc} now follows from a trace estimate and definition \eqref{eq:meshdep}.

\end{proof}

We notice that $\bftau\in\Sigma^0(E)$ implies $\pi_E\bftau\in\Sigma^0_h(E)$, cf. \eqref{eq:cont-loc_spaces} and \eqref{eq:affine-sigma}. Indeed, by definition \eqref{eq:local_stress}, $\bdiv(\pi_E\bftau)\in RM(E)$. Since $RM(E)\subseteq U_h(E)$, we can take
$\bbv_h = \bdiv(\pi_E\bftau)$ in \eqref{eq:com-loc}, to obtain (using also the integration by parts):

\begin{equation}\label{eq:fortin4}
\int_E |\bdiv(\pi_E\bftau)|^2 = -
b_E(\pi_E\bftau, \bdiv(\pi_E\bftau)) =
-b_E(\bftau, \bdiv(\pi_E\bftau))=
\int_E \bdiv\bftau\cdot\bdiv(\pi_E\bftau) =0 .
\end{equation}
This observation, together with Proposition \ref{pr:loal-comp}, immediately leads to the following result.

\begin{cor}\label{cor:local-comp}
	Under assumptions $\mathbf{(A1)}$ and $\mathbf{(A2)}$, there exists a linear operator $\pi_E\, :\, \Sigma^0(E)\to \Sigma_h^0(E)$ such that:
	
	\begin{equation}\label{eq:com-loc-bis}
	\begin{aligned}
	&b_E(\pi_E\,\bftau,\bbv_h)= b_E(\bftau,\bbv_h)\qquad \forall\, \bftau\in\Sigma^0(E)\, , \ \forall\, \bbv_h\in U_h(E),\\
	& || (\pi_E\,\bftau)\bbn||_{h,\partial E}\lesssim ||\bftau||_{\Sigma(E)}.
	\end{aligned}
	\end{equation}
	
\end{cor}

Recalling \eqref{eq:cont-spaces}, we define the linear operator $\pi_h\,:\,\Sigma^0\to \Sigma_h^0$ by adding the local contributions $\pi_E$, i.e.:

\begin{equation}\label{eq:fortin-glob}
\pi_{h|E} : = \pi_E\qquad \forall\, E\in\Th .
\end{equation}
Obviously, Corollary \ref{cor:local-comp} and \eqref{eq:cont-glob_forms} give:

\begin{prop}\label{pr:fortin-glob}
	Under assumptions $\mathbf{(A1)}$ and $\mathbf{(A2)}$, there exists a linear operator $\pi_h\, :\, \Sigma^0\to \Sigma_h^0$ such that:
	
	\begin{equation}\label{eq:com-glob}
	\begin{aligned}
	&b(\pi_h\,\bftau,\bbv_h)= b(\bftau,\bbv_h)\qquad \forall\, \bftau\in\Sigma^0\, , \ \forall\, \bbv_h\in U_h^0,\\
	& || \pi_h\,\bftau||_{h}\lesssim ||\bftau||_{\Sigma}.
	\end{aligned}
	\end{equation}
\end{prop}
With Proposition \ref{pr:fortin-glob} at hand, the following {\em inf-sup} condition is easily proved (it is nothing but Fortin's trick, see \cite{Bo-Bre-For}).

\begin{prop}\label{pr:fortin-glob2}
	Under assumptions $\mathbf{(A1)}$ and $\mathbf{(A2)}$, we have:
	
	\begin{equation}\label{eq:com-glob2}
	  \sup_{\bftau_h\in \Sigma_h^0}\frac{b( \bftau_h,\bbv_h)}{|| \bftau_h||_h }  \gtrsim ||\bbv_h||_{U^0/H} \qquad \forall\, \bbv_h\in U_h^0/H .
	\end{equation}
\end{prop}

\begin{proof}
Fix $\bbv_h\in U^0_h$.
Using \eqref{eq:com-glob} and the {\em inf-sup} condition for the continuous problem (see \eqref{eq:cont-inf-sup}), we infer

\begin{equation}\label{eq:fortin-glob2}
\sup_{\bftau_h\in \Sigma_h^0}\frac{b( \bftau_h,\bbv_h)}{|| \bftau_h||_h } \geq
\sup_{\bftau\in \Sigma^0}\frac{b( \pi_h\bftau,\bbv_h)}{|| \pi_h\bftau||_h }=
\sup_{\bftau\in \Sigma^0}\frac{b( \bftau,\bbv_h)}{|| \pi_h\bftau||_h }\gtrsim
\sup_{\bftau\in \Sigma^0}\frac{b( \bftau,\bbv_h)}{|| \bftau||_\Sigma }\gtrsim ||\bbv_h||_{U^0/H} .
\end{equation}
	
\end{proof}

\subsection{An interpolation operator for the stresses}\label{ss:interpol-oper}

We now recall the interpolation operator introduced in \cite{ADLP_HR}.
We first set, given $r > 2$:

\begin{equation}\label{eq:regspace}
W^r(E):=\left\{  \bftau  \ : \bftau\in L^r(E)^{2\times 2} \ , \  \bftau=\bftau^T\ , \ \bdiv\bftau\in L^2(E)^2    \right\} .
\end{equation}
To continue, we locally define the operator $\IE : W^r(E)\to  \Sigma_h(E)$ as follows. Given $\bftau\in W^r(E)$, $\IE\bftau\in \Sigma_h(E)$ is determined by:

\begin{equation}\label{eq:loc-interp_0}
\int_{\partial E} (\IE \bftau) \bbn\cdot \bfvarphi_\ast = \int_{\partial E}  \bftau\bbn\cdot \bfvarphi_\ast \qquad \forall \bfvarphi_\ast\in  R_\ast(\partial E) ,
\end{equation}
where

\begin{equation}\label{eq:Rast}
R_\ast(\partial E) = \left\{
\bfvarphi_\ast\in L^2(\partial E)^2 \,: \,
\bfvarphi_{\ast | e} = \bfgamma_e + \delta_e (\bbx -\bbx_C)^\perp \quad \bfgamma_e\in\R^2, \  \delta_e\in\R, \ \forall e\in\partial E  \right\}.
\end{equation}
If $\bftau$ is not sufficiently regular, the integral in the right-hsnd side of \eqref{eq:loc-interp_0} is intended as a duality between $W^{-\frac{1}{r},r}(\partial E)^2$ and $W^{\frac{1}{r},r'}(\partial E)^2$. If $\bftau$ is a regular function, the above condition is equivalent to require:

\begin{equation}\label{eq:loc-interp}
\left\lbrace{
	\begin{aligned}
	&\int_e (\IE \bftau) \bbn = \int_e  \bftau\bbn \qquad \forall e\in \partial E ;\\
	&\int_e (\IE \bftau) \bbn\cdot (\bbx - \bbx_C)^\perp = \int_e  \bftau\bbn \cdot (\bbx - \bbx_C)^\perp \qquad \forall e\in \partial E .
	\end{aligned}
} \right.
\end{equation}

%
%
%
	

The global interpolation operator $\Ih : W^r(\O)\to \Sigma_h$ is then defined by simply glueing the local contributions provided by $\IE$. More precisely, we set $(\Ih\tau)_{|E} :=\IE\bftau_{|E}$ for every $E\in\Th$ and $\bftau\in W^r(\O)$. It can be proved, see \cite{ADLP_HR}, that the commuting diagram property:

\begin{equation}\label{eq:commuting}
\bdiv(\Ih\bftau) = \Pi_{RM}(\bdiv\bftau)
\end{equation}
holds true, $\Pi_{RM}$ being the $L^2$-projection operator onto the space of local rigid body motions.
Furthermore, the following estimates have been proved in \cite{ADLP_HR}.

\begin{prop}\label{pr:approxest} Under assumptions $\mathbf{(A1)}$ and $\mathbf{(A2)}$,
	for the interpolation operator $\IE$ defined in \eqref{eq:loc-interp}, the following estimates hold:
	
	\begin{equation}\label{eq:l2est}
	|| \bftau -\IE\bftau||_{0,E}\lesssim h_E |\bftau|_{1,E} \qquad \forall \bftau\in  \widetilde\Sigma(E) \cap H^1(E)^4_s .
	\end{equation}
	
	\begin{equation}\label{eq:divest}
	\begin{aligned}
	|| \bdiv(\bftau -\IE\bftau)||_{0,E}  \lesssim h_E |\bdiv\bftau|_{1,E}   \ \
	\forall \bftau\in \widetilde\Sigma(E) \cap H^1(E)^4_s
	\mbox{ \rm s.t.  $\bdiv\bftau\in H^1(E)^2$}.
	\end{aligned}
	\end{equation}	
	
\end{prop}



\subsection{Error estimates}\label{ss:errest}


We denote with $\P_0(\Th)$ the space of piecewise constant functions with respect to the given mesh $\Th$.

Once one has the stability conditions of estimate \eqref{eq:ellker-nd} and Proposition \ref{pr:fortin-glob2}, along with the interpolation estimates of Proposition \ref{pr:approxest}, an error analysis can be derived using the techniques of \cite{BFM} or \cite{ADLP_HR}. Indeed, one can prove:

\begin{prop}\label{pr:error-est} Suppose that  assumptions $\mathbf{(A1)}$ and $\mathbf{(A2)}$ are fulfilled. The following error estimate holds:
	
	\begin{equation}\label{eq:erroreq}
	|| \bfsigma- \bfsigma_h||_{0,\O} + || \bbu - \bbu_h ||_{U^0/H} \leq C(\Omega,\bfsigma,\bbu)\, h ,
\end{equation}
where $(\bfsigma,\bbu)\in \Sigma^f\times U^0/H$ is the solution to the continuous Problem \ref{cont-pbl-v2}, and $(\bfsigma_h,\bbu_h)$ is such that $\bfsigma_h=\bfsigma^0_h + \widehat\bfsigma_{h,f}$, $(\bfsigma_h,\bbu_h)\in \Sigma^f_h\times U^0_h/H$ being the solution to the discrete problem \ref{eq:discr-pbl-ls}.
Furthermore, $C(\Omega,\bfsigma,\bbu)$ is independent of $h$ but depends on the domain $\Omega$ and on the Sobolev regularity of $\bfsigma$ and $\bbu$.
%
\end{prop}

\begin{proof}
	We take $\bfsigma_I\in \Sigma_h^f$ defined as $\bfsigma_I=\Ih\bfsigma$, see section \ref{ss:interpol-oper}. Due to the splitting $\bfsigma=\bfsigma^0+\widehat{\bfsigma}_f$, with $\bfsigma^0\in\Sigma^0$, we have
	
	\begin{equation}\label{eq:splitting}
	\bfsigma_I=\Ih\bfsigma= \Ih \bfsigma^0+\Ih\widehat{\bfsigma}_f := \bfsigma^0_I + \widehat{\bfsigma}_{I,f} .
	\end{equation}
	From \eqref{eq:commuting}, we get $\bfsigma^0_I\in\Sigma^0_h$ and $\widehat{\bfsigma}_{I,f}\in \Sigma^f_h$. We also take $\bbu_I\in U^0_h$ as the usual VEM interpolant of $\bbu$, see for example \cite{volley} or \cite{BLRXX}.

	We now form $(\bfsigma_h - \bfsigma_I,\bbu_h - \bbu_I)\in\Sigma_h^0\times U_h^0$. Then, using the
	{\em ellipticity-on-the-kernel} condition of estimate \eqref{eq:ellker-nd} and the {\em inf-sup} condition of Proposition \ref{pr:fortin-glob2}, there exists $(\bftau_h, \bbv_h)\in\Sigma_h^0\times U_h^0$ such that (see \cite{Bo-Bre-For} and \cite{Braess:book}, for instance):
	
	\begin{equation}\label{eq:costest}
	|| \bftau_h||_h + || \bbv_h ||_{U^0/H}
	\lesssim 1
	\end{equation}
	and
	
	\begin{equation}\label{eq:stabest}
	||\bfsigma_h - \bfsigma_I||_h + || \bbu_h - \bbu_I ||_{U^0/H}
	\lesssim \A_h (\bfsigma_h - \bfsigma_I,\bbu_h - \bbu_I;\bftau_h,\bbv_h) .
	\end{equation}
	
	We have, considering the splitting $\bfsigma=\bfsigma^0+\widehat{\bfsigma}_f$ and \eqref{eq:splitting}, and using both \eqref{cont-pbl} and \eqref{discr-pbl-cpt}, together with \eqref{eq:magic-rhs}:
	
	\begin{equation}\label{eq:stabest2}
	\begin{aligned}
	\A_h & (\bfsigma_h  - \bfsigma_I,\bbu_h - \bbu_I ;\bftau_h,\bbv_h)=
	\A_h (\bfsigma_h ,\bbu_h ;\bftau_h,\bbv_h) - \A_h ( \bfsigma_I,  \bbu_I;\bftau_h,\bbv_h) \\
	& = \left[a(\bfsigma,\bftau_h) - a_h(\bfsigma_I,\bftau_h)\right]
	+ b(\bftau_h, \bbu-\bbu_I) +b( \bfsigma-\bfsigma_I, \bbv_h) + [F_h(\bftau_h) - F(\bftau_h)] \\
	&= T_1+T_2+T_3+ T_4 .
	\end{aligned}
	\end{equation}
	
	For both the choices \eqref{eq:pi0} and \eqref{eq:pi1}, the term $T_1$ can be treated using the techniques of \cite{ADLP_HR}, to obtain:

	\begin{equation}\label{eq:stabest6}
	T_1\lesssim \left(   ||\bfsigma- \bfsigma_I||_{0,\O}+ ||\bfsigma - \bfsigma_\pi)||_{0,\O} +h\,||\bdiv\bfsigma_I||_{0,\O}
	\right)  ||\bftau_h||_h ,
	\end{equation}
	where $\bfsigma_\pi$ is the $L^2$-projection of $\bfsigma$ onto $\P_0(\Th)^{2\times 2}_s$.
	
	Regarding $T_2$, using the Agmon's trace inequality (see for example \cite{Agmon}), one has:
	
	\begin{equation}\label{eq:stabest7}
		\begin{aligned}	
		T_2 = - \sum_{E\in\Th}
		\int_{\partial E}\bftau_h\bbn&\cdot(\bbu-\bbu_I)\lesssim \sum_{E\in\Th} h_E^{1/2} || \bftau_h||_{0,\partial E} h_E^{-1/2} || \bbu-\bbu_I||_{0,\partial E}\\
		& \lesssim \sum_{E\in\Th} h_E^{1/2} || \bftau_h||_{0,\partial E} \left( h_E|| \bbu-\bbu_I||_{0,E} + | \bbu-\bbu_I|_{1,E} \right)\\
		&\lesssim || \bbu-\bbu_I||_{U^0}\, || \bftau_h||_{h} .
		\end{aligned}
	\end{equation}

	Term $T_3$ can be treated using standard trace inequalities, to obtain:
	
	\begin{equation}\label{eq:stabest7bis}
	T_3 \lesssim || \bfsigma-\bfsigma_I||_{H(\bdiv)}\, || \bbv_h||_{U^0} .
	\end{equation}

To estimate the term $T_4$ we first recall that, see \eqref{eq:magic-rhs}:	

\begin{equation}\label{eq:magic-rhs2}
F_h(\bftau_h) - F(\bftau_h) =
\sum_{E\in\Th}a_E(\widehat{\bfsigma}_{f}- \Pi_E \widehat{\bfsigma}_{f},\bftau_h) .
\end{equation}	
If $\Pi_E$ is selected according with \eqref{eq:pi1}, then $\Pi_E \widehat{\bfsigma}_{f}= \widehat{\bfsigma}_{f}$, and $T_4$ vanishes. If $\Pi_E$ is selected according with \eqref{eq:pi0}, we have $\Pi_E \widehat{\bfsigma}_{f}=0$ (cf. \eqref{eq:part-const}). Then we get, also using \eqref{eq:normequiv}:

\begin{equation}\label{eq:t4}
\begin{aligned}
T_4 = F_h(\bftau_h) - F(\bftau_h)
&\lesssim \sum_{E\in\Th}a_E(\widehat{\bfsigma}_{f},\bftau_h)\lesssim
\left( \sum_{E\in\Th} || \widehat{\bfsigma}_{f} ||_{0,E}^2 \right)^{1/2}
\left( \sum_{E\in\Th}||\bftau_h ||_{0,E}^2  \right)^{1/2}\\
& \lesssim
\left( \sum_{E\in\Th} || \widehat{\bfsigma}_{f} ||_{0,E}^2 \right)^{1/2} || \bftau_h||_{h} .
\end{aligned}
\end{equation}	
A direct computation taking into account \eqref{eq:part-const} gives $|| \widehat{\bfsigma}_{f} ||_{0,E}\lesssim h_E || \bbf ||_{0,E}$. Therefore, we obtain

\begin{equation}\label{eq:t4_bis}
T_4  \lesssim
\left( \sum_{E\in\Th} h_E^2 || \bbf ||_{0,E}^2 \right)^{1/2} || \bftau_h||_{h} .
\end{equation}

	From \eqref{eq:stabest}, \eqref{eq:stabest2}, \eqref{eq:stabest6}, \eqref{eq:stabest7}, \eqref{eq:stabest7bis} and \eqref{eq:t4_bis}, we get:
	
	\begin{equation}\label{eq:stabest8}
	\begin{aligned}
	||\bfsigma_h - \bfsigma_I||_h & + || \bbu_h - \bbu_I ||_{U^0/H}
	\lesssim  \Big( ||\bfsigma- \bfsigma_I||_{H(\bdiv)}+ ||\bfsigma - \bfsigma_\pi)||_{0,\O} \\
	&+ h\,||\bdiv\bfsigma_I||_{0,\O}+||\bbu -\bbu_I||_{U^0} + h\, ||\bbf||_{0,\O} \Big)
	\left(  || \bftau_h||_h + || \bbv_h ||_{U^0} \right) .
	\end{aligned}
	\end{equation}

Using \eqref{eq:costest}, standard approximation results and estimates \eqref{eq:l2est}-\eqref{eq:divest}, we infer:
	
	\begin{equation}\label{eq:conv-est}
	||\bfsigma_h - \bfsigma_I||_h  + || \bbu_h - \bbu_I ||_{U^0/H}
	 \leq C(\Omega,\bfsigma,\bbu)\, h ,
	\end{equation}
	where $C(\Omega,\bfsigma,\bbu)$ is independent of $h$ but depends on the domain $\Omega$ and on the Sobolev regularity of $\bfsigma$ and $\bbu$. We now use the triangle inequality and the estimate
	
	$$
	||\bftau_h||_{0,\O}\lesssim || \bftau_h||_h \qquad \forall \bftau_h\in\Sigma_h .
	$$
	
	Exploiting \eqref{eq:conv-est}, again standard approximation results and \eqref{eq:l2est}, we thus obtain:
	
	\begin{equation}\label{eq:stabest9}
	\begin{aligned}
	||\bfsigma - \bfsigma_h||_{0,\O}  + || \bbu - \bbu_h ||_{U^0/H}
	&\leq  ||\bfsigma - \bfsigma_I||_{0,\O}  + || \bbu - \bbu_I ||_U +  ||\bfsigma_I - \bfsigma_h||_{0,\O}  + || \bbu_I - \bbu_h ||_U \\
	&\lesssim
	||\bfsigma - \bfsigma_I||_{0,\O}  + || \bbu - \bbu_I ||_U +  ||\bfsigma_I - \bfsigma_h||_h  + || \bbu_I - \bbu_h ||_U \\
	&\leq C(\Omega,\bfsigma,\bbu)\, h .
	\end{aligned}
	\end{equation}

\end{proof}

%
%
%

\section{Numerical results}\label{s:numer}
The present section is devoted to validation of the proposed dual hybrid methods. First, in section \ref{ss:convacc} convergence and accuracy are numerically assessed on a couple of benchmarks having a closed-form solution. Subsequently, in section \ref{ss:fold}}, an elastic problem stemming from a simple electromechanical application is considered, proving applicability of the method to the analysis and simulation of real structures. In all presented tests reference is made for comparison to the standard displacement-based linear Virtual Element Method detailed in \cite{ABLS_part_I}.
\subsection{Convergence and accuracy assessment}
\label{ss:convacc}
A set of two boundary value problems on the unit square domain $\Omega = [0, 1]^2$ is considered, for which an analytical solution is available \cite{BeiraoLovaMora}.
Material parameters are assigned in terms of Lam\'e constants $\lambda = 1$, $\mu = 1$, assuming plane strain and homogeneous isotropy conditions. The tests are defined by choosing a required solution and deriving the corresponding body force $\bbf$, as reported in the following:
\begin{itemize}
	\item[$\bullet$]
	Test $a$
	\begin{eqnarray}
		\label{eq:test_2a_sol}
		\left\{
		\begin{array}{l}
			u_1 = x^3 - 3 x y^2 \\
			u_2 = y^3 - 3 x^2 y \\
			\bbf = \bfzero
		\end{array}
		\right .
	\end{eqnarray}
	\item[$\bullet$]
	Test $b$
	\begin{eqnarray}
		\label{eq:test_2b_sol}
		\left\{
		\begin{array}{l}
			u_1 = u_2 = \sin(\pi x)  \sin(\pi y) \\
			f_1 = f_2 = -\pi^2 \left[ -(3 \mu + \lambda ) \sin(\pi x) \sin ( \pi y) + ( \mu + \lambda ) \cos ( \pi x) \cos ( \pi y ) \right]
		\end{array}
		\right .
	\end{eqnarray}
\end{itemize}
Test $a$ has Dirichlet non-homogeneous boundary conditions, zero body force and a polynomial solution; Test $b$ has Dirichlet homogeneous boundary conditions, trigonometric body force and a trigonometric solution.

The tests are run on two sets of structured [resp. unstructured] square, hexagon, and concave quadrilateral [resp. triangle, quadrilateral, and Voronoi] simple polygonal meshes, each type being represented and labeled in Fig. \ref{fig:mesh}, for a sequence of five uniform mesh refinements.

Numerical solutions for the above tests are sought with the proposed dual hybrid method, with the two versions of projection operator Eq. \eqref{eq:pi0} [resp. Eq. \eqref{eq:pi1}], which will be labeled DH P$0$ [resp. DH P$1$].
For comparison purposes, a further numerical solution with the linear displacement based VEM presented in \cite{ABLS_part_I} is computed and labeled $\rm DISP$.

Convergence rate and accuracy level are investigated computing the following error quantities:
\begin{itemize}
	\item[$\bullet$] Discrete relative error quantity for the stress field:
	\begin{equation}
	\label{eq:stress_err_norm}
	E_{\bfsigma}   :=\left( \frac{\sum_{E \in \Th} \int_{E} \,|| \bfsigma_h - \bfsigma ||^2}
	                             {\sum_{E \in \Th} \int_{E} \,||              \bfsigma ||^2}  \right)^{1/2} .
	\end{equation}
	\item[$\bullet$] Discrete relative weighted error quantity for the inter-element traction field:
	\begin{equation}
	\label{eq:traction_err_norm}
	E_{\bbt_{\bbn}}   :=\left( \frac{\sum_{e \in \Eh} |e| \int_{e} \,|| \bbt_{\bbn,h} - \bbt_{\bbn} ||^2}
	                                {\sum_{e \in \Eh} |e| \int_{e} \,||                 \bbt_{\bbn} ||^2}  \right)^{1/2} .
	\end{equation}
    where $\bbn$ is one outward unit normal to the edge $e$	chosen once and for all. Quantities $\bbt_{\bbn,h}$ are the average of the two contributions stemming from the two elements adjacent to edge $e$.
	\item[$\bullet$] Discrete $H^1$-type error quantity for the inter-element displacement field:
	\begin{equation}
	\label{eq:displ_err_norm}
	E_{{\bbu}} :=\left( \sum_{e \in \Eh} |e| \int_{e} \, \left|\left| \frac{\partial \bbu_h}{\partial \bbe} - \frac{\partial \bbu}{\partial \bbe} \right|\right|^2 \right)^{1/2} .
	\end{equation}
	where $\bbe$ is the unit tangent to the edge $e$ chosen once and for all.
\end{itemize}

Inspection of Figs. \ref{fig:Test_a_E_sig}-\ref{fig:Test_b_E_tn} confirms expected convergence rates for the three compared methods.
In terms of accuracy, for both tests and for all mesh types, dual hybrid virtual element methods DH P0 and DH P1 outperform displacement based virtual element method DISP, for stress and inter-element traction field, respectively. Comparatively, DH P1 shows the highest edge on Test b (cf. also Remark \ref{rm:twoproj}). It is noted that the three stress [resp. traction] fields coincide for Test a adopting triangles. Figs. \ref{fig:Test_a_E_u}-\ref{fig:Test_b_E_u} show that DH P1 and DISP methods are fairly comparable in terms of displacement field accuracy, for both tests and for all mesh types, with some selected cases in which DISP presents the lowest error levels while DH P0 the highest ones. It is noted that the three displacement fields coincide for Test a adopting structured quadrilaterals, and for both tests adopting triangles.

\subsection{Structural application: folded-beam suspension}
\label{ss:fold}
A representative section of a typical microelectromechanical system (MEMS) \cite{Patil_2014}, consisting of two bulky portions connected with four slender beams (see Fig. \ref{fig:comb_geom}), is considered, as a structural application on which we test the VEM methods described in the previous section. As an electromechanical plate-like device, a two-dimensional linear elastic analysis under plane stress assumption is carried out. Geometry, dimensions, boundary conditions and loading are represented in Fig. \ref{fig:comb_geom}. Material parameters are $E = 60$ GPa, $\nu = 0.22$; edge traction $q = 10^{-2}$ N/m. The relevant half domain is meshed with triangles, quadrilaterals, and Voronoi polygons as previously done. The latter spatial discretization makes use of non-convex polygons in the zones surrounding re-entering corners, which, given the particular geometry under consideration, further enlights the broader mesh capability offered by VEM methods in respect with standard FEM discretizations.

Progressively finer meshes are considered for DISP, DH P0, DH P1 method, respectively, while a reference solution is computed with CPE4H hybrid element implemented in COMSOL \cite{Comsol:2008} on a very fine mesh. Results in terms of relative error on the vertical displacement of target node $A$ (cf. Fig. \ref{fig:comb_geom}) against total number of degrees of freedom are shown in Fig. \ref{fig:comb_rslt}, which confirms the efficiency of the proposed dual hybrid method as a tool for structural analysis.
\clearpage
\newpage

\section{Conclusions}\label{s:conclusions}

We have presented a Virtual Element Method for 2D linear elastic problems, based on a dual hybrid variational formulation. The scheme offers two variants, which differ from each other according to the choice of the VEM stress projection. A stability and convergence analysis has been developed, and several numerical tests have been performed, confirming the theoretical predicitions. Our study shows that dual hybrid VEMs represents a valid alternative to standard displacement-based VEM schemes, especially if one is interested in an accurate description of the stress field. We finally remark that hybrid elasticity methods have been used, in the FEM framework, to tackle Structural Mechanics problems (e.g. plate problems, see \cite{deMirandaUbertini}): our VEM scheme might be fruitfully employed within that context as well.

\medskip

\begin{center}
	{\large {\bf Aknowledgements}}
\end{center}
E.A. gratefully acknowledges the partial financial support of the University of Rome Tor Vergata Mission Sustainability
Programme through project SPY-E81I18000540005.
\medskip

\begin{figure}[h!]
	\centering
	\renewcommand{\thesubfigure}{}
	\subfigure[QuadS]{\includegraphics[bb = 120 290 600 600,clip,scale=0.3]{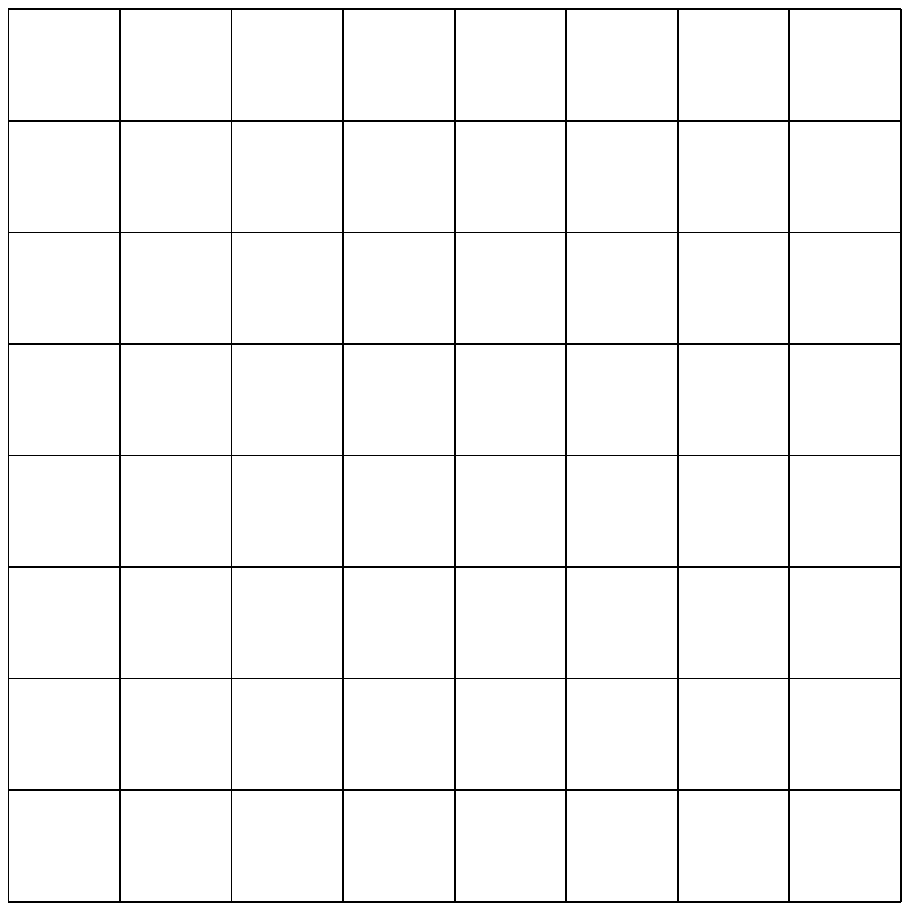}}
	\subfigure[HexS]{\includegraphics[bb = 170 290 600 600,clip,scale=0.3]{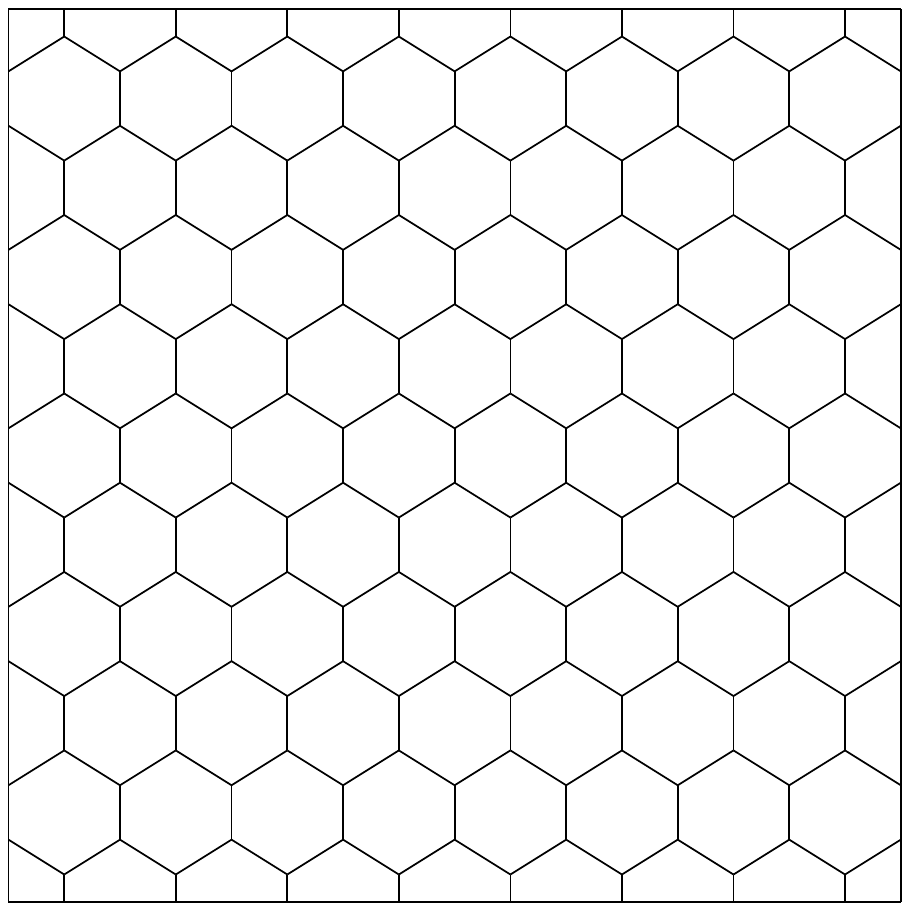}}
	\subfigure[ConcS]{\includegraphics[bb = 170 290 600 600,clip,scale=0.3]{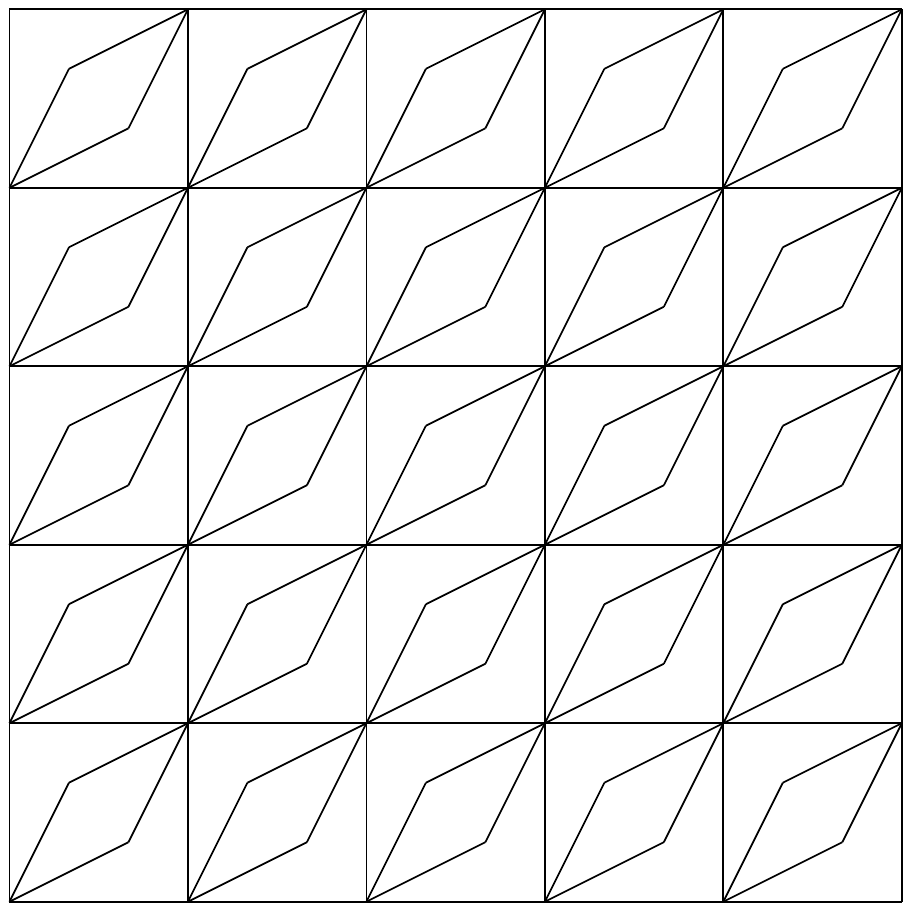}}
	\subfigure[TriU]{\includegraphics[bb = 120 290 600 600,clip,scale=0.3]{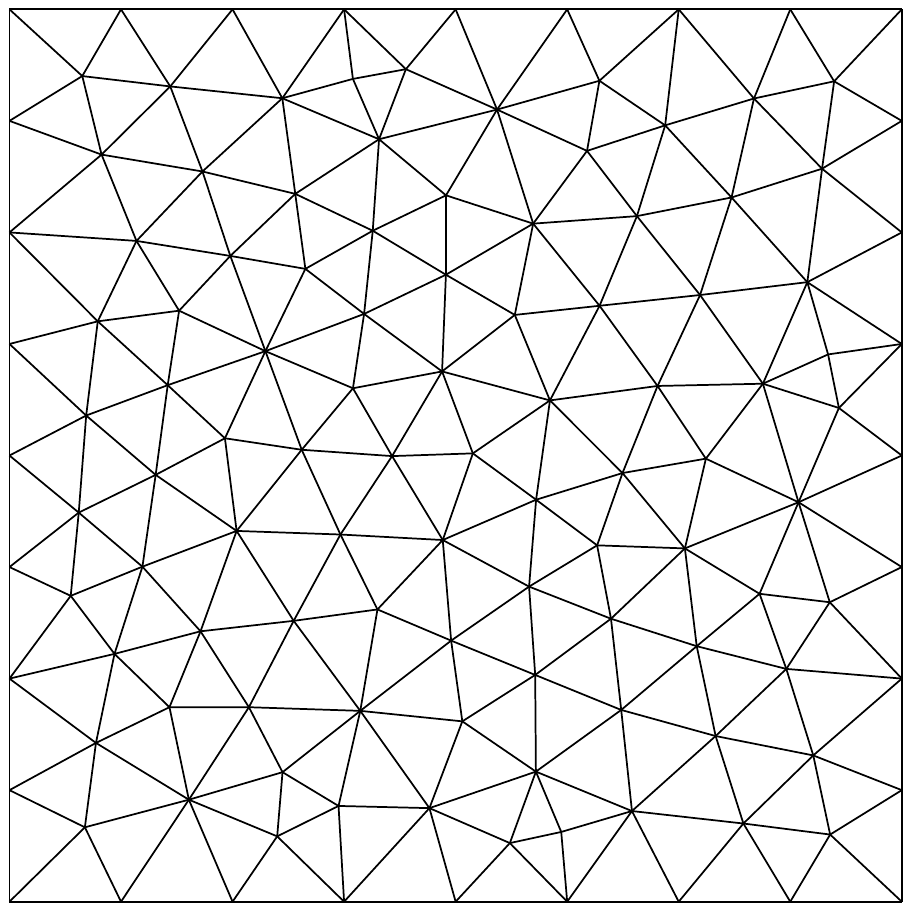}}
	\subfigure[QuadU]{\includegraphics[bb = 170 290 600 600,clip,scale=0.3]{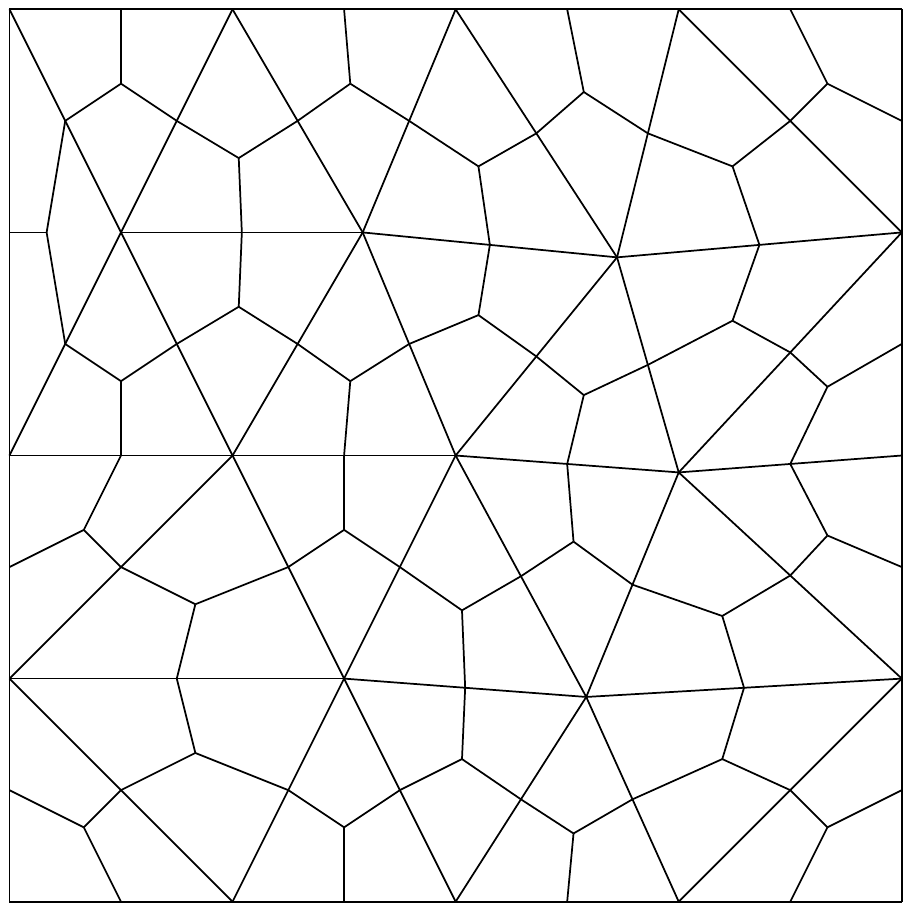}}
	\subfigure[voroU]{\includegraphics[bb = 170 290 600 600,clip,scale=0.3]{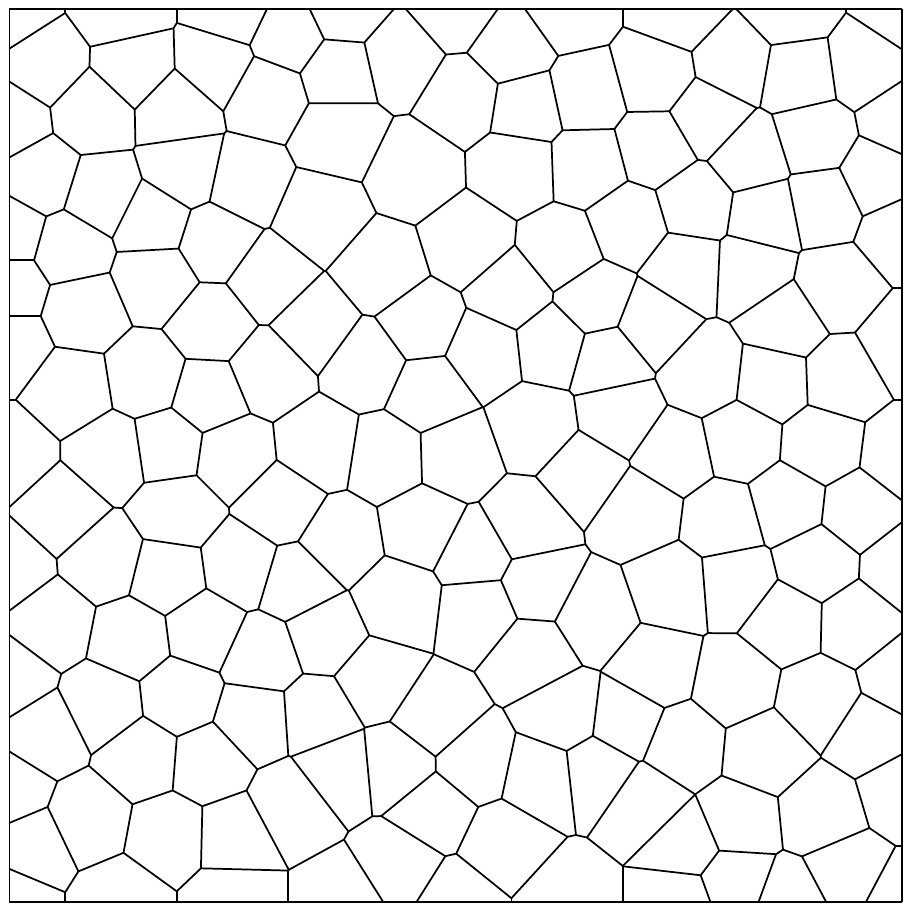}}
	\caption{Mesh types and labels for Test a and Test b. Upper row - Structured: Quad/Hexagon/Concave polygons.
Lower row - Unstructured: Tri/Quad/Voronoi polygons.}
	\label{fig:mesh}
\end{figure}

\clearpage
\newpage
\begin{figure}
\subfigure
{\includegraphics[bb = 135 260 550 556,clip,scale=0.55]
                {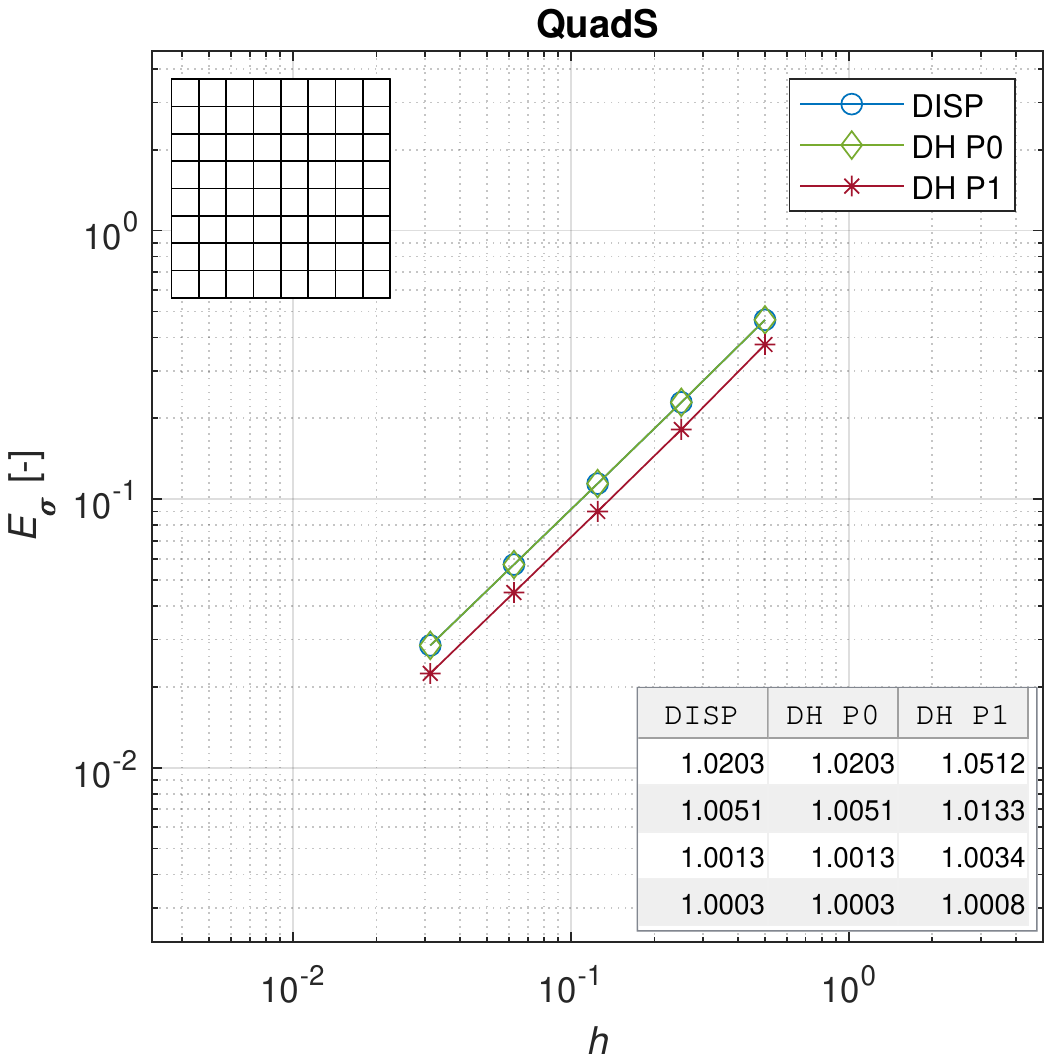}}
\subfigure
{\includegraphics[bb = 150 260 600 556,clip,scale=0.55]
                {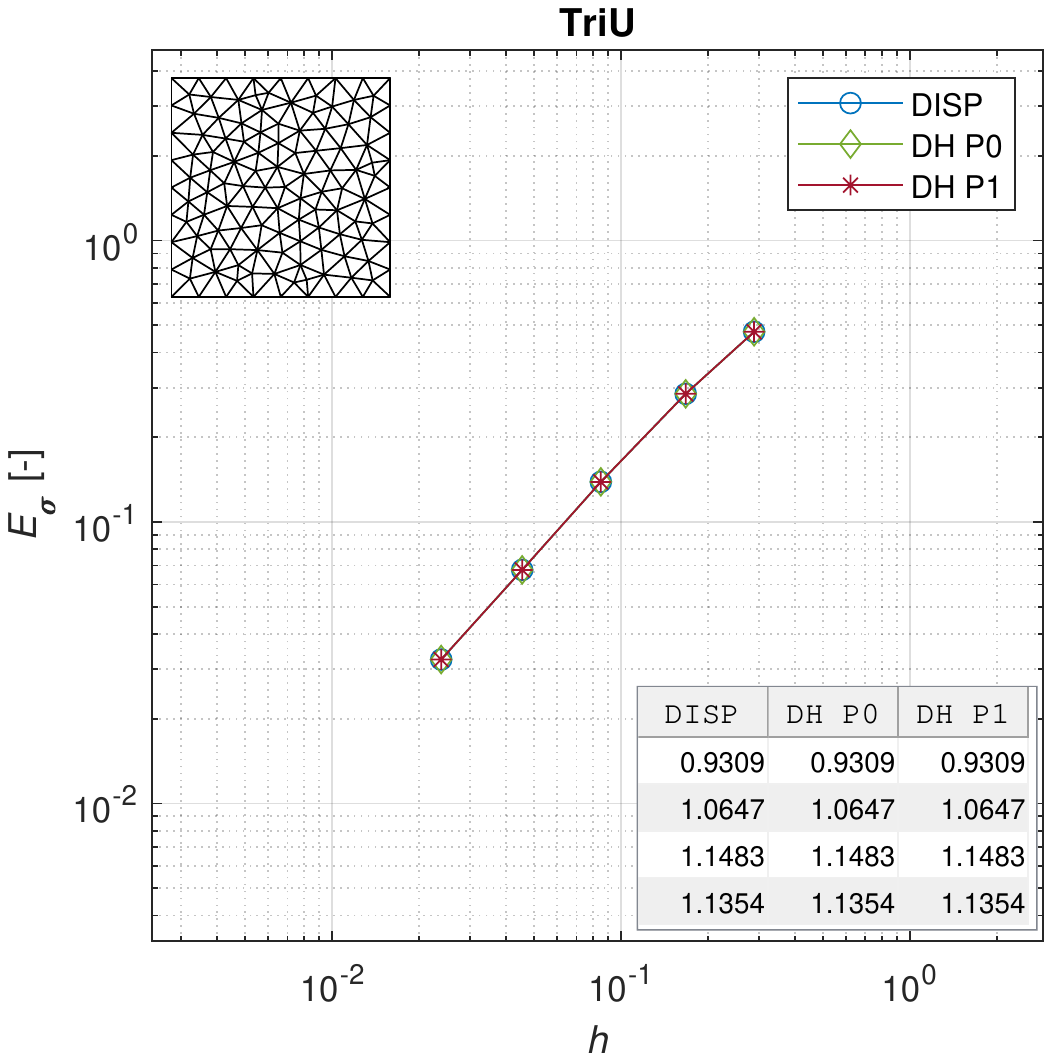}}
\subfigure
{\includegraphics[bb = 135 260 550 556,clip,scale=0.55]
                {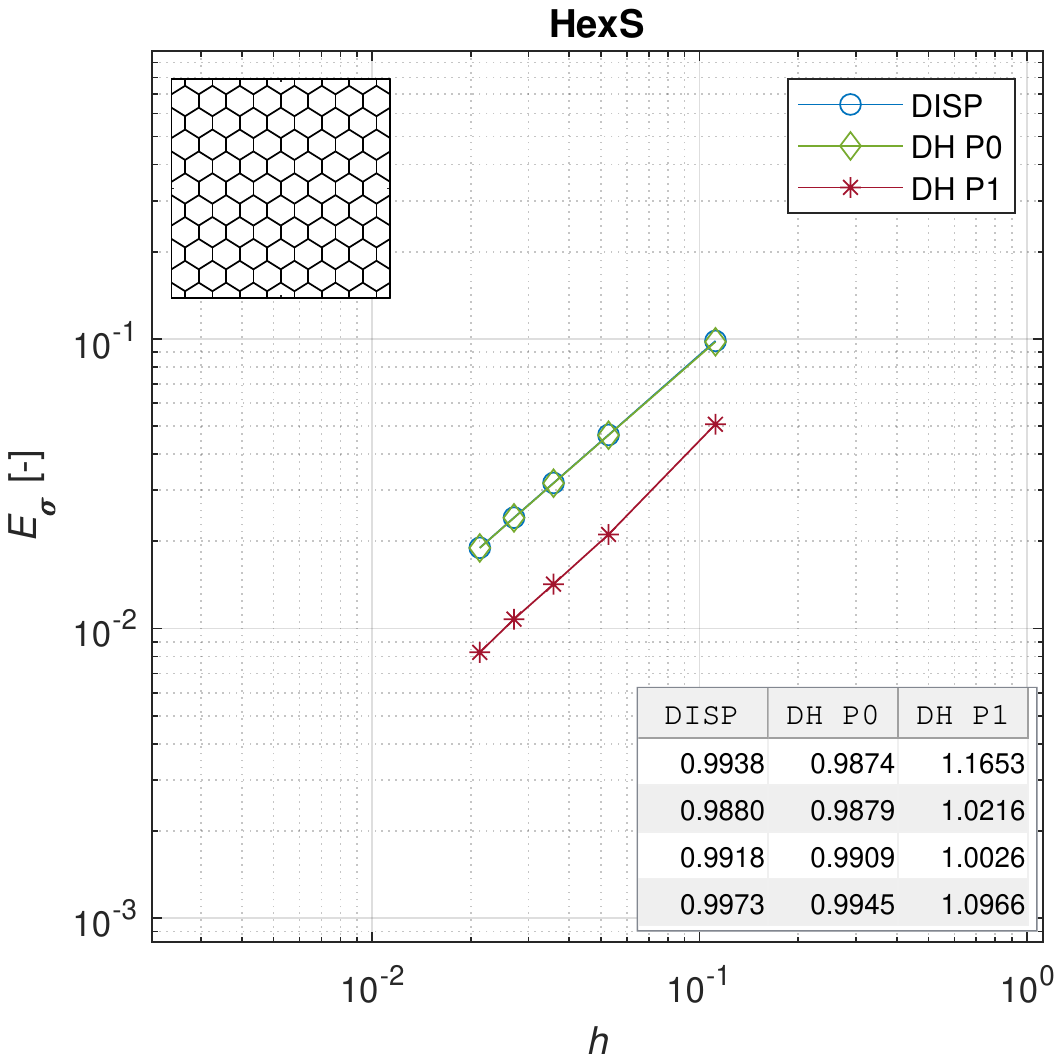}}
\subfigure
{\includegraphics[bb = 150 260 600 556,clip,scale=0.55]
                {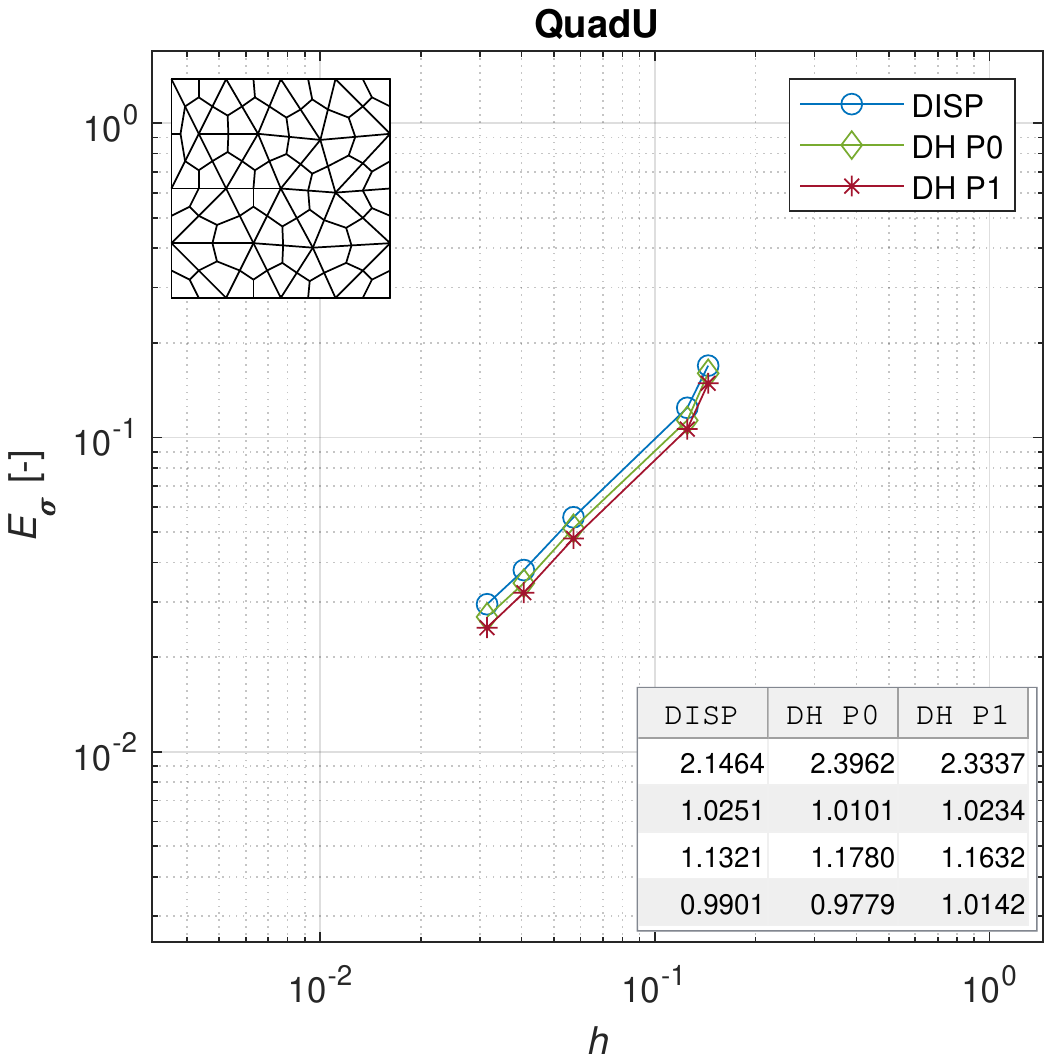}}
\subfigure
{\includegraphics[bb = 135 260 550 556,clip,scale=0.55]
                {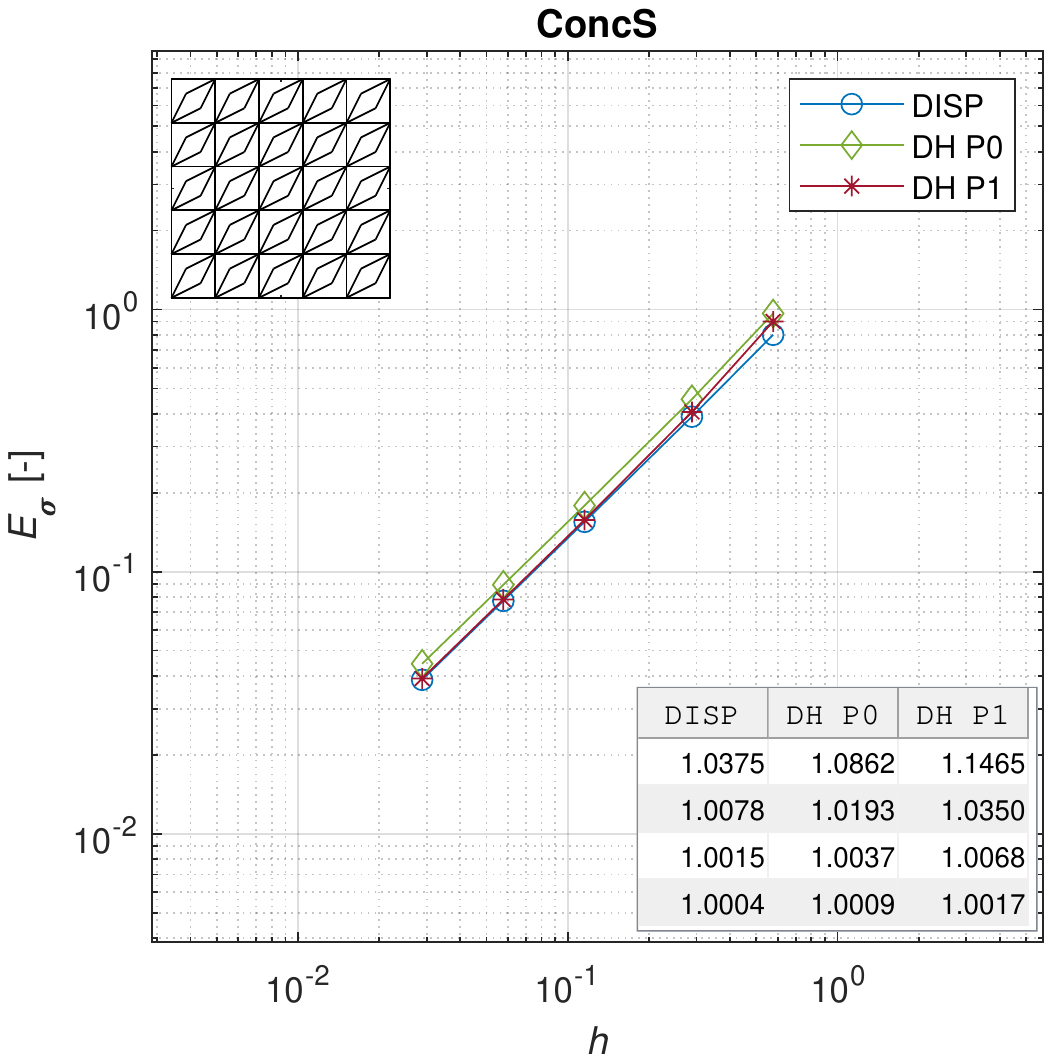}}
\subfigure
{\includegraphics[bb = 150 260 600 556,clip,scale=0.55]
                {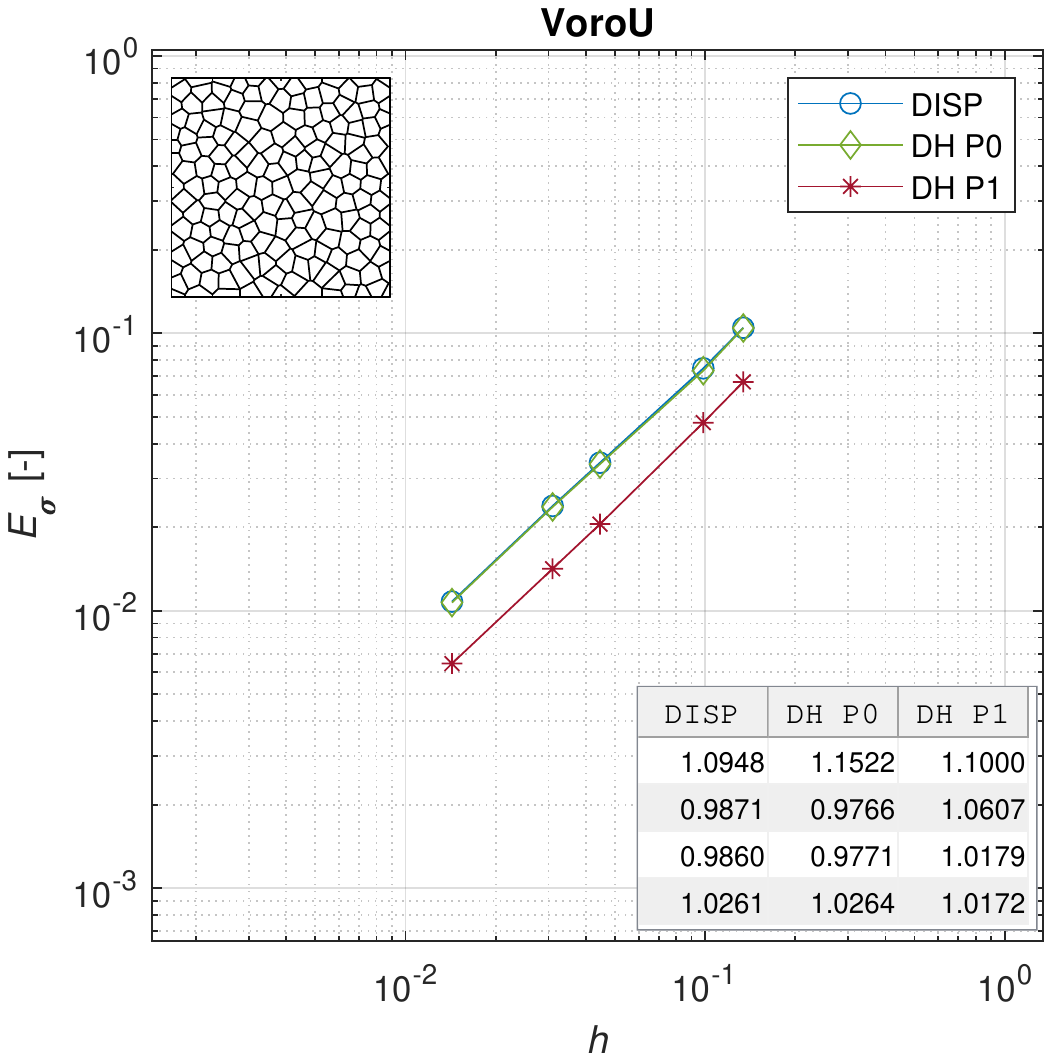}}
\caption{Test a - $E_{\bfsigma}$ {\it vs.} $h$ curves with branch slopes. Structured mesh - left. Unstructured mesh - right.}
\label{fig:Test_a_E_sig}
\end{figure}

\clearpage
\newpage
\begin{figure}
\subfigure
{\includegraphics[bb = 135 260 550 556,clip,scale=0.55]
                {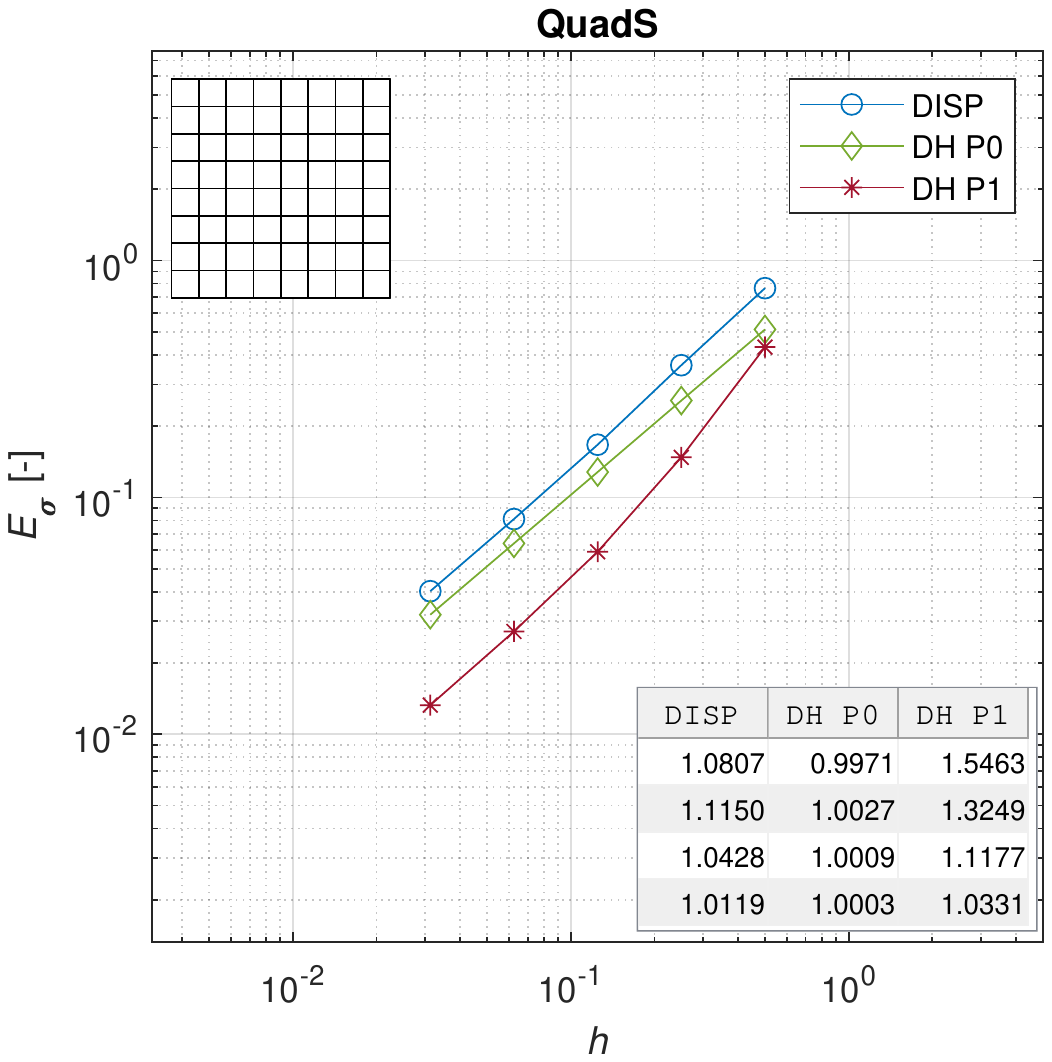}}
\subfigure
{\includegraphics[bb = 150 260 600 556,clip,scale=0.55]
                {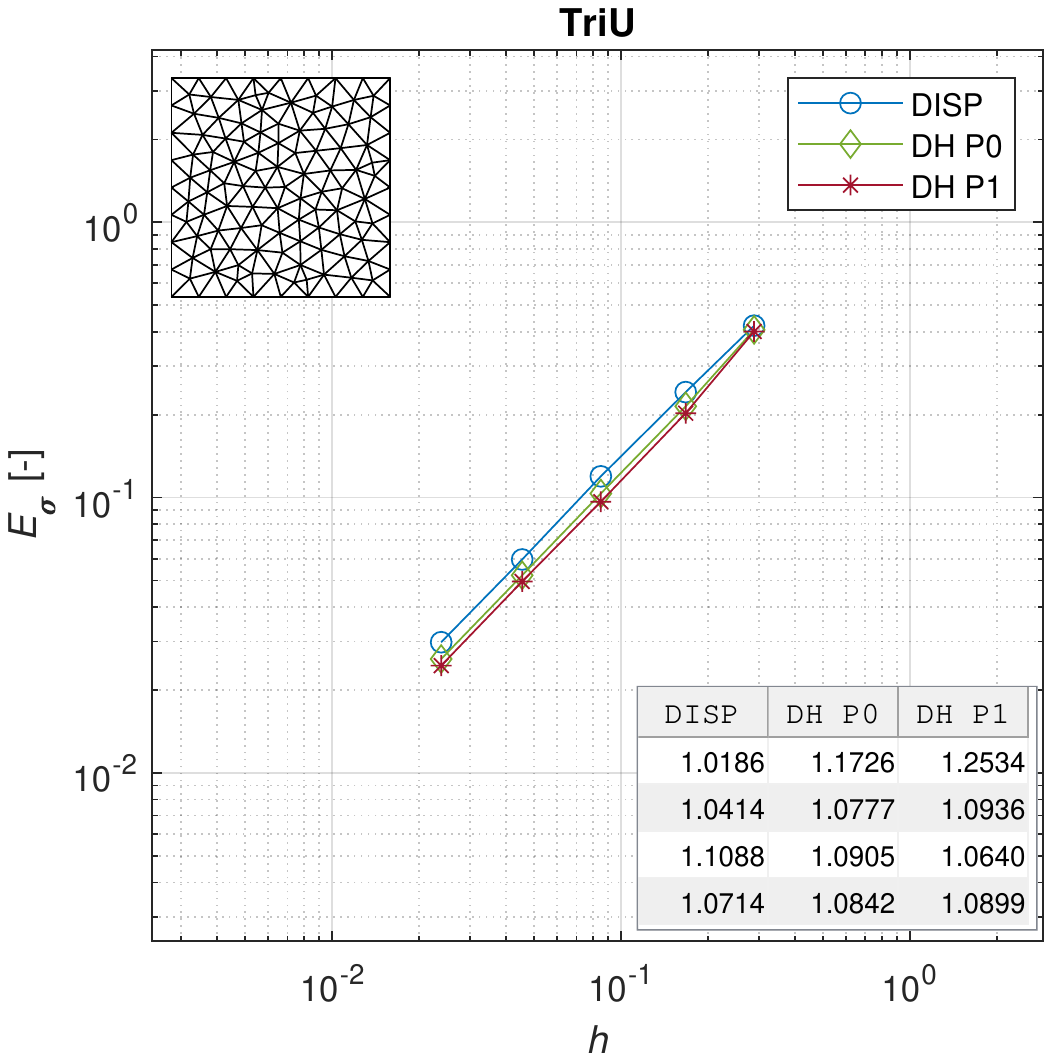}}
\subfigure
{\includegraphics[bb = 135 260 550 556,clip,scale=0.55]
                {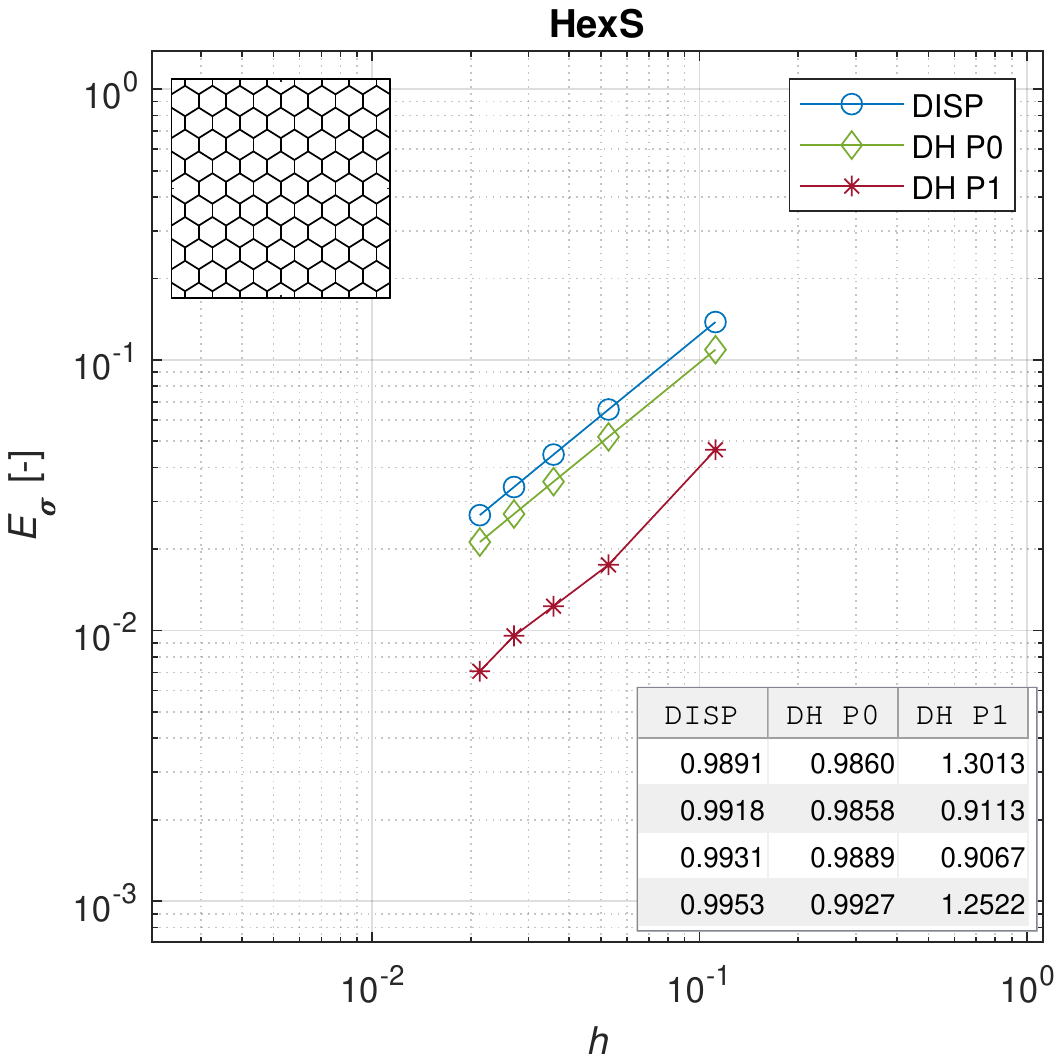}}
\subfigure
{\includegraphics[bb = 150 260 600 556,clip,scale=0.55]
                {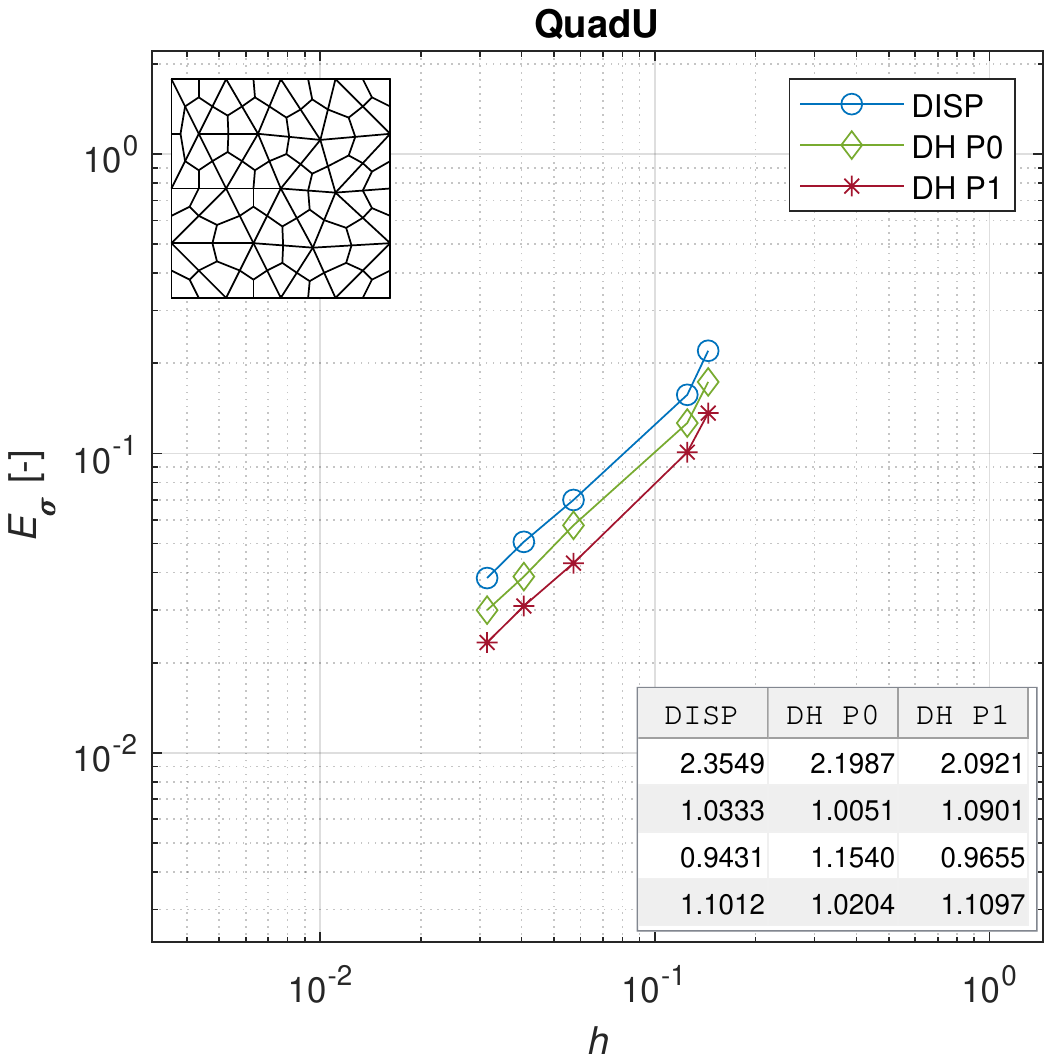}}
\subfigure
{\includegraphics[bb = 135 260 550 556,clip,scale=0.55]
                {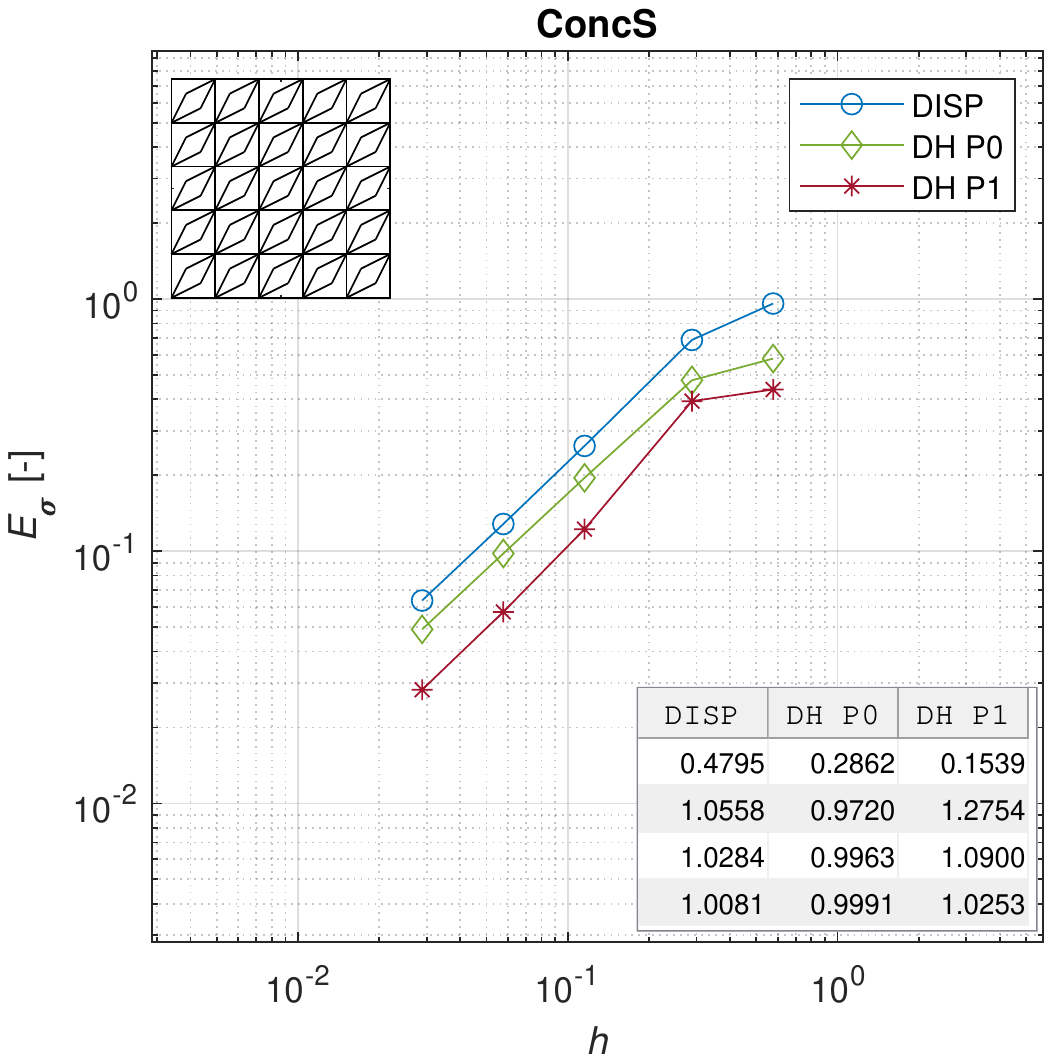}}
\subfigure
{\includegraphics[bb = 150 260 600 556,clip,scale=0.55]
                {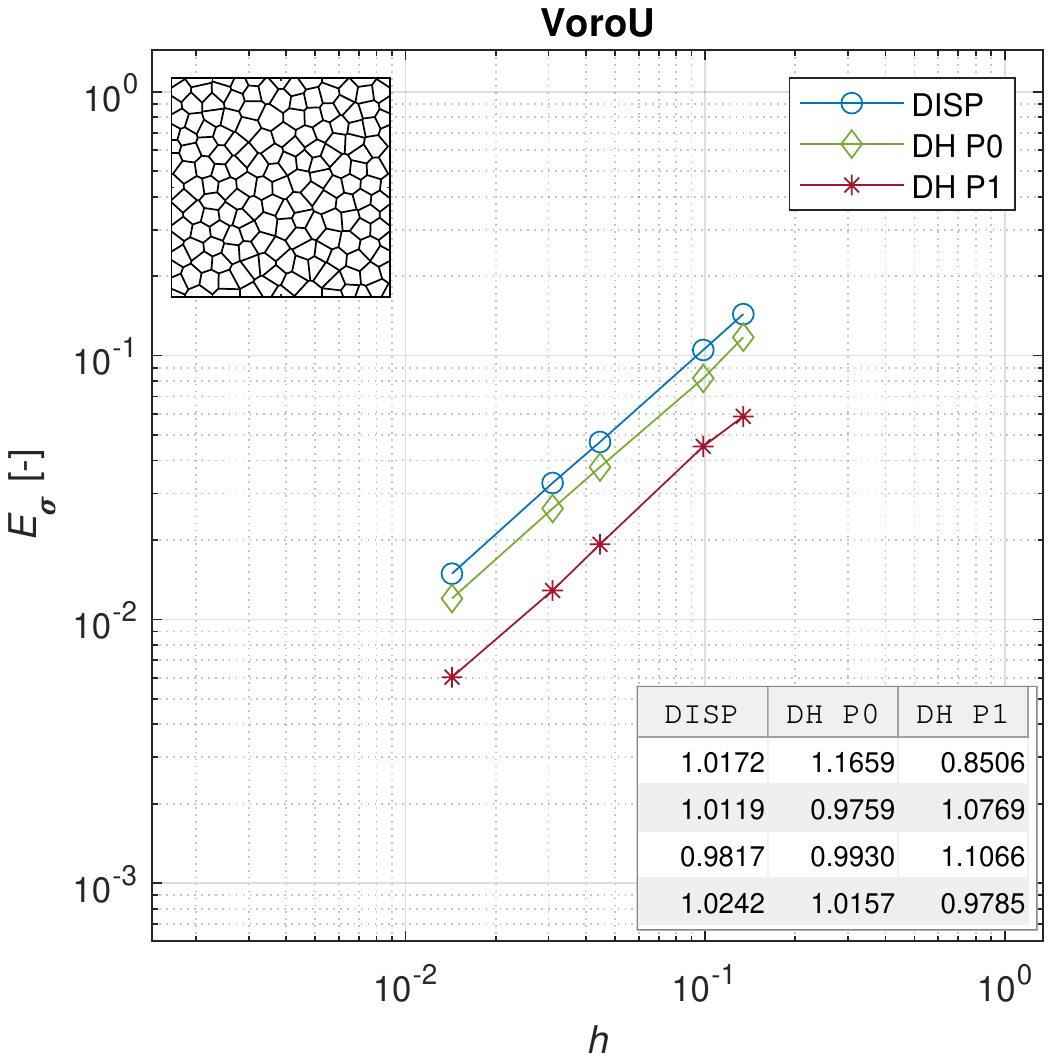}}
\caption{Test b - $E_{\bfsigma}$ {\it vs.} $h$ curves with branch slopes. Structured mesh - left. Unstructured mesh - right.}
\label{fig:Test_b_E_sig}
\end{figure}

\clearpage
\newpage
\begin{figure}
\subfigure
{\includegraphics[bb = 135 260 550 556,clip,scale=0.55]
                {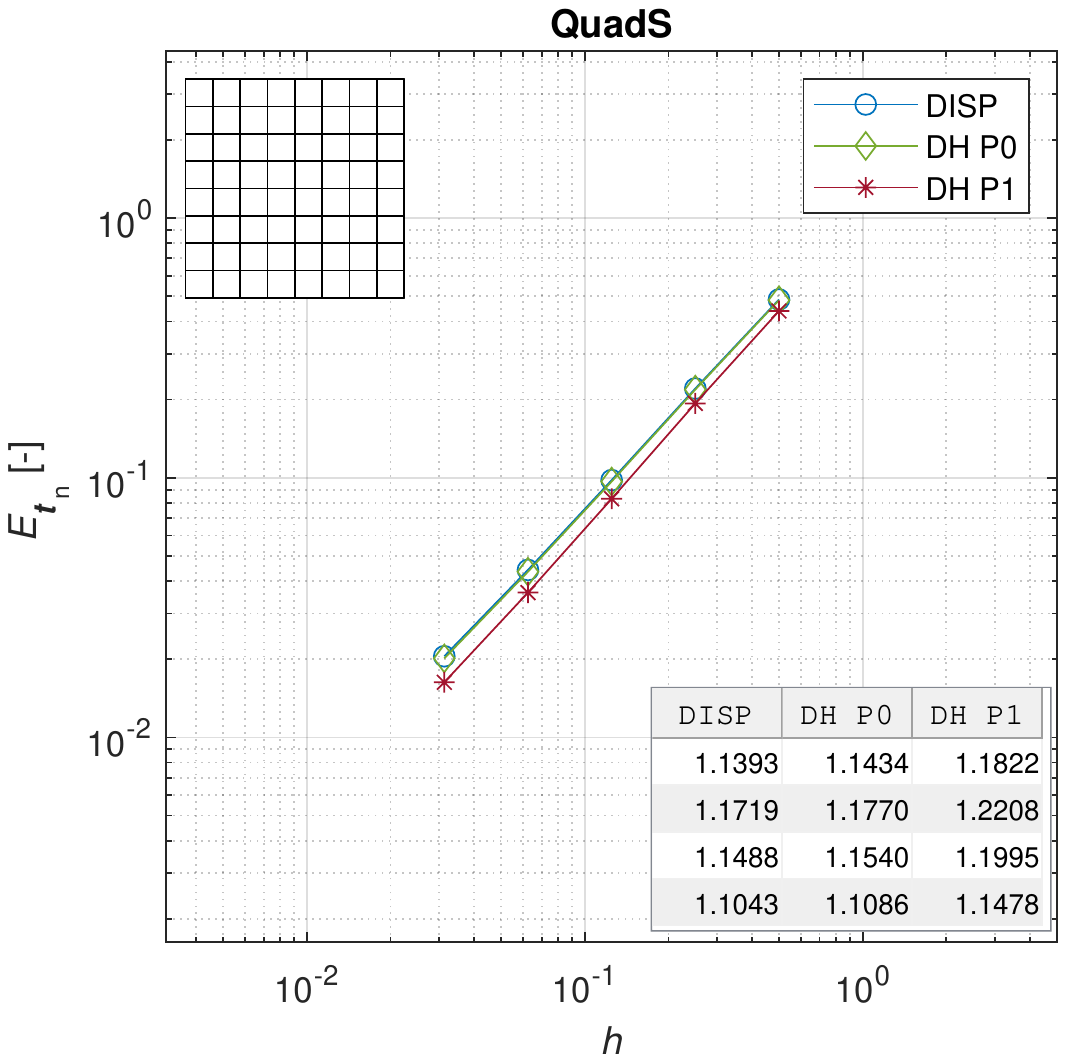}}
\subfigure
{\includegraphics[bb = 150 260 600 556,clip,scale=0.55]
                {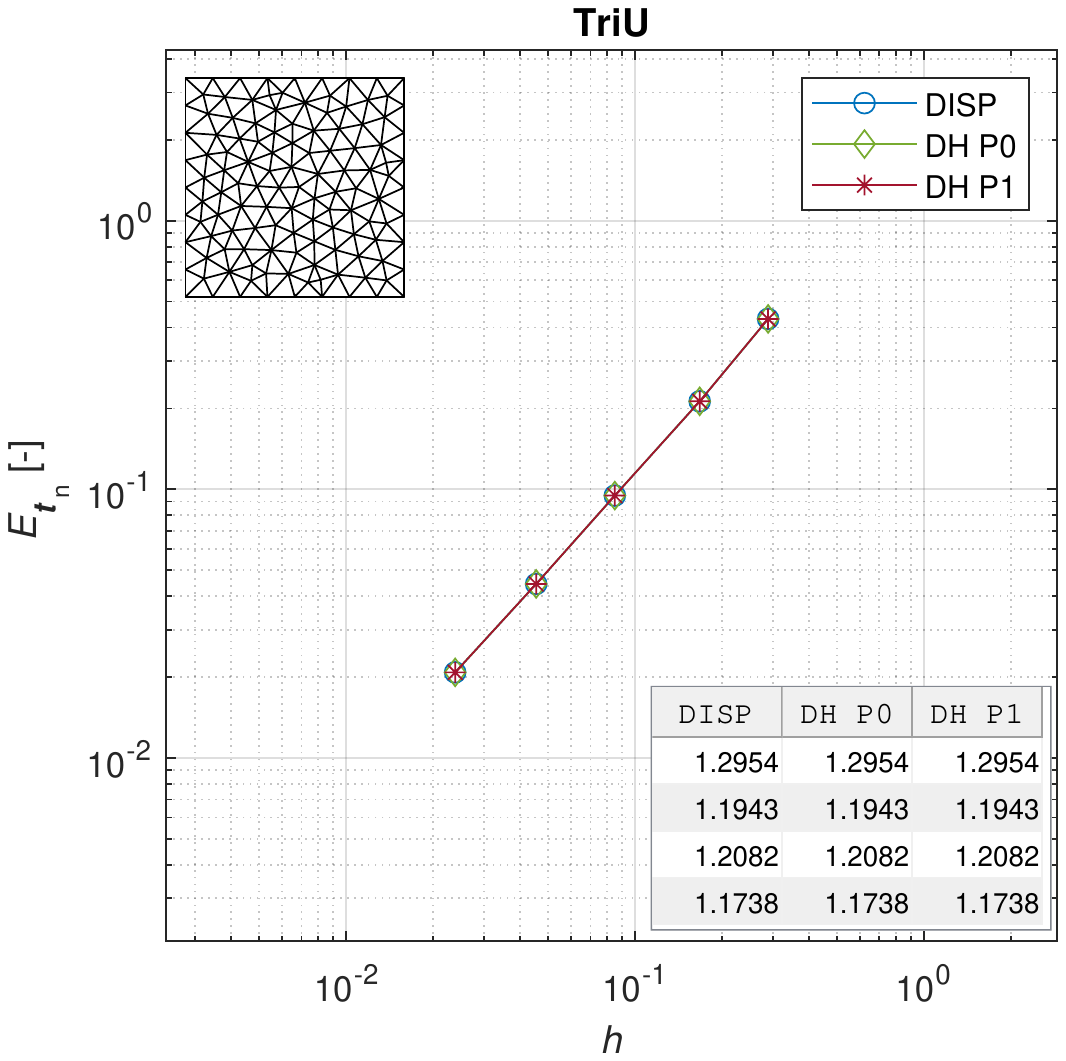}}
\subfigure
{\includegraphics[bb = 135 260 550 556,clip,scale=0.55]
                {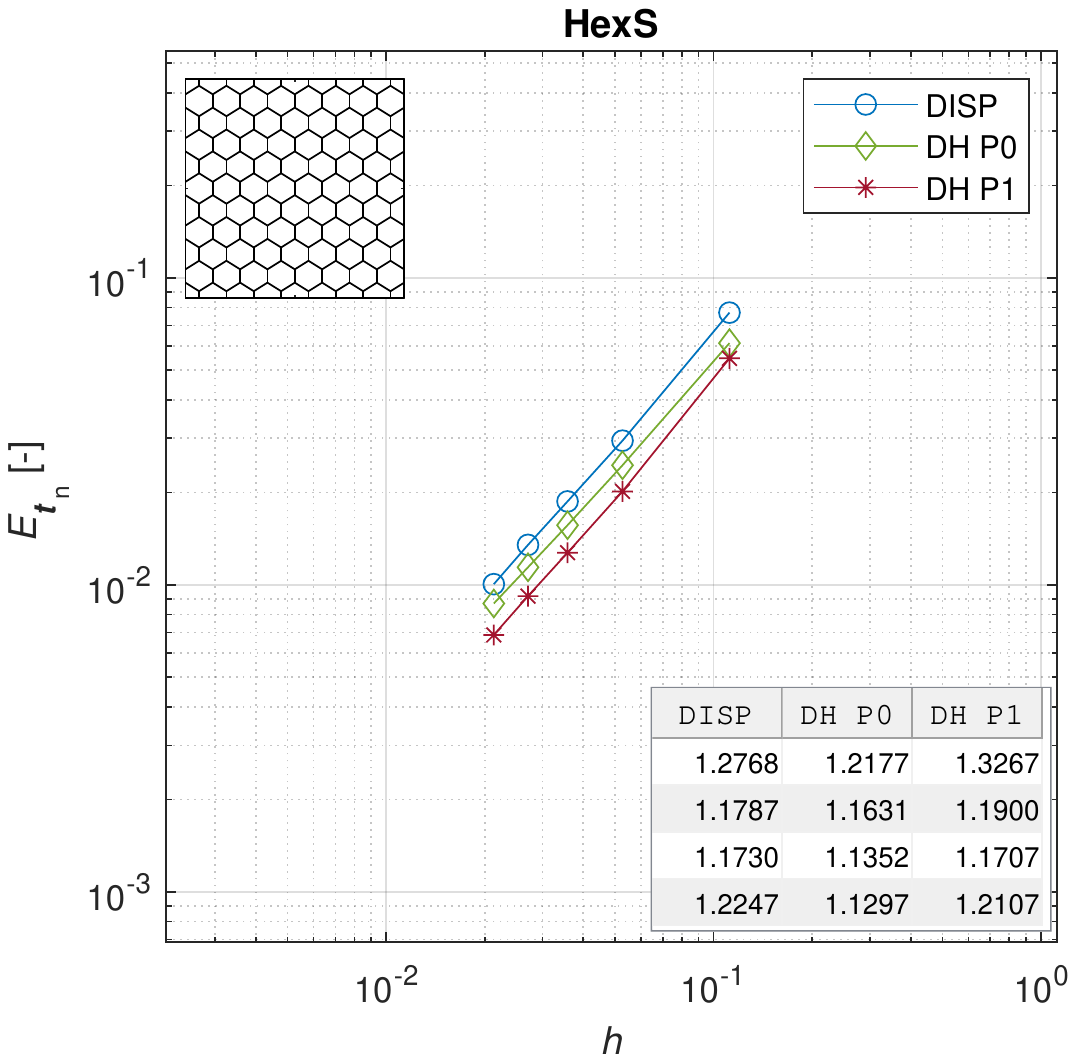}}
\subfigure
{\includegraphics[bb = 150 260 600 556,clip,scale=0.55]
                {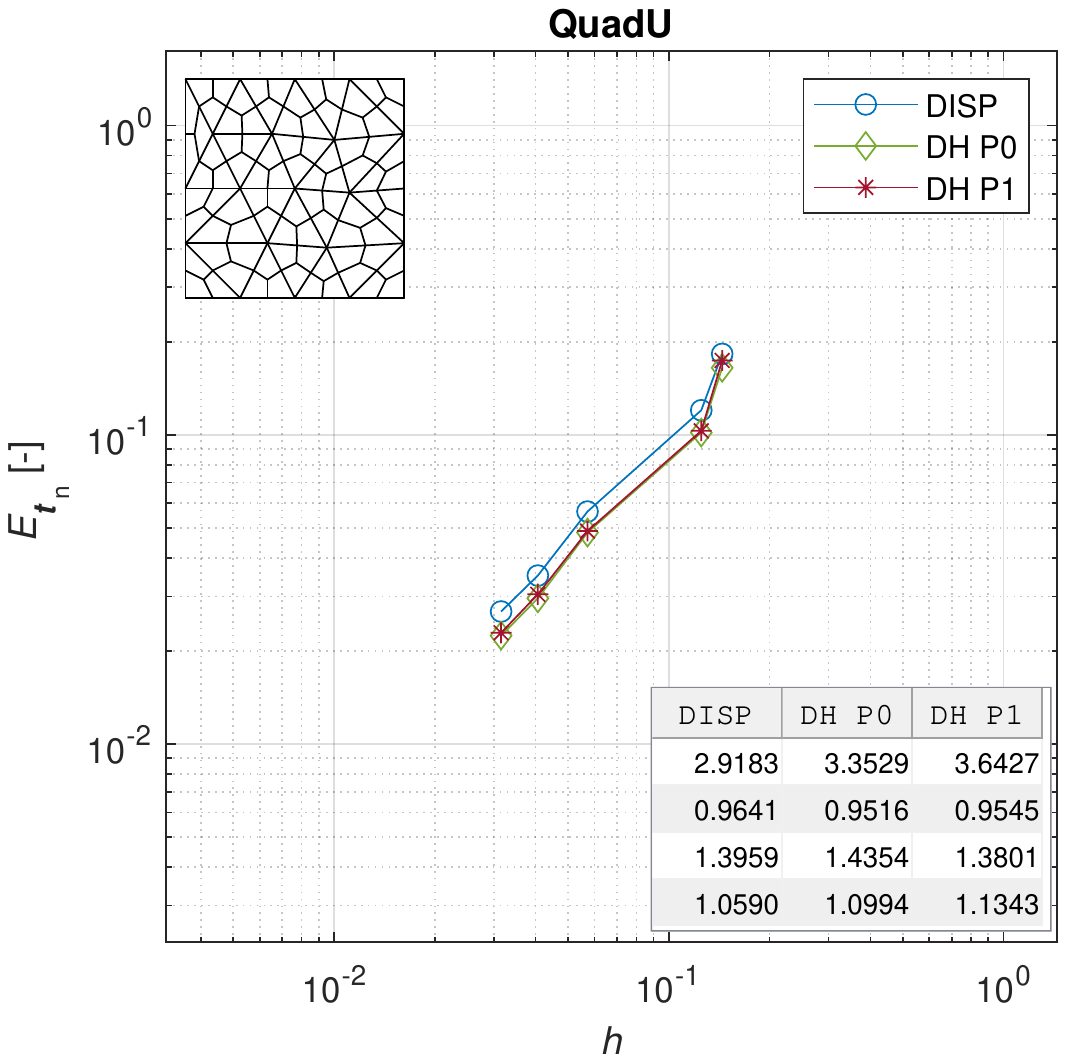}}
\subfigure
{\includegraphics[bb = 135 260 550 556,clip,scale=0.55]
                {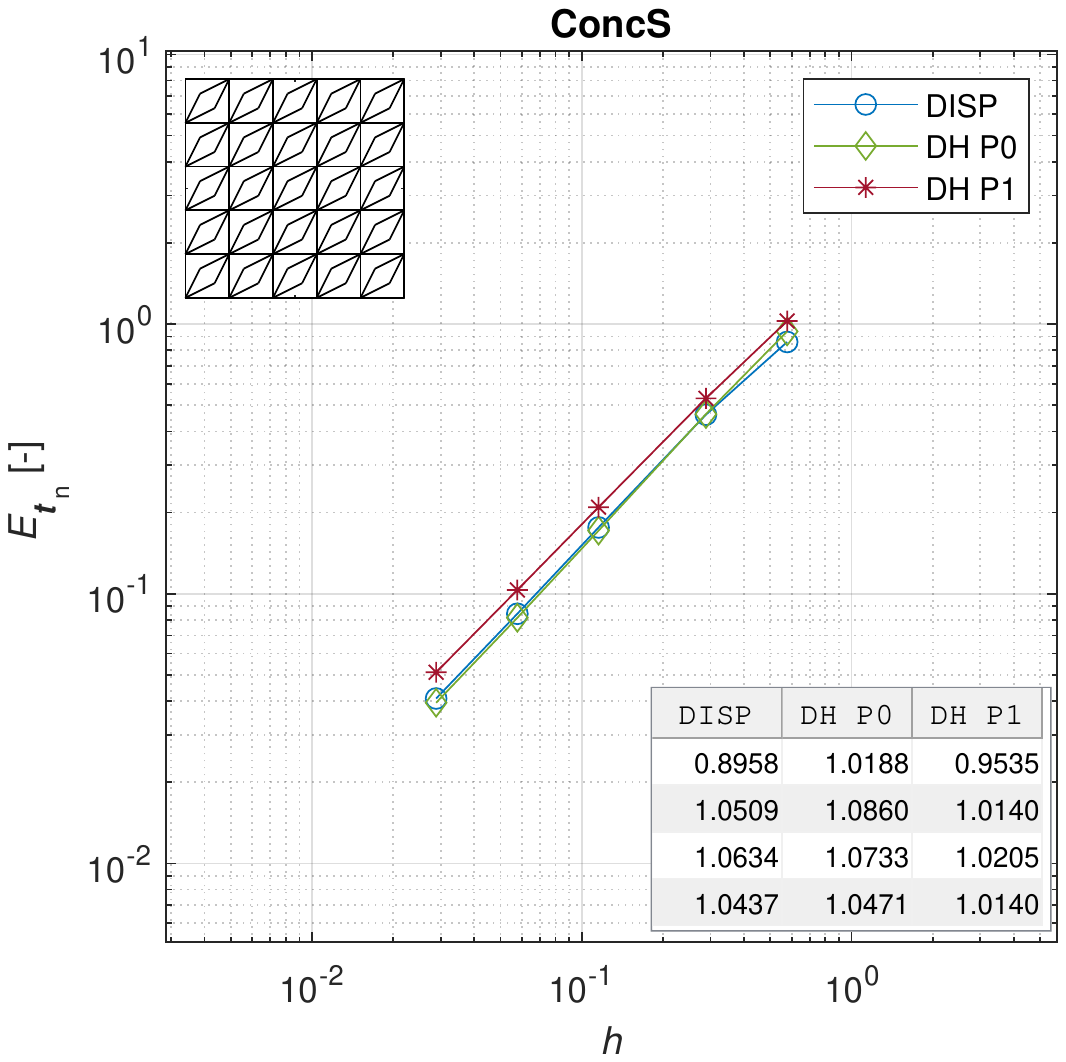}}
\subfigure
{\includegraphics[bb = 150 260 600 556,clip,scale=0.55]
                {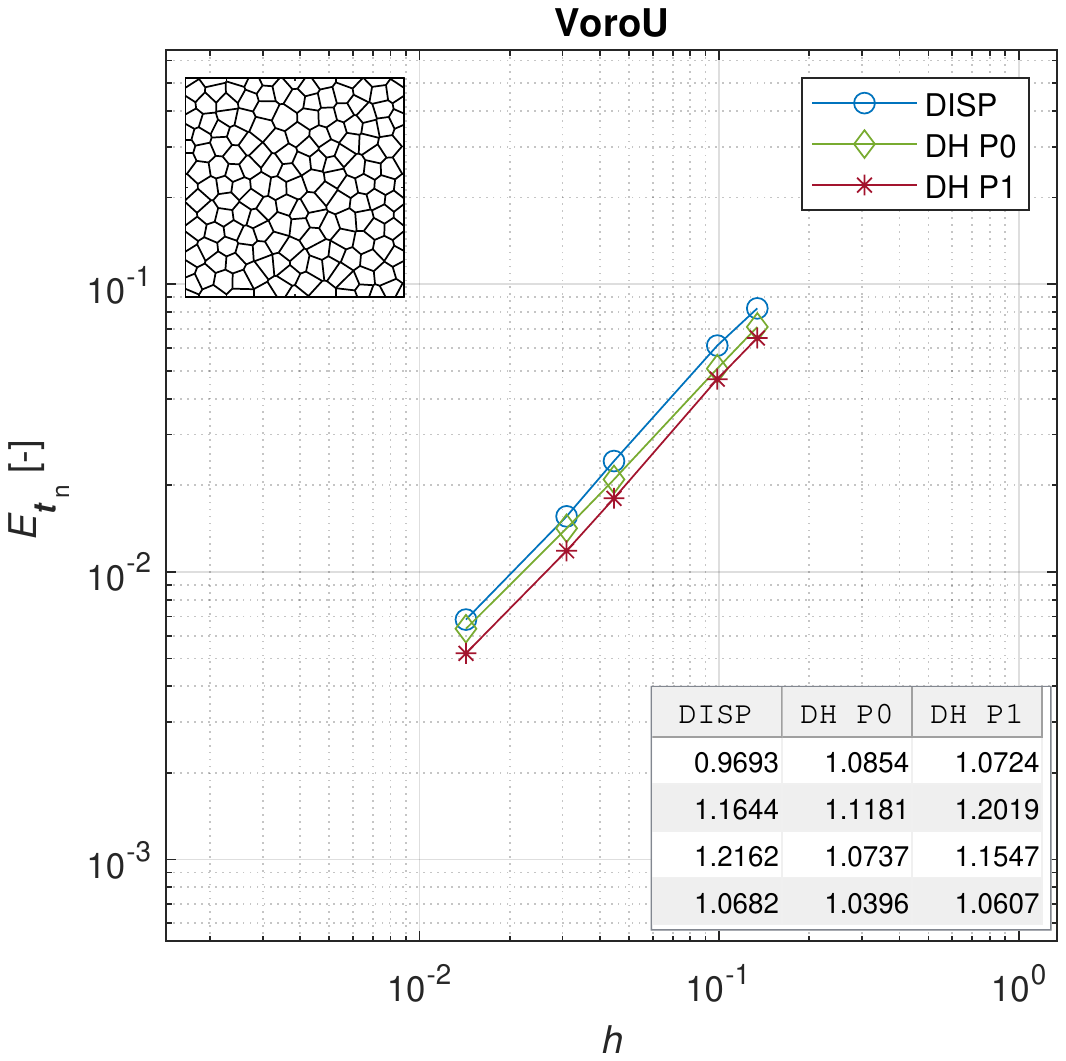}}
\caption{Test a - $E_{\bbt_{\bbn}}$ {\it vs.} $h$ curves with branch slopes. Structured mesh - left. Unstructured mesh - right.}
\label{fig:Test_a_E_tn}
\end{figure}

\clearpage
\newpage
\begin{figure}
\subfigure
{\includegraphics[bb = 135 260 550 556,clip,scale=0.55]
                {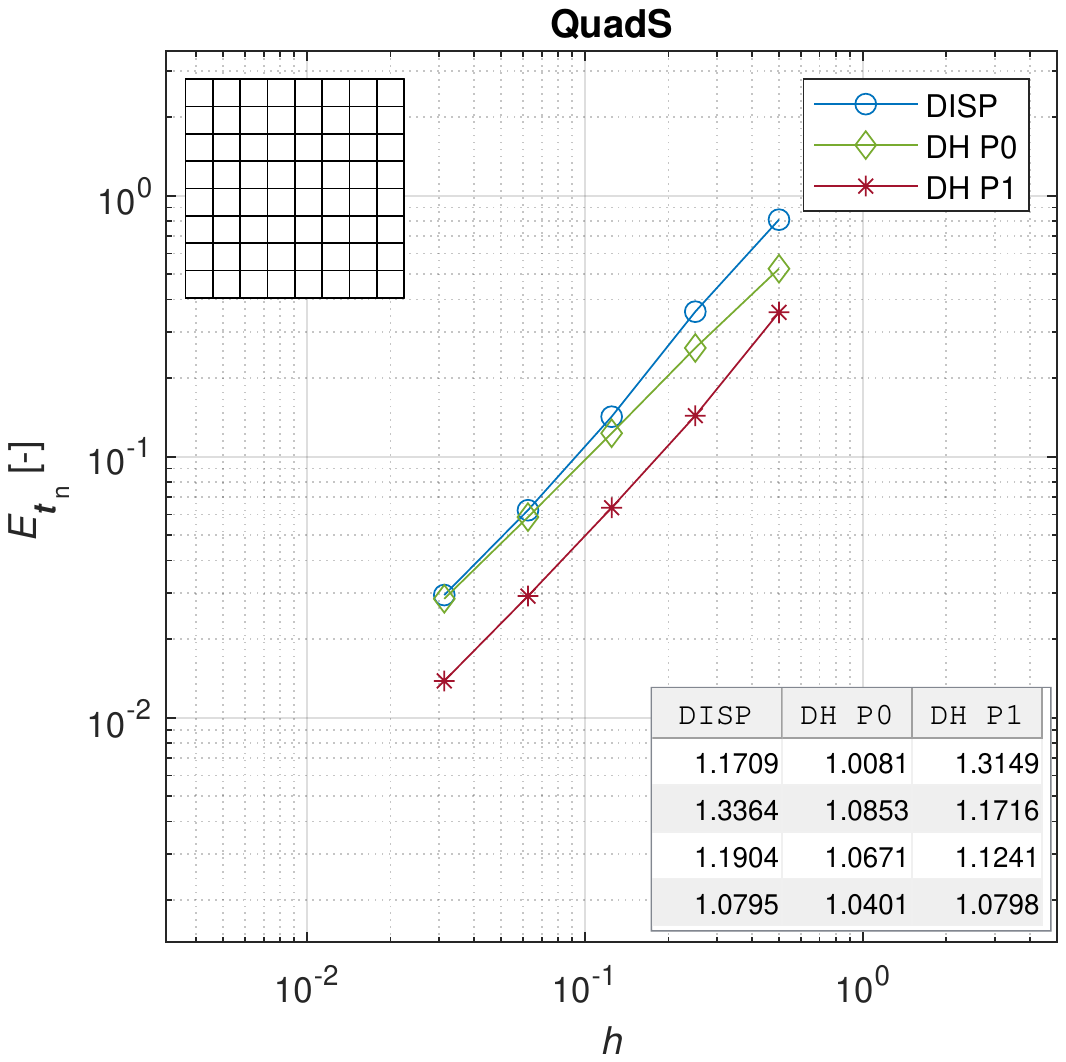}}
\subfigure
{\includegraphics[bb = 150 260 600 556,clip,scale=0.55]
                {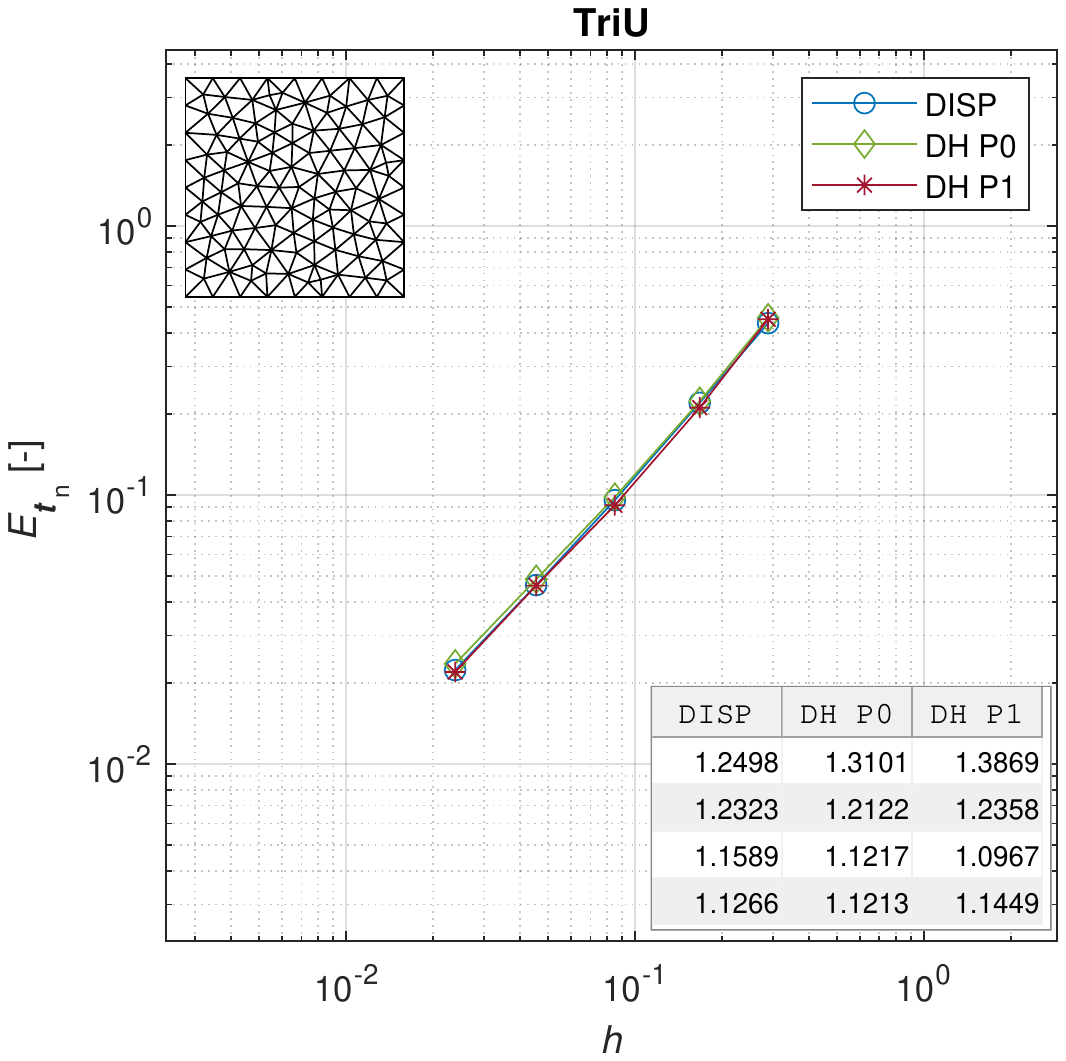}}
\subfigure
{\includegraphics[bb = 135 260 550 556,clip,scale=0.55]
                {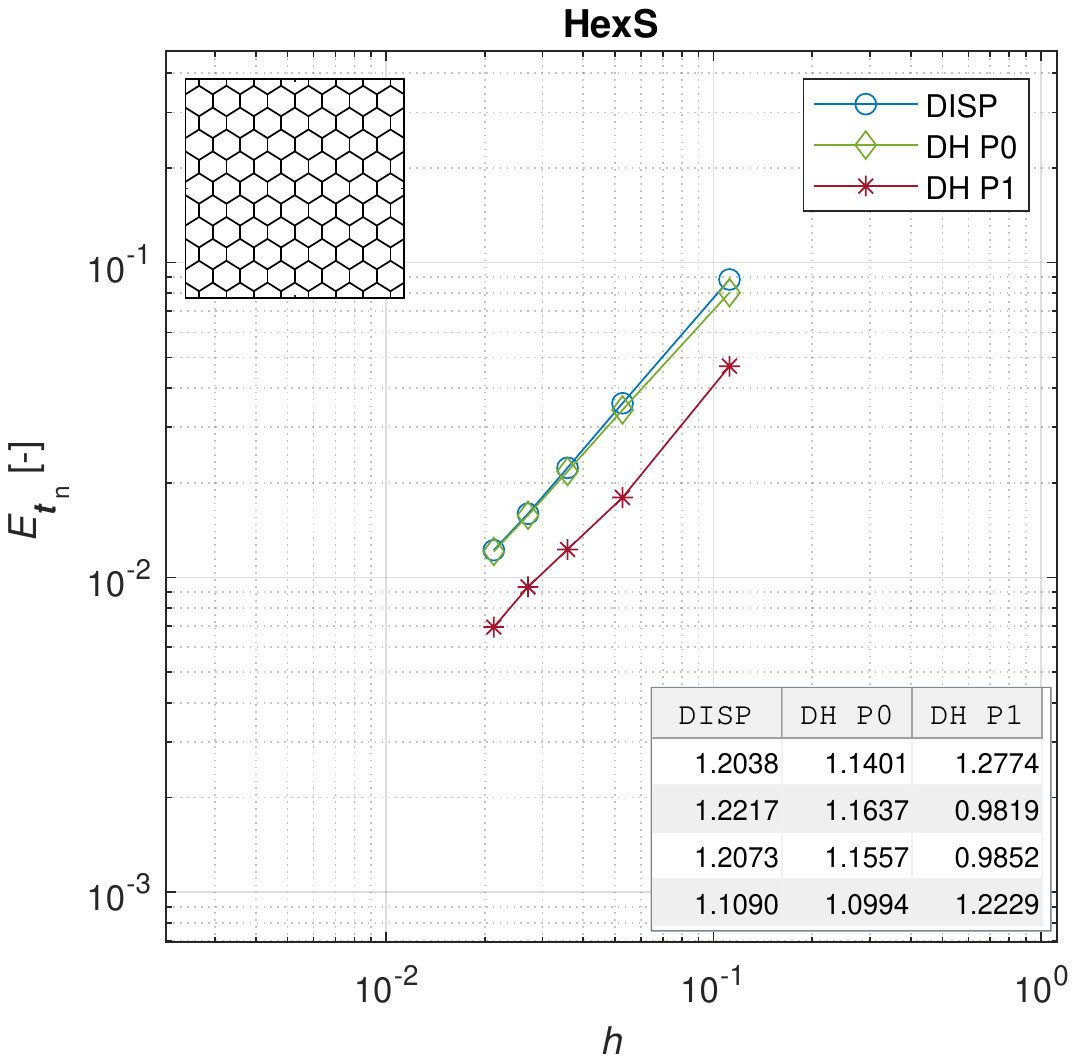}}
\subfigure
{\includegraphics[bb = 150 260 600 556,clip,scale=0.55]
                {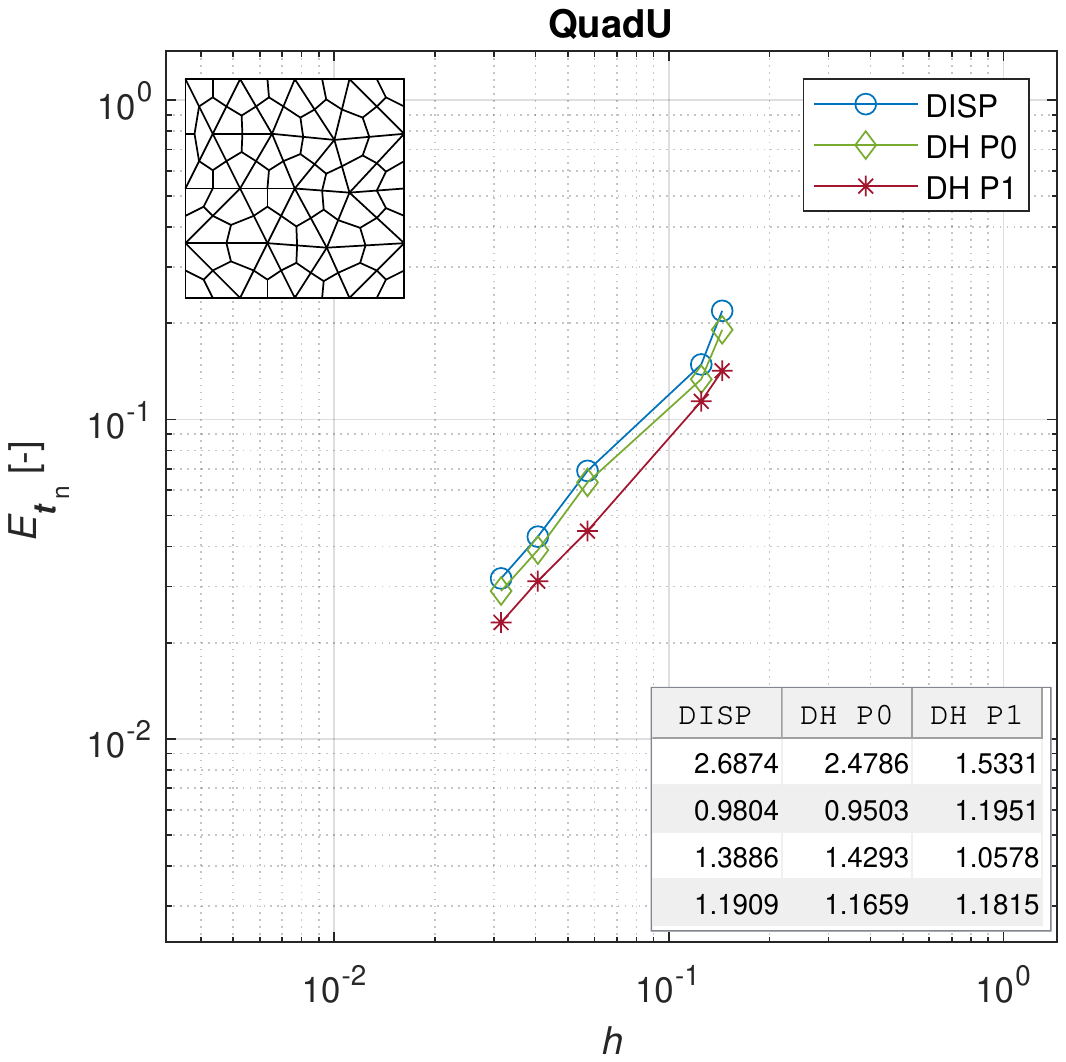}}
\subfigure
{\includegraphics[bb = 135 260 550 556,clip,scale=0.55]
                {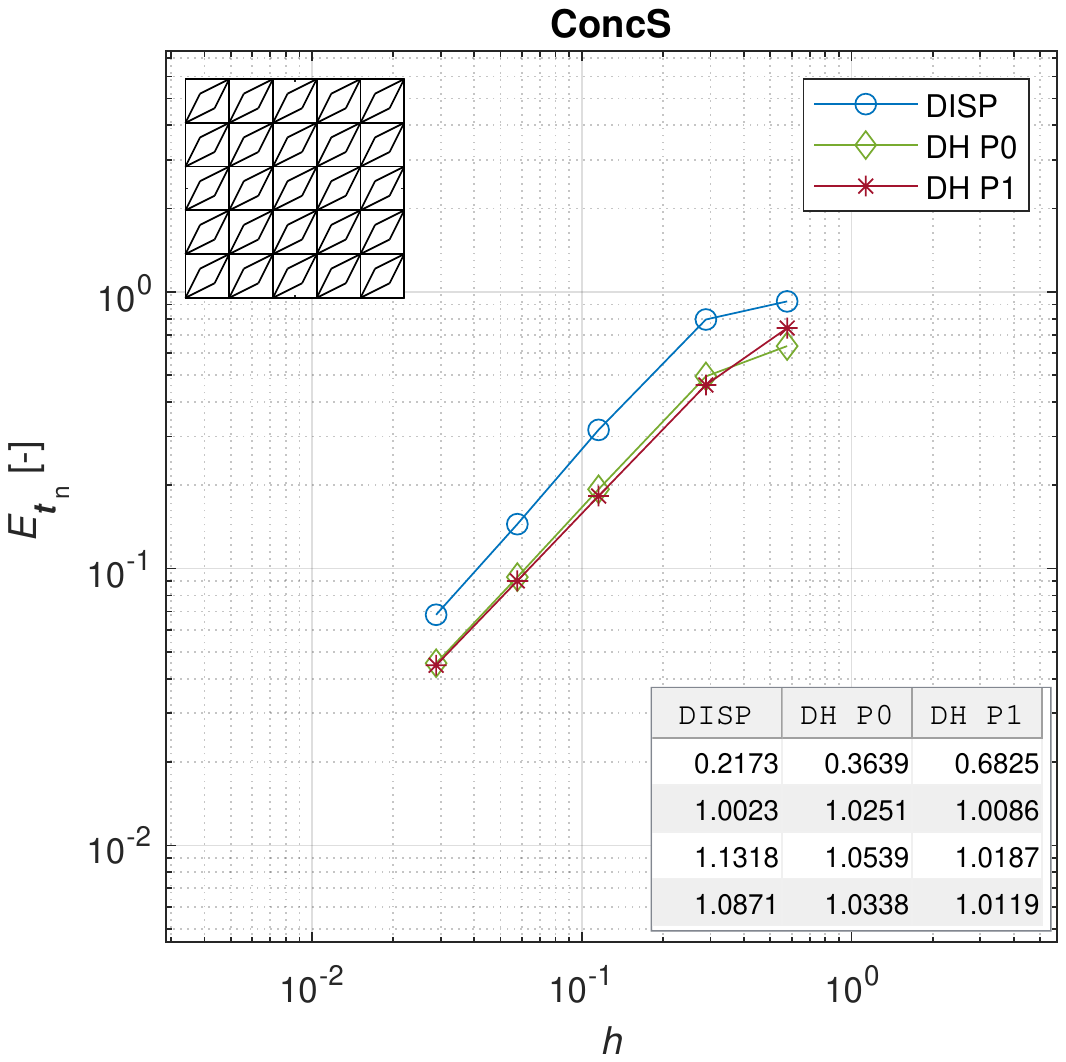}}
\subfigure
{\includegraphics[bb = 150 260 600 556,clip,scale=0.55]
                {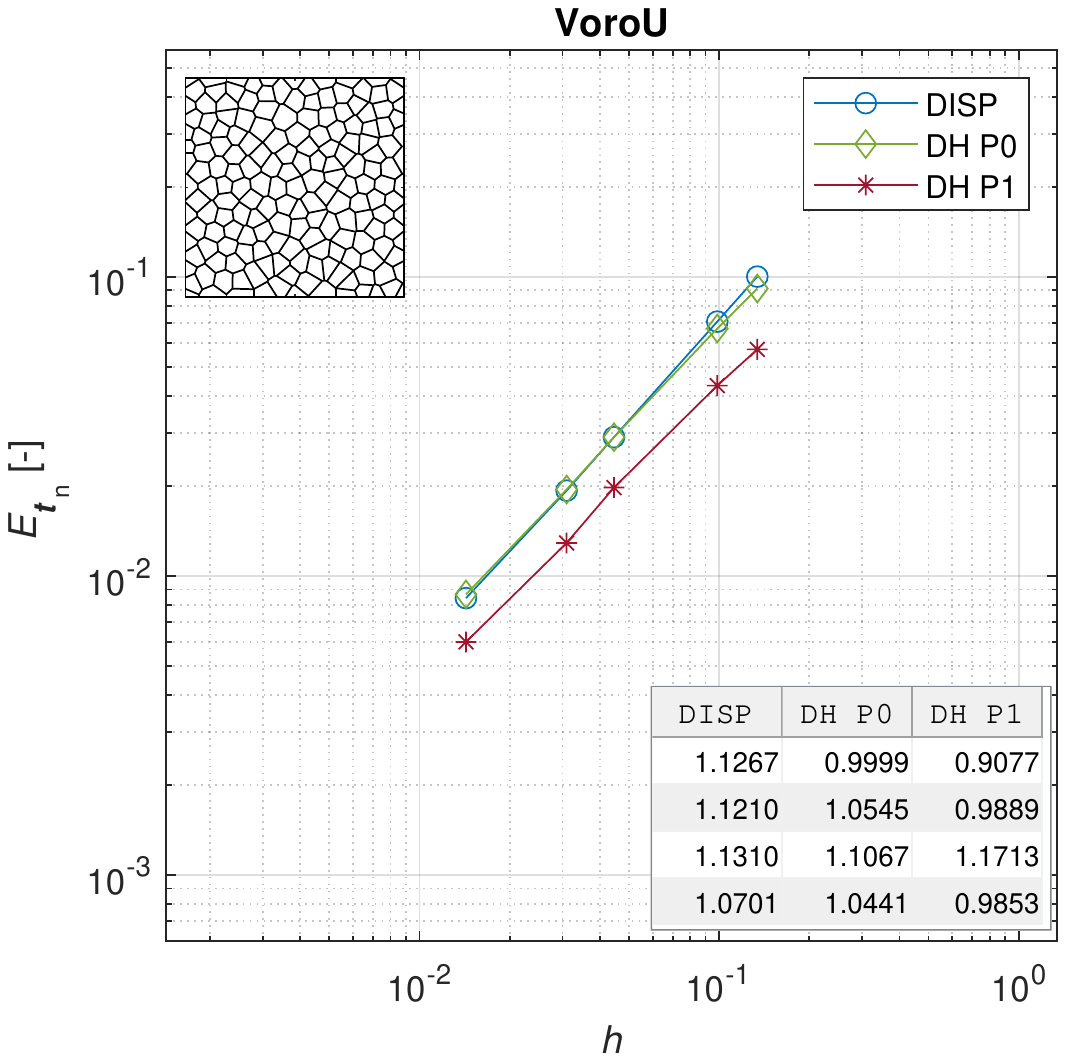}}
\caption{Test b - $E_{\bbt_{\bbn}}$ {\it vs.} $h$ curves with branch slopes. Structured mesh - left. Unstructured mesh - right.}
\label{fig:Test_b_E_tn}
\end{figure}

\clearpage
\newpage
\begin{figure}
\subfigure
{\includegraphics[bb = 135 260 550 556,clip,scale=0.55]
                {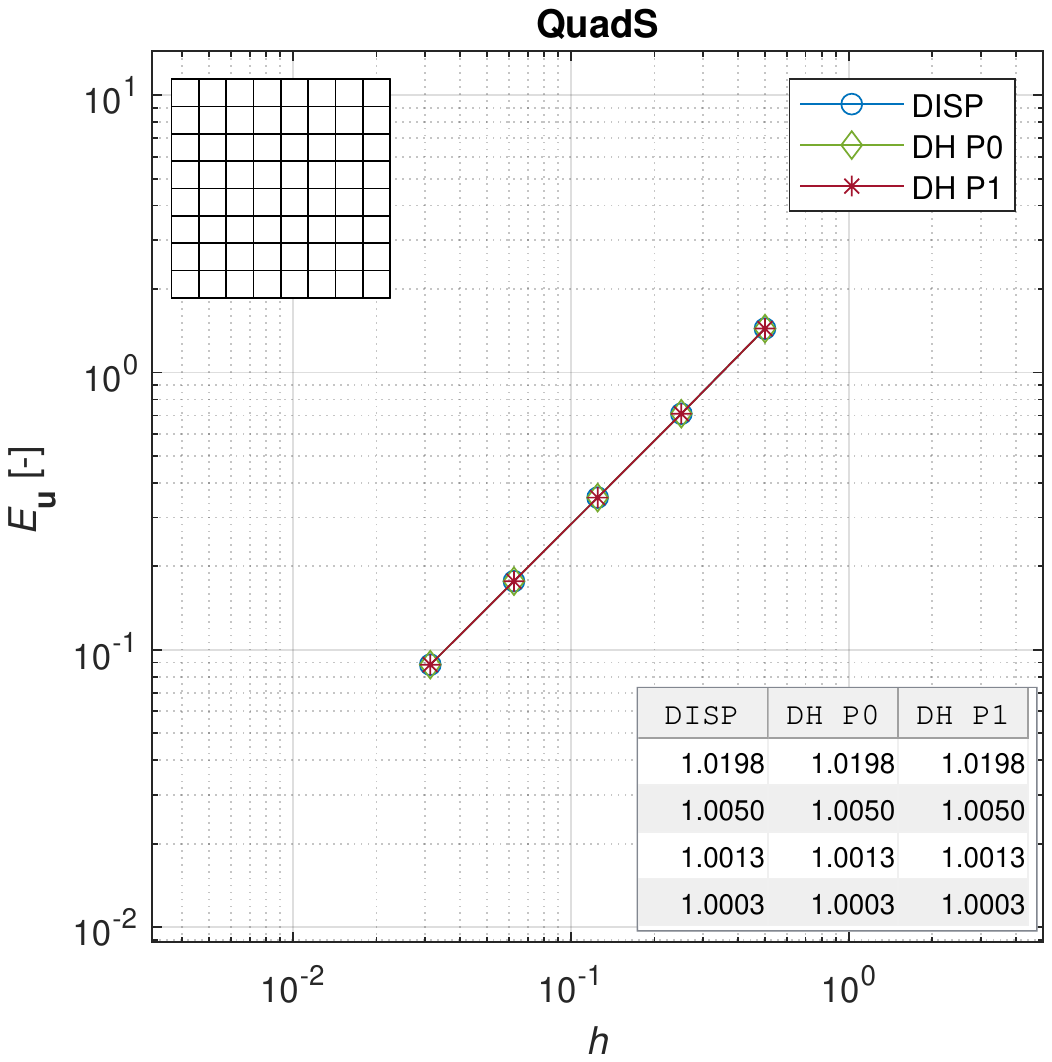}}
\subfigure
{\includegraphics[bb = 150 260 600 556,clip,scale=0.55]
                {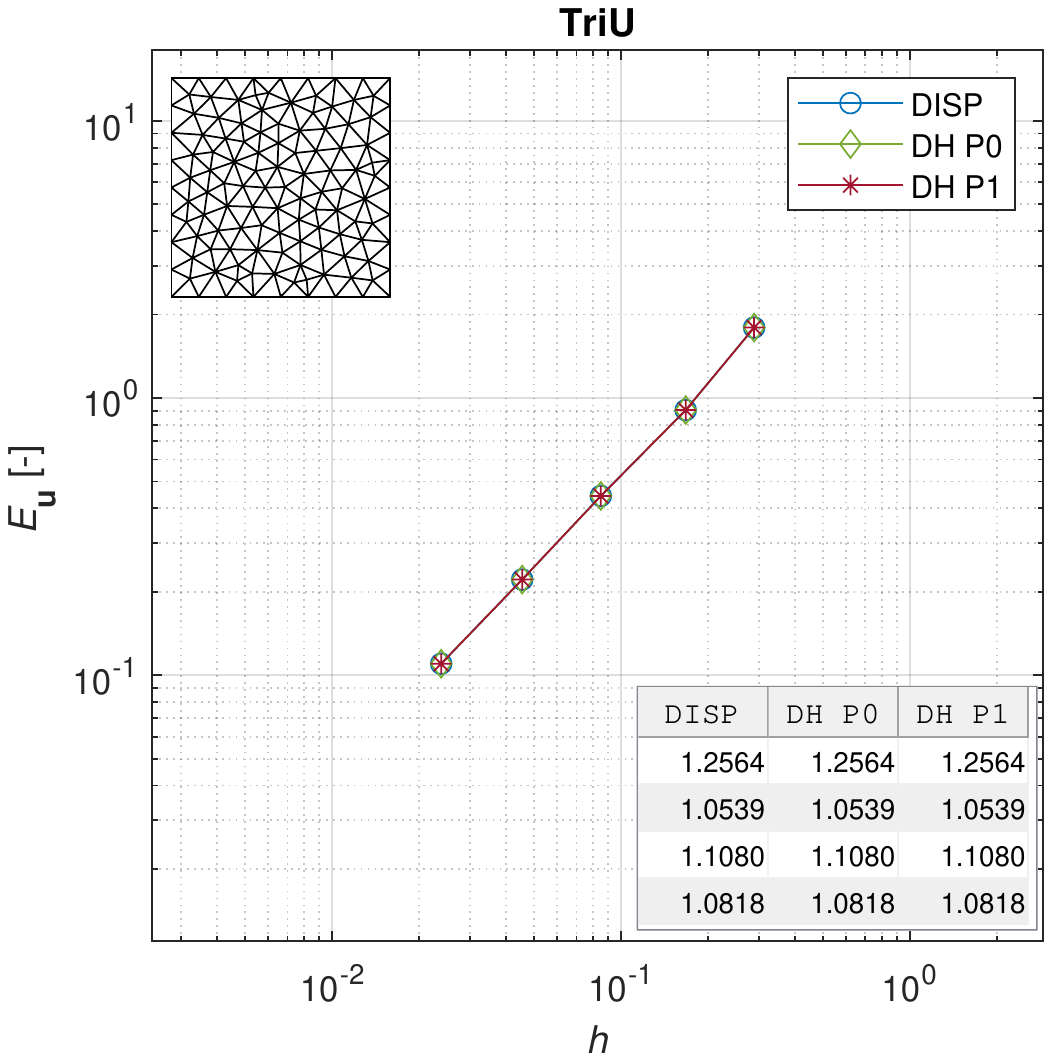}}
\subfigure
{\includegraphics[bb = 135 260 550 556,clip,scale=0.55]
                {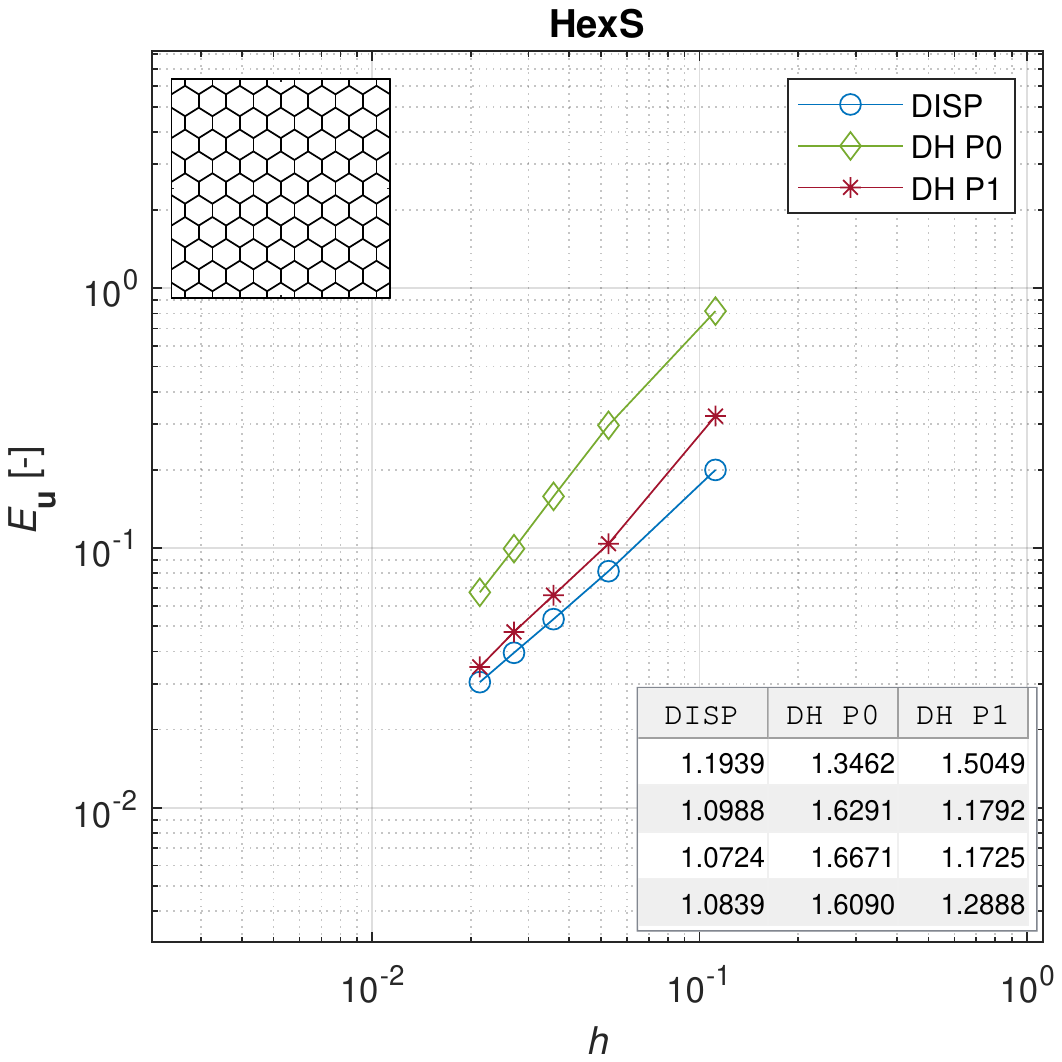}}
\subfigure
{\includegraphics[bb = 150 260 600 556,clip,scale=0.55]
                {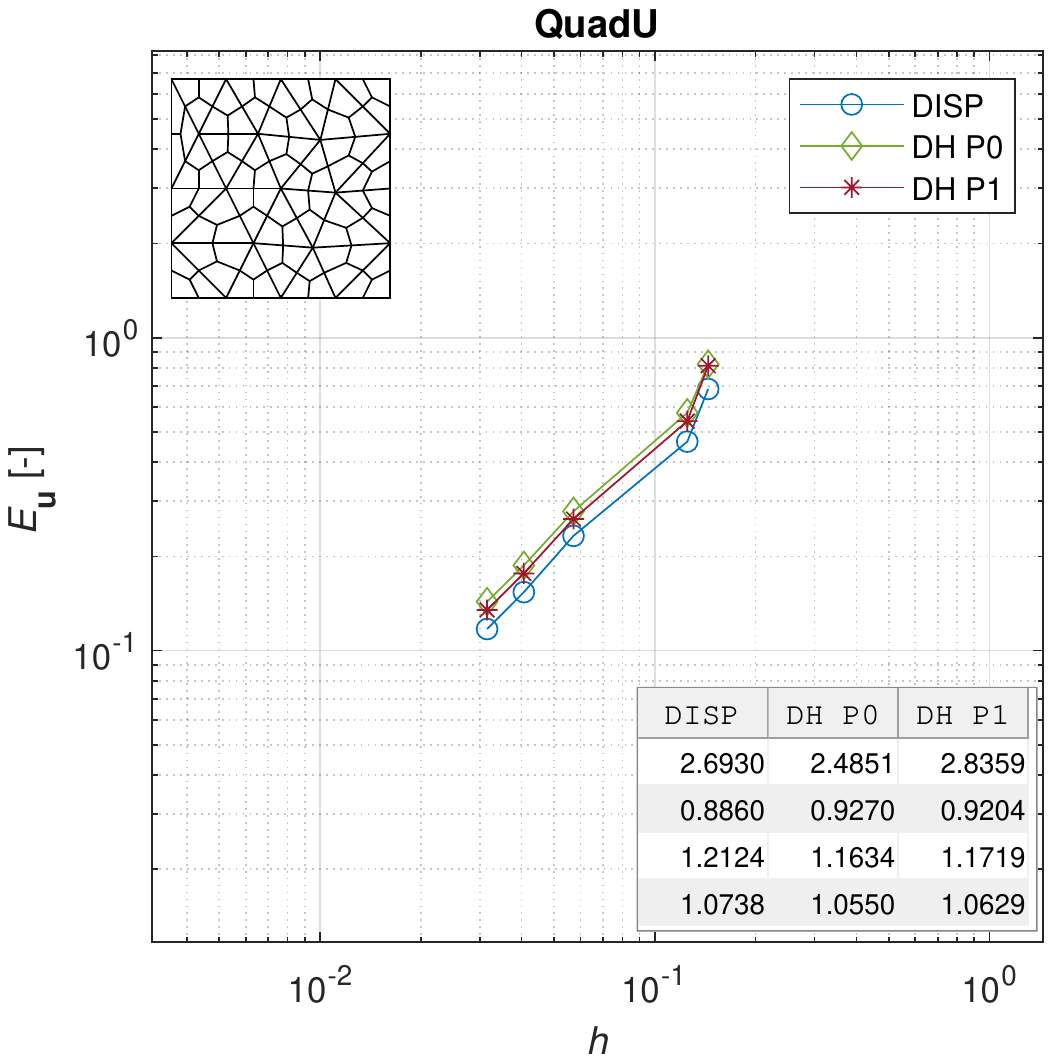}}
\subfigure
{\includegraphics[bb = 135 260 550 556,clip,scale=0.55]
                {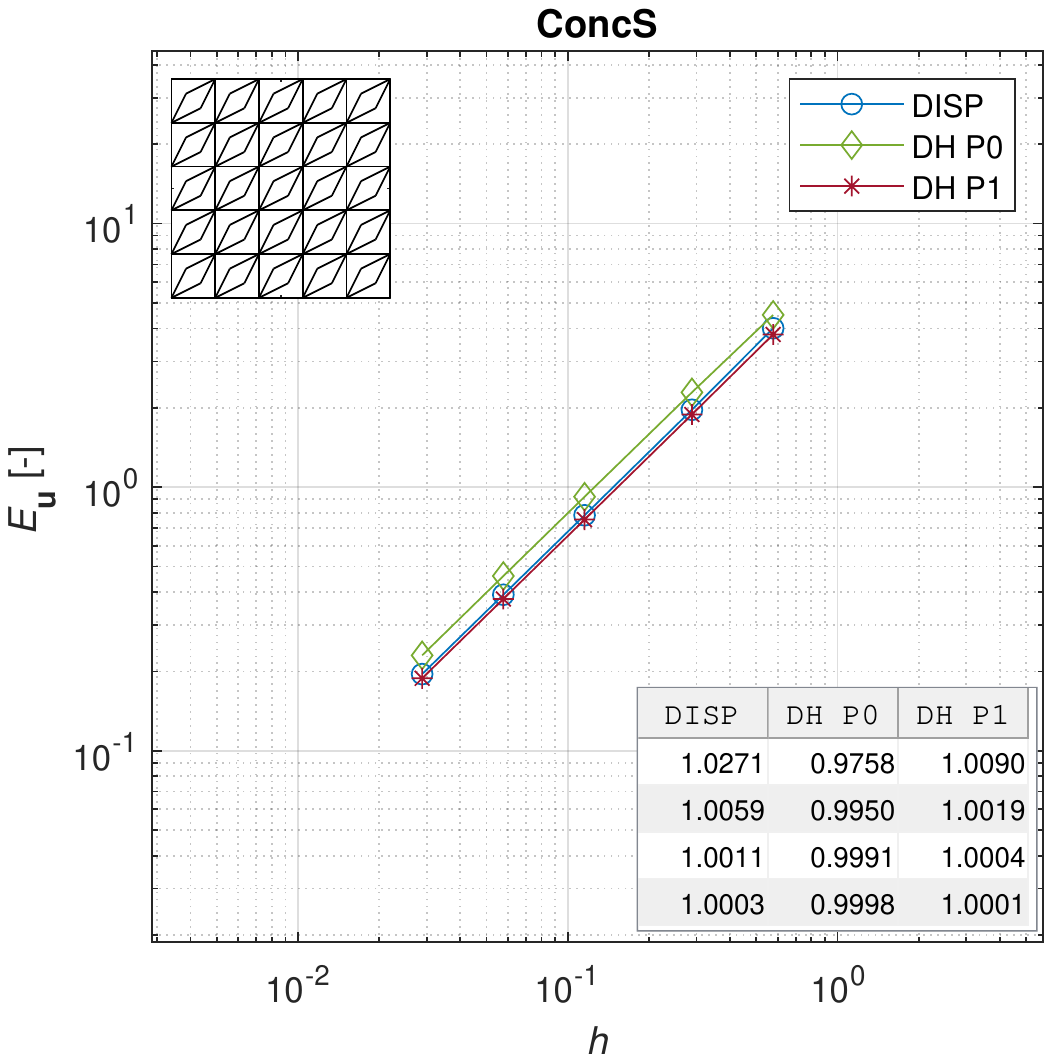}}
\subfigure
{\includegraphics[bb = 150 260 600 556,clip,scale=0.55]
                {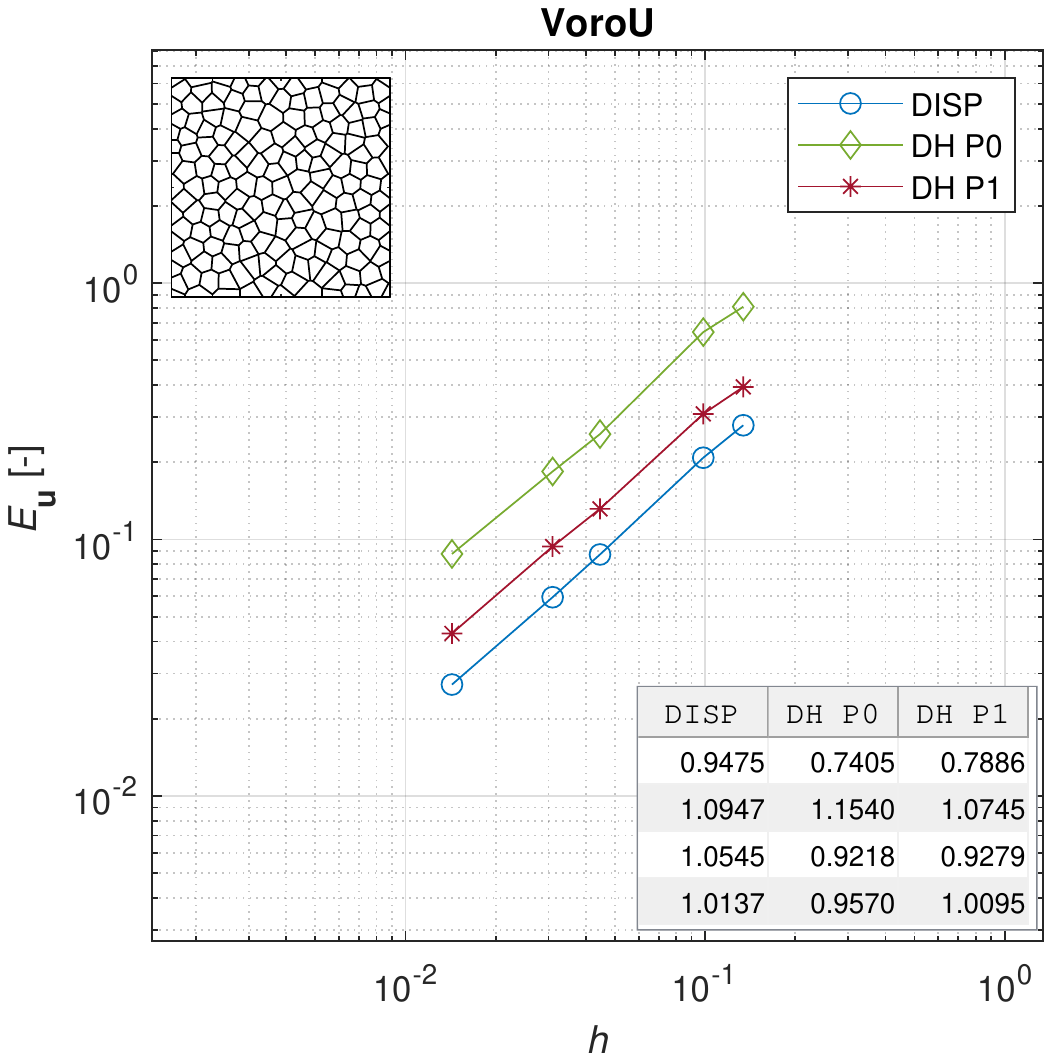}}
\caption{Test a - $E_{\bbu}$ {\it vs.} $h$ curves with branch slopes. Structured mesh - left. Unstructured mesh - right.}
\label{fig:Test_a_E_u}
\end{figure}

\clearpage
\newpage
\begin{figure}
\subfigure
{\includegraphics[bb = 135 260 550 556,clip,scale=0.55]
                {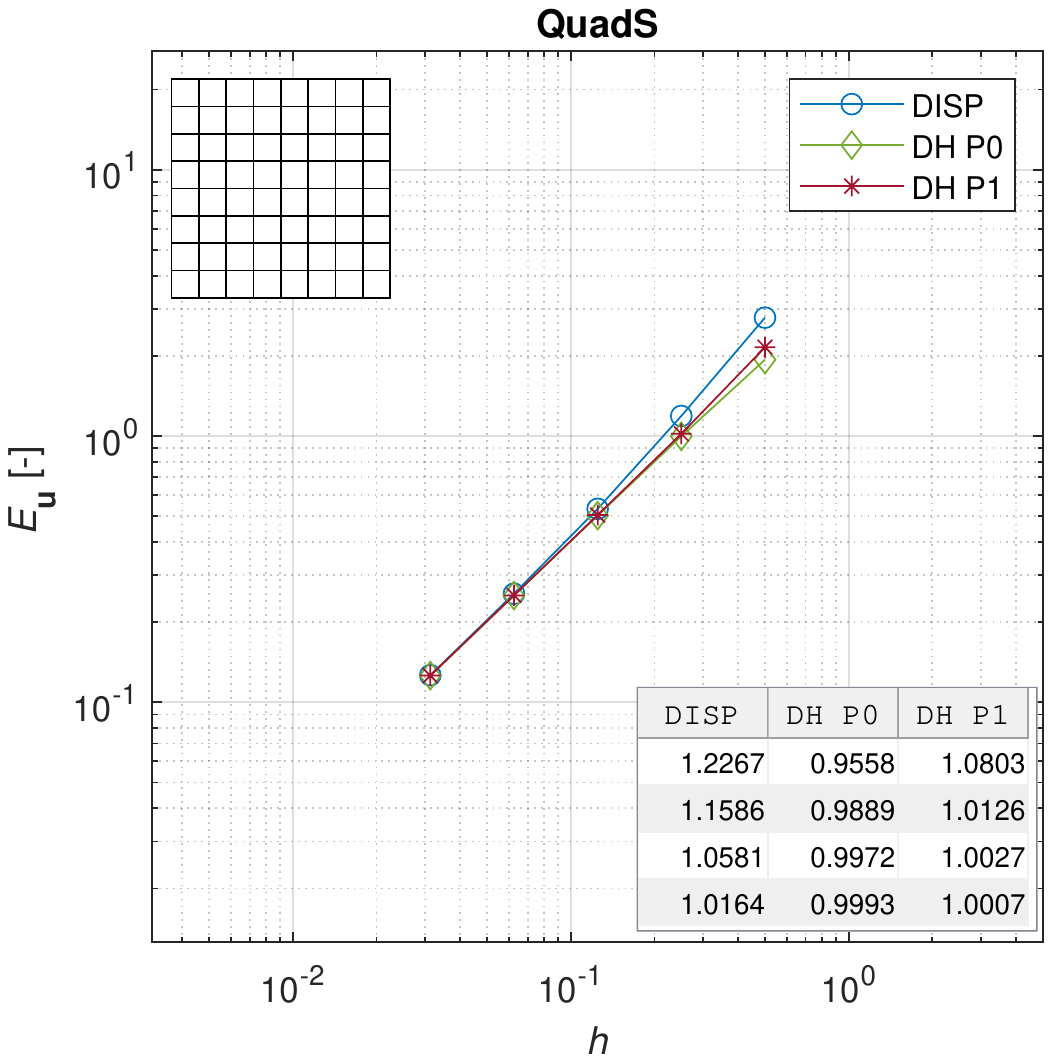}}
\subfigure
{\includegraphics[bb = 150 260 600 556,clip,scale=0.55]
                {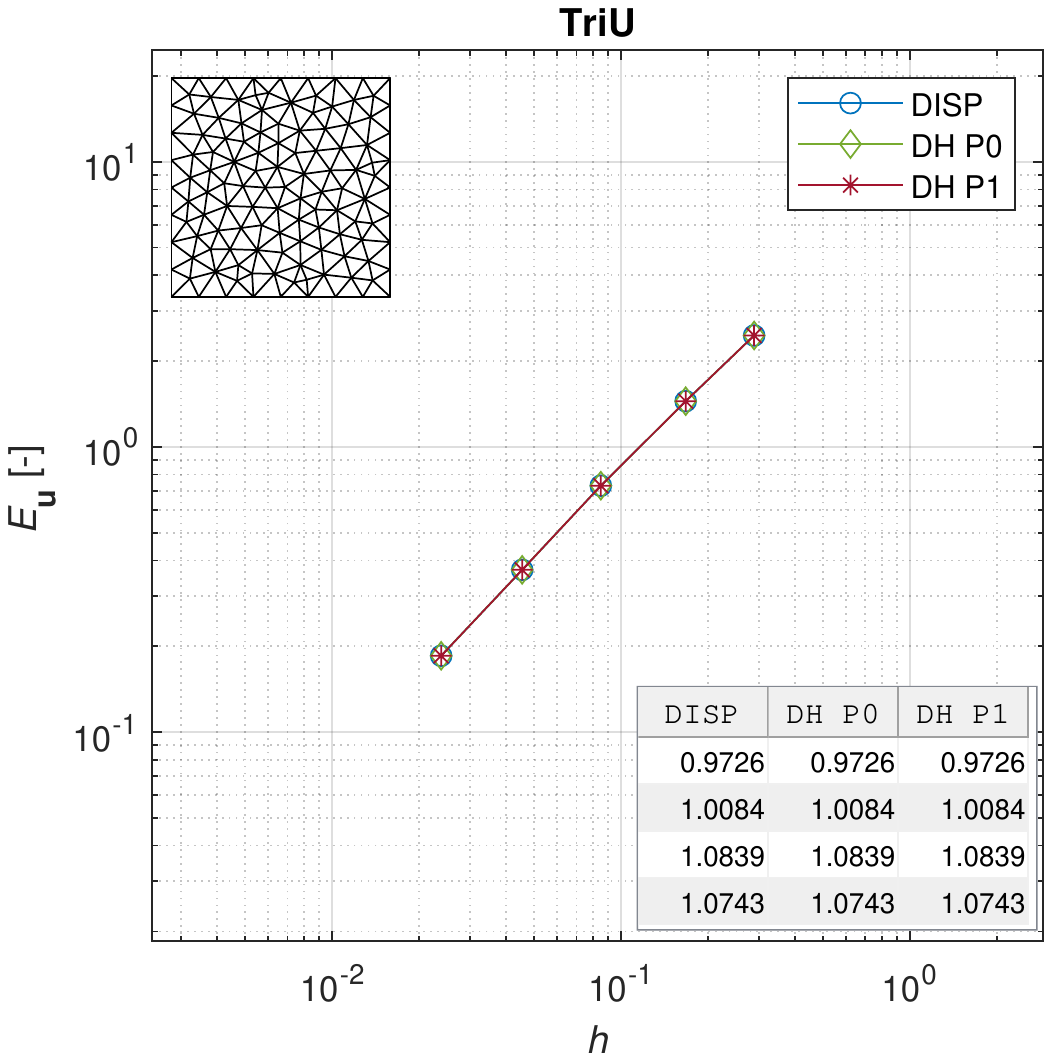}}
\subfigure
{\includegraphics[bb = 135 260 550 556,clip,scale=0.55]
                {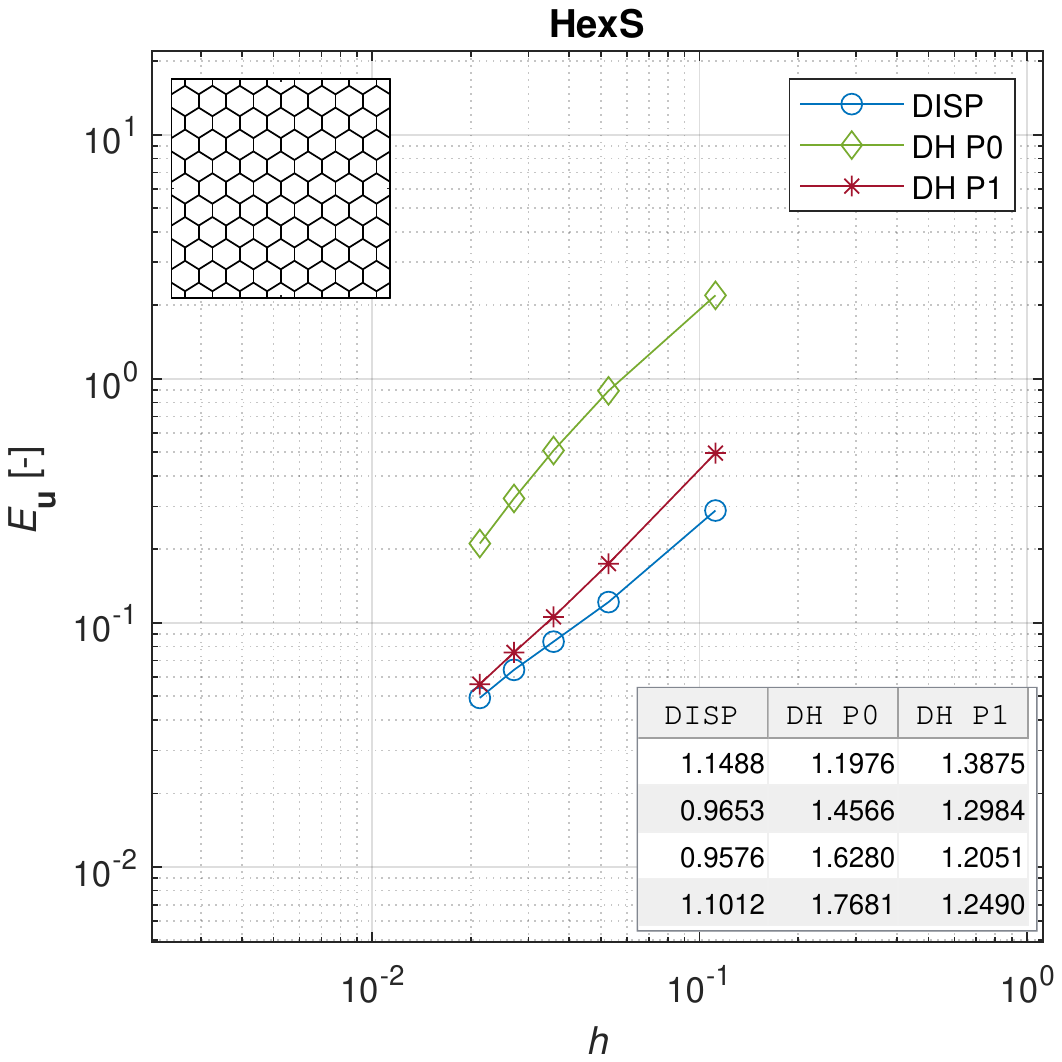}}
\subfigure
{\includegraphics[bb = 150 260 600 556,clip,scale=0.55]
                {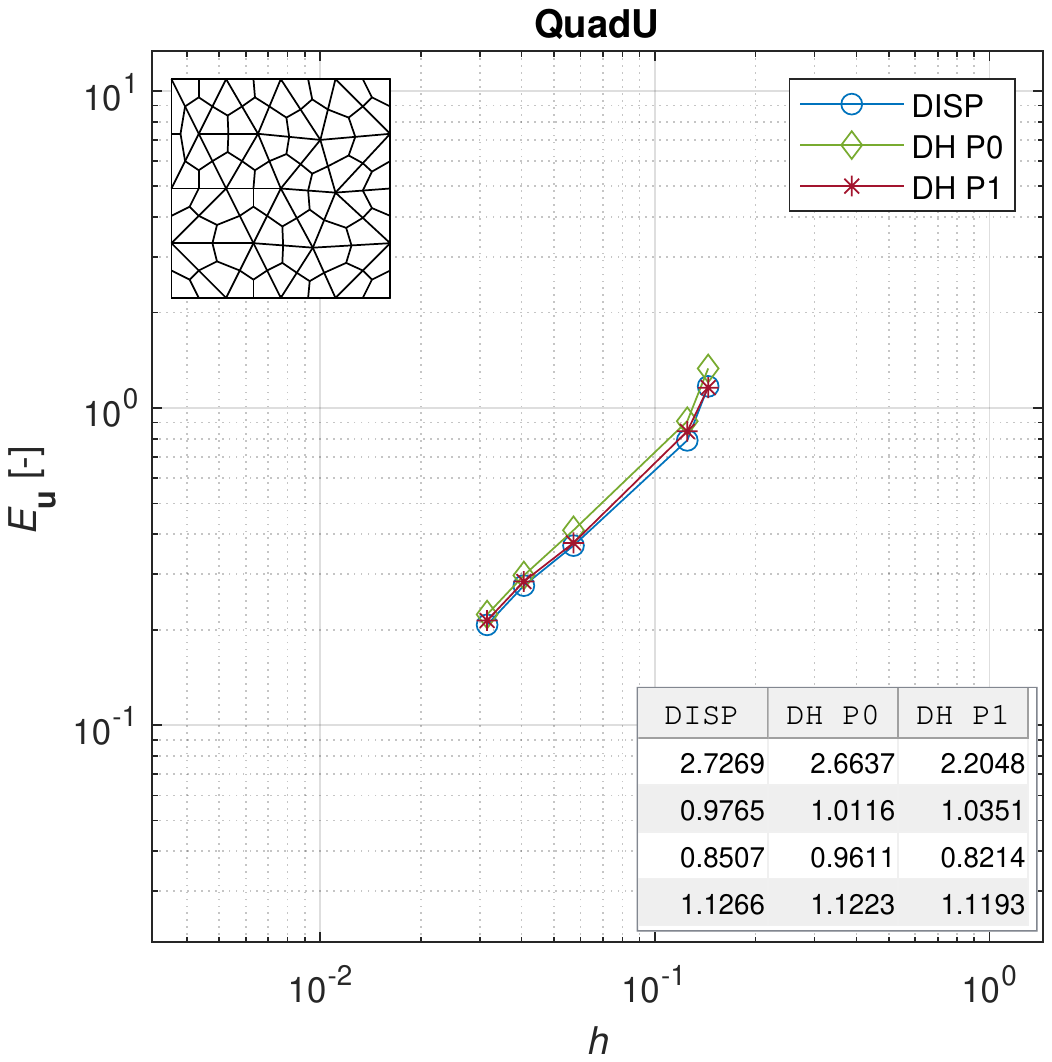}}
\subfigure
{\includegraphics[bb = 135 260 550 556,clip,scale=0.55]
                {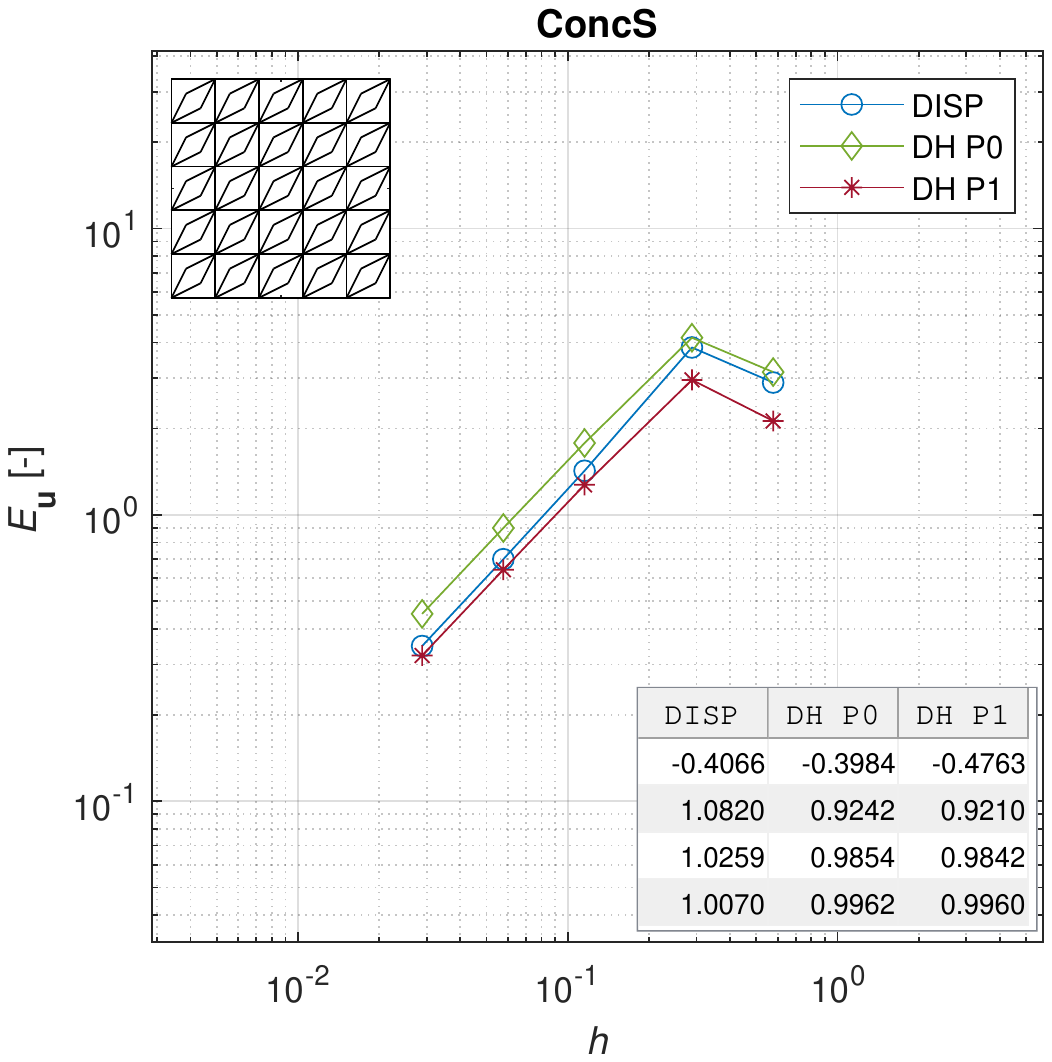}}
\subfigure
{\includegraphics[bb = 150 260 600 556,clip,scale=0.55]
                {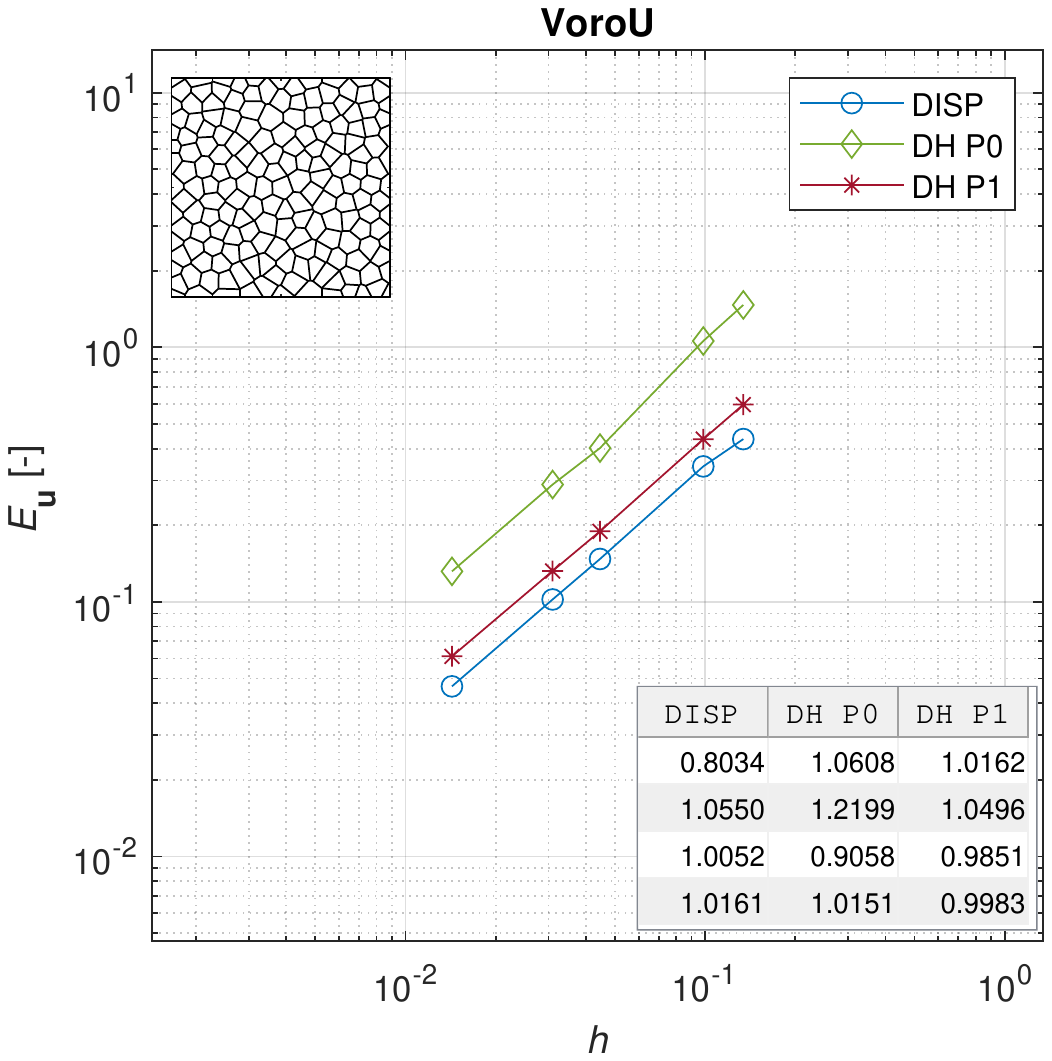}}
\caption{Test b - $E_{\bbu}$ {\it vs.} $h$ curves with branch slopes. Structured mesh - left. Unstructured mesh - right.}
\label{fig:Test_b_E_u}
\end{figure}

%
%
\clearpage
\newpage
\begin{figure}[h!]
\subfigure
{
\includegraphics[bb = 76 180 625 100,scale=0.39]
                {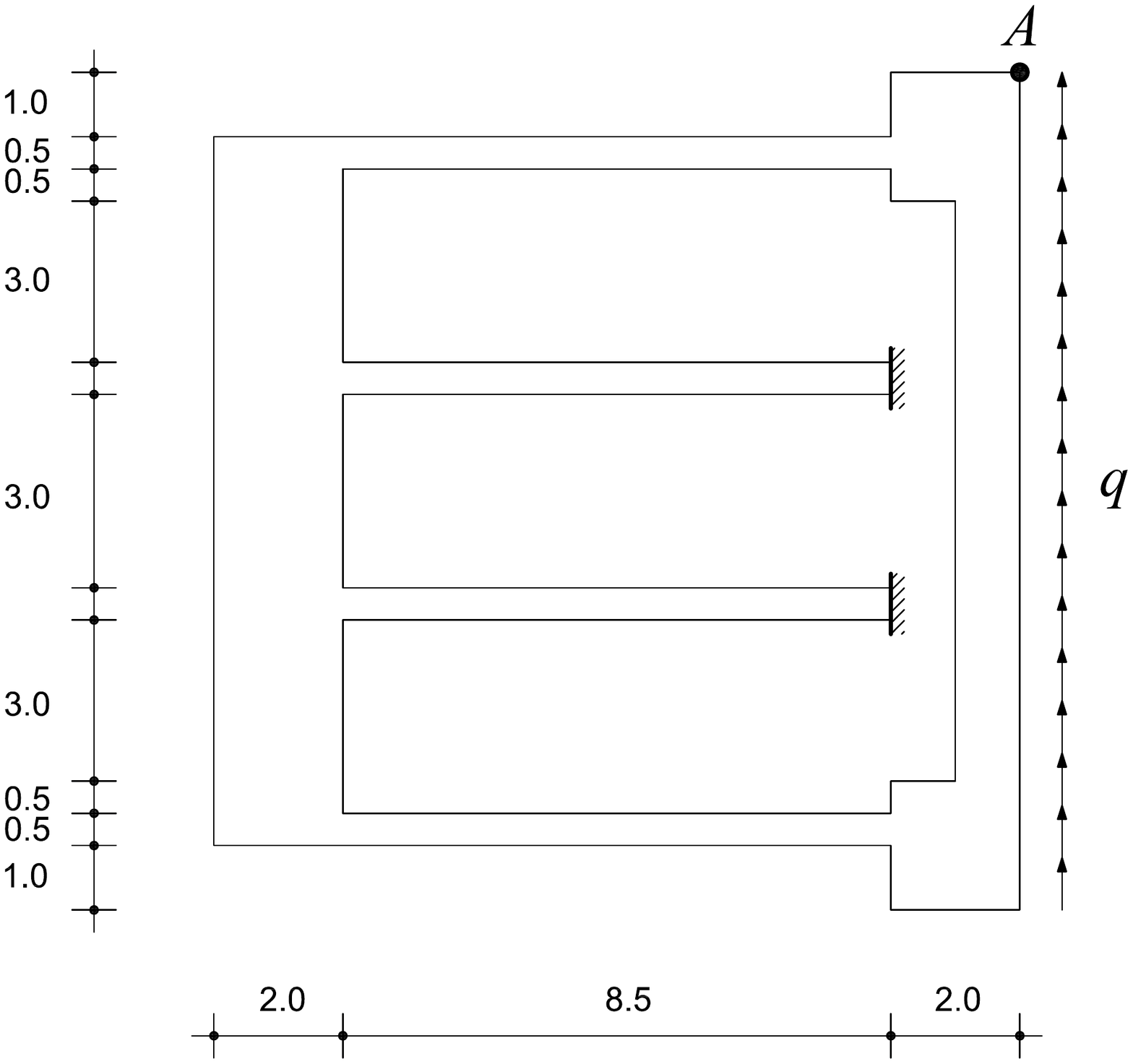}
                }
\subfigure
{
\includegraphics[bb = 120 260 600 556,clip,scale=0.450]
                {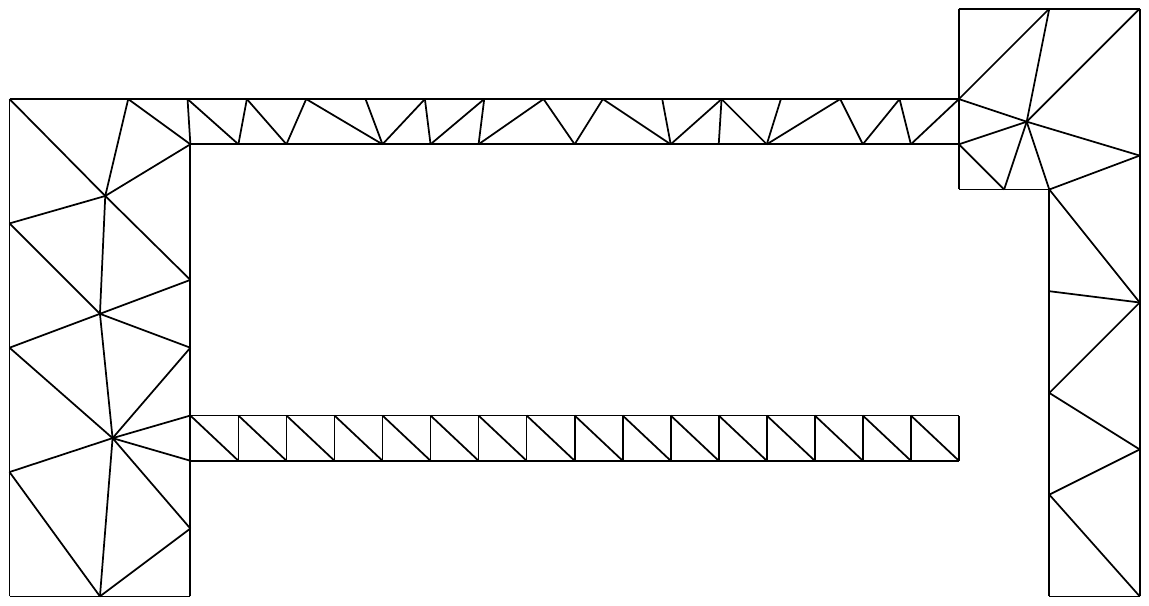}
                }
\subfigure
{
\includegraphics[bb = 129 260 600 556,clip,scale=0.450]
                {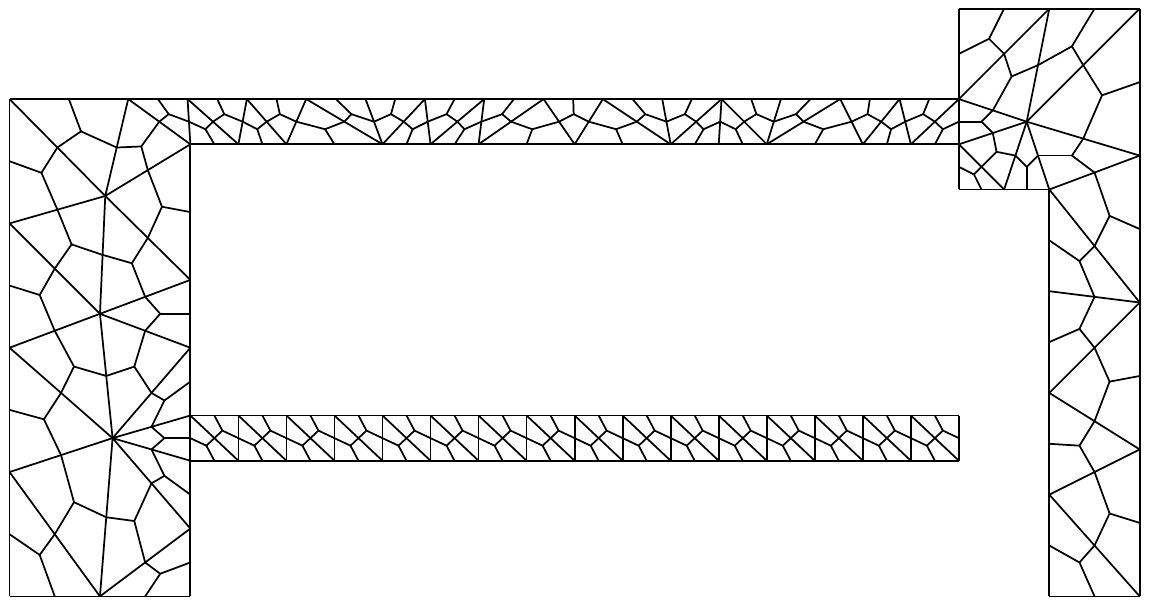}
                }
\subfigure
{
\includegraphics[bb = 125 260 600 556,clip,scale=0.450]
                {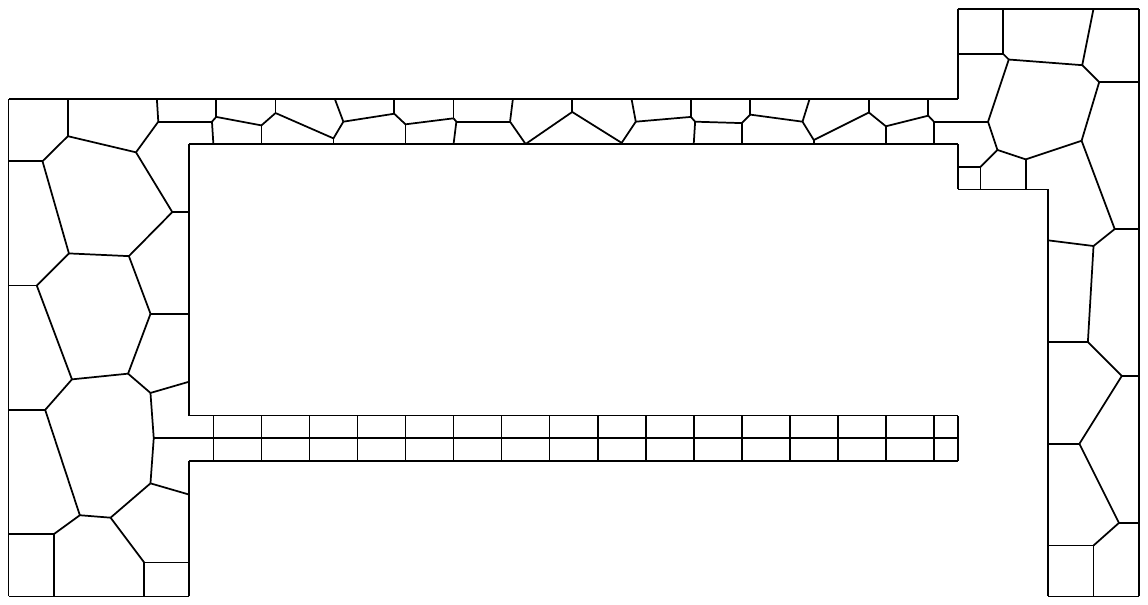}
                }
\caption{Folded-beam suspension. Geometry (quotes in $\mu$m), boundary conditions, loading. Mesh types and labels for relevant half-domain depicted for coarsest discretization adopted.}
\label{fig:comb_geom}
\end{figure}

\clearpage
\newpage
\begin{figure}[h]
\centering
\includegraphics[bb = 125 260 500 560,clip,scale=0.60]
                {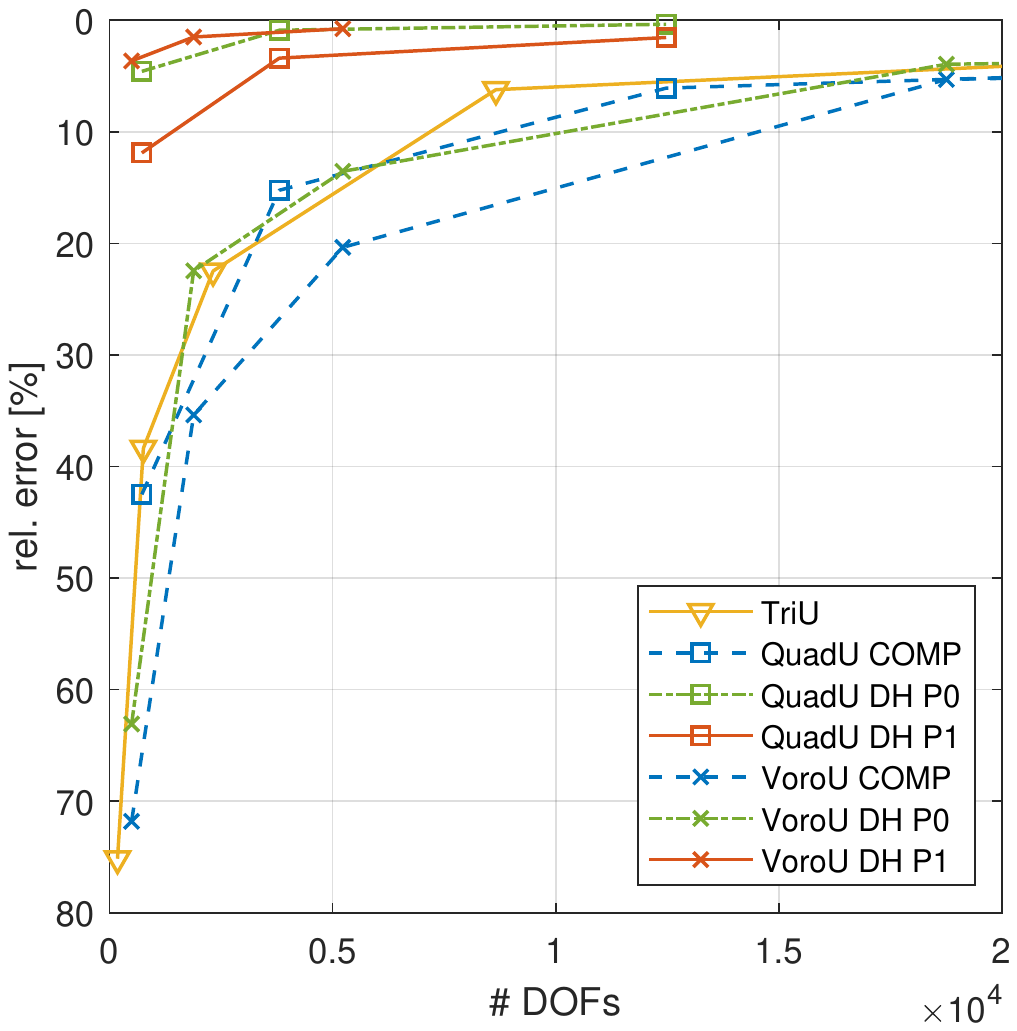}
\caption{Folded-beam suspension. Relative error plot for vertical displacement of target node A.}
\label{fig:comb_rslt}
\end{figure}
\clearpage
\newpage
\bibliographystyle{plain}
\bibliography{general-bibliography,biblio,VEM}
 \end{document}